\documentclass{amsart}
\usepackage{amssymb}
\usepackage{amscd}
\usepackage{verbatim}
\usepackage{epsfig}
\def \novt {\frac{n}{2}}
\def \nsq {\frac{n^2}{4}}
\def \intx {\stackrel{\circ}{X}}
\newcommand\Mand{\text{ and }}
\newcommand\Mat{\text{ at }}

\newcommand\Mwhere{\ \text{where}\ }

\newcommand\Mor{\ \text{or}\ }

\newcommand\Mif{\ \text{if}\ }

\newcommand\Mnear{\ \text{near}\ }

\def \del {\delta}   
\def \sigh {\frac{\sigma}{h}}

\newcommand\Diag{\operatorname{Diag}}
\newcommand\Id{\operatorname{Id}}

\newcommand\RR{\mathbb{R}}
\newcommand\bbR{\mathbb{R}}
\newcommand\bbB{\mathbb{B}}
\newcommand\bbS{\mathbb{S}}
\newcommand\mr{\mathbb{R}}

\newcommand\Cx{\mathbb{C}}
\newcommand\mc{\mathbb{C}}

\newcommand\NN{\mathbb{N}}
\newcommand\CC{\mathbb{C}}
\newcommand\mn{\mathbb{N}}
\newcommand\ms{\mathbb{S}}
\newcommand\Nat{\mathbb{N}}

\newcommand\BB{\mathbb{B}}
\newcommand\Bn{ {\mathbb{B} }^{n+1}}

\newcommand\mb{\mathbb{B}}

\newcommand\im{\operatorname{Im}}
\newcommand\re{\operatorname{Re}}

\newcommand\pa{\partial}
\newcommand\p{\partial}

\newcommand\ff{\operatorname{ff}}

\newcommand\mcv{{\mathcal V}}

\newcommand\mcs{{\mathcal S}}
\newcommand\mck{{\mathcal K}}

\newcommand\cL{{\mathcal L}}

\newcommand\ha{{\frac{1}{2}}}
\newcommand\la{{\lambda}}
\newcommand\eps{{\epsilon}}

\newcommand\oq{{\frac{1}{4}}}

\newcommand\cF{{\mathcal F}}
\newcommand\mcf{{\mathcal F}}

\newcommand\cS{{\mathcal S}}
\def \mcs {\mathcal {S}}
\def \mca {\mathcal {A}}
\newcommand\mcr{{\mathcal R}}
\newcommand\mcl{{\mathcal L}}
\newcommand\CI{{\mathcal C}^{\infty}}

\newcommand\dist{{\mathcal C}^{-\infty}}

\newcommand\dCI{{\dot{\mathcal C}}^{\infty}}

\renewcommand{\Box}{{\square}}

\newcommand\Diffb{\operatorname{Diff}_{\text{b}}}

\newcommand\Fr{{\mathcal F}}

\newcommand\zi{{}^{0}}

\newcommand\zT{\zi T}

\newcommand\bl{{\text b}}

\newcommand\Tb{{}^\bl T}
\newcommand\Tbff{{}^{\bl,\ff} T}
\newcommand\Tph{{}^{\phi}T}
\newcommand\To{{}^0 T^*}
\newcommand\oT{{}^0T}
\newcommand\Tos[1]{{}^0 T^*_{#1}}
\newcommand\bo{{}^0}
\newcommand\boH{{}^0\mathsf{H}}

\newcommand\sH{\mathsf{H}}

\newcommand\Vb{{\mathcal V}_{\bl}}

\newcommand\semi{\hbar}

\newcommand\ep{\epsilon}

\newtheorem{lemma}{Lemma}[section]
\newtheorem{prop}[lemma]{Proposition}
\newtheorem{thm}[lemma]{Theorem}

\newtheorem*{thm*}{Theorem}
\newtheorem*{prop*}{Proposition}
\newtheorem*{cor*}{Corollary}
\newtheorem*{conj*}{Conjecture}
\numberwithin{equation}{section}
\theoremstyle{remark}

\newtheorem*{rem*}{Remark}
\theoremstyle{definition}

\newtheorem*{Def*}{Definition}

\newcommand\Bo{(\mb^3)^\circ}
\newcommand\Bno{(\mb^n)^\circ}
\newcommand\Bc{\mb^3}
\newcommand\Bnc{\mb^n}
\renewcommand\dist{\mathrm{dist}}

\newcommand\CIo{{\mathcal{C}}^{\infty}_0}

\newcommand\CO{{\mathcal{C}}^{0}}

\newcommand\CmI{{\mathcal{C}}^{-\infty}}
\newcommand\sI{\text{sI}}
\newcommand\bH{\text{bH}}

\newcommand\paperintro%
        {%
         }
\newcommand\paperbody%
        {%
         }
\begin{document}

\title[Semiclassical resolvent estimates]
{Analytic continuation and semiclassical resolvent estimates on
  asymptotically hyperbolic spaces}

\author[Richard Melrose, Ant\^onio S\'a Barreto and
Andr\'as Vasy]{Richard Melrose, Ant\^onio S\'a Barreto and Andr\'as Vasy}
\date{March 17, 2011}
\thanks{The authors were partially supported by the National Science Foundation under
grants DMS-1005944 (RM), DMS-0901334 (ASB) and DMS-0801226 (AV), a
Chambers Fellowship from Stanford University (AV) and are grateful for the
stimulating environment at the MSRI in Berkeley
where some of this work was written in Autumn 2008.}

\begin{abstract}
In this paper we construct a parametrix for the high-energy
asymptotics of the analytic continuation of the resolvent on a
Riemannian manifold which is a small perturbation of the Poincar\'e
metric on hyperbolic space. As a result, we obtain non-trapping high energy
estimates for this analytic continuation.
\end{abstract}
\maketitle

\paperintro
\section*{Introduction}

Under appropriate conditions, the resolvent of the Laplacian on an
asymptotically hyperbolic space continues analytically through the spectrum
\cite{Mazzeo-Melrose:Meromorphic}.  In this paper we obtain estimates on
the analytic continuation of the resolvent for the Laplacian of a metric
that is a small perturbation of the Poincar\'e metric on hyperbolic
space. In particular we show for these perturbations of the metric, and
allowing in addition a real-valued potential, that there are only a finite
number of poles for the analytic continuation of the resolvent to any
half-plane containing the physical region and that the resolvent satisfies
polynomial bounds on appropriate weighted Sobolev spaces near infinity in
such a strip. This result, for a small strip, is then applied to the
twisted Laplacian which is the stationary part of the d'Alembertian on de
Sitter-Schwarzschild. In a companion paper
\cite{Melrose-SaBarreto-Vasy:Asymptotics} the decay of solutions to the
wave equation on de Sitter-Schwarzschild space is analyzed using these
estimates.

In the main part of this paper constructive semiclassical methods are used
to analyze the resolvent of the Laplacian, and potential perturbations of
it, for a  complete, asymptotically hyperbolic, metric on the interior of the
ball, $\Bn,$
\begin{equation}
\begin{gathered}
g_{\delta}=g_0+\chi_\delta(z)H,\\
g_0=\frac{ 4 dz^2}{(1-|z|^2)^2},\
\chi_\delta(z)=\chi\left(\frac{(1-|z|)}{\delta}\right).
\end{gathered}
\label{SeClRe.14}\end{equation}
Here $g_{0}$ is the standard hyperbolic metric, $H=H(z,dz)$ is a symmetric
2-tensor which is smooth up to the boundary of the ball and $\chi\in
\CI(\bbR),$ has $\chi(s)=1$ if $|s|<\ha,$ $\chi(s)=0$ if $|s|>1.$ The
perturbation here is always the same at the boundary but is cut off closer
and closer to it as $\delta \downarrow0.$ For $\delta >0$ small enough we
show that the analytic continuation of the resolvent of the Laplacian is
smooth, so has no poles, in the intersection of the exterior of a
sufficiently large ball with any strip around the real axis in the
non-physical half-space as an operator between weighted $L^2$ or Sobolev
spaces, and obtain high-energy estimates for this resolvent in this strip.

A special case of these estimates is as follows: Let
$x=\frac{1-|z|}{1+|z|},$ $W\in\CI(\Bn)$ be real-valued and let $R_\delta
(\sigma)=(\Delta_{g_\delta}+x^2W-\sigma^2-n/2^2)^{-1}$ denote the resolvent
of $\Delta_{g_\delta}+x^2W.$ The spectral theorem shows that
$R_\delta(\sigma)$ is well defined as a bounded operator in $L^2(\Bn;dg)$
if $\im \sigma <<0,$ and the results of Mazzeo and the first author show
that it continues meromorphically to the upper half plane.  Here we show
that there exists a strip about the real axis such that if $\delta$ is
small and $a$ and $b$ are suitably chosen $x^a R_\delta(\sigma) x^b $ has
no poles, provided $|\sigma|$ is large, and moreover we obtain a polynomial
bound for the norm of $x^a R_\delta(\sigma) x^b.$ More precisely, with
$H^k_0(\Bn)$ the $L^2$-based Sobolev space of order $k$, so for $k=0$,
$H^0_0(\Bn)=L^2(\Bn;dg):$

\begin{thm}(See Theorem~\ref{resolvent-bounds} for the full statement.)
\label{resolvent-bounds-simple} There exist $\delta_0>0,$  such that if
$0\leq \delta\leq \delta_0,$ then 
$x^a R_{\delta}(\sigma) x^b$ continues holomorphically to the region $\im
\sigma < M,$  $M>0$  $| \sigma| >K(\delta,M),$ provided $ \im\sigma< b,$
and   $a>\im \sigma.$ Moreover, there exists $C>0$ such that
\begin{gather}
\begin{gathered}
|| x^a R_\delta (\sigma) x^b v||_{H^k_0(\Bn)} \leq
C |\sigma|^{-1+\frac{n}{2} +k}  ||v||_{L^2(\Bn)}, \;\ k=0,1,2, \\
|| x^a R_\delta (\sigma) x^b v||_{L^2(\Bn)} \leq
C |\sigma|^{-1+\frac{n}{2} +k}  ||v||_{H_0^{-k}(\Bn)}, \;\ k=0,1,2, \\
\end{gathered} \label{sobolev1-simple}
\end{gather}
\end{thm}

This estimate is not optimal; the optimal bound is expected to be $O(
|\sigma|^{-1+k}).$ The additional factor of $|\sigma|^{\frac{n}{2}}$ results
from ignoring the oscillatory behavior of the Schwartz kernel of the
resolvent; a stationary phase type argument (which is made delicate by the
intersecting Lagrangians described later) should give an improved
result. However, for our application, which we now describe, the polynomial
loss suffices, as there are similar losses from the trapped geodesics in
compact sets.

As noted above, these results and the underlying estimates can be applied 
to study the wave equation on $1+(n+1)$ dimensional de Sitter-Schwarzschild
space. This model is given by 
\begin{equation}
M=\RR_t\times \intx,  \;\ X=[r_\bH,r_{\sI}]_r\times\ms^{n}_\omega,
\label{model-ndim}
\end{equation}
with the Lorentzian metric
\begin{equation}\label{eq:dS-Sch-metric}
G=\alpha^2\,dt^2-\alpha^{-2}\,dr^2-r^2\,d\omega^2,
\end{equation}
where
$d\omega^2$ is the standard metric on $\ms^{n},$
\begin{equation}\label{eq:alpha-def}
  \alpha=\left(1-\frac{2m}{r}-\frac{\Lambda r^2}{3}\right)^{1/2},
\end{equation}
with $\Lambda$ and $m$ positive constants satisfying
$0<9m^2\Lambda<1,$ and $r_\bH,r_{\sI}$ are the two positive roots of
$\alpha=0.$ 

The d'Alembertian associated to the metric $G$ is
\begin{equation}
\square= \alpha^{-2}(D_t^2-\alpha^2r^{-n}D_r(r^{n}\alpha^2
D_r)-\alpha^2r^{-2} \Delta_\omega).
\label{dal}\end{equation}
An important goal is then to give a precise description of the asymptotic
behavior, in all regions, of space-time, of the solution of the wave
equation, $\Box u=0$, with initial data which are is necessarily compactly
supported. The results given below can be used to attain this goal; see
the companion paper \cite{Melrose-SaBarreto-Vasy:Asymptotics}.

Since we are only interested here in the null space of the d'Alembertian,
the leading factor of $\alpha ^{-2}$ can be dropped. The results above can
be applied to the corresponding stationary
operator, which is a twisted Laplacian
\begin{equation}
\Delta_X=\alpha^2r^{-n}D_r(\alpha^2 r^{n}D_r)
+\alpha^2r^{-2}\Delta_\omega.
\label{modelX}
\end{equation}
In what follows we will sometimes consider $\alpha$ as a boundary defining
function of $X.$ This amounts to a change in the $\CI$ structure of
$X;$ we denote the new manifold by $X_{\ha}.$

The second order elliptic operator $\Delta_X$ in \eqref{modelX} is
self-adjoint, and non-negative, with respect to the measure  
\begin{equation}
\Omega=\alpha^{-2} r^{n} \,d r \,d \omega,\label{measure}
\end{equation}
with $\alpha$ given by \eqref{eq:alpha-def}. So, by the spectral theorem,
the resolvent
\begin{gather}
R(\sigma)=(\Delta_X - \sigma^2 )^{-1} : L^2 ( X ;  \Omega ) 
\longrightarrow  L^2 ( X ;  \Omega ) \label{resolv1}
\end{gather}
is holomorphic for $\im \sigma < 0 .$ In
\cite{Sa-Barreto-Zworski:Distribution} the second author and Zworski, using
methods of Mazzeo and the first author from
\cite{Mazzeo-Melrose:Meromorphic}, Sj\"ostrand and Zworski 
\cite{Sjostrand-Zworski:ComplexScaling} and Zworski \cite{Zworski:LimitSet}
prove that the resolvent family has an analytic continuation. 

\begin{thm}(S\'a Barreto-Zworski, see
  \cite{Sa-Barreto-Zworski:Distribution})\label{mercont} As operators
\begin{gather*}
R ( \sigma ): \CIo( \stackrel{o}{X} )\longrightarrow\CI( \stackrel{o}{X} )
\end{gather*}
the family \eqref{resolv1} has a meromorphic continuation to $\mc$ with
isolated poles of finite rank.  Moreover, there exists $\eps>0$ such that
the only pole of $R(\sigma)$ with $\im \sigma<\eps$ is at $\sigma=0;$ it
has multiplicity one. 
\end{thm}
\noindent
Theorem \ref{mercont} was proved for $n+1=3,$ but its proof easily extends
to higher dimensions.

In order to describe the asymptotics of wave propagation precisely on $M$
via $R(\sigma)$ it is necessary to understand the action of $R(\sigma)$ on
weighted Sobolev spaces for $\sigma$ in a strip about the real axis as
$|\re\sigma|\to\infty$. The results of \cite{Mazzeo-Melrose:Meromorphic}
actually show that $R(\sigma)$ is bounded as a map between the weighted
spaces in question; the issue is uniform control of the norm at high energies.  The strategy is to obtain bounds for $R(\sigma)$ for $\re(\sigma)$ large 
in the interior of $X,$  then near the ends $r_{\bH}$ and $r_{\sI},$  and later glue those estimates. In the case $n+1=3,$ we can use the following result of Bony and H\"afner to obtain  bounds for the resolvent in the interior
\begin{thm}(Bony-H\"afner, see
\cite{Bony-Haefner:Decay})\label{bhthm} There exists $\eps>0$
and $M\geq 0$ such that if $|\im\sigma|<\eps$   and $|\re \sigma|>1,$
then for any  $\chi\in C_0^\infty(\stackrel{o}{X})$ there exists $C>0$ such that  if $n=3$ in 
\eqref{model-ndim}, then
\begin{equation}
||\chi R(\sigma) \chi f||_{L^2(X ;\Omega) } \leq C |\sigma|^M||f||_{L^2(X;\Omega)}.
\label{bonyhafnerest}\end{equation}
\end{thm}

This result is not known in higher dimensions (though the methods of
Bony and H\"afner would work even then),
and to prove our main theorem we use the results of Datchev and the third author \cite{Datchev-Vasy}  and Wunsch and Zworski \cite{Wunsch-Zworski} to handle the general case.  The advantage of the method of \cite{Datchev-Vasy} is that one does not need to obtain a bound for the exact resolvent in the interior and we may work with the approximate model of \cite{Wunsch-Zworski} instead.  We decompose the manifold $X$ in two parts
\begin{gather}
\begin{gathered}
X= X_0 \cup X_1, \text{ where  } \\
X_0=[r_{\bH}, r_\bH+4\del) \times \ms^n \cup (r_{\sI}-4\del, r_\sI] \times \ms^n   \text{ and } \\
X_1=( r_\bH+\del, r_{\sI}-\del) \times \ms^n.
\end{gathered}\label{Xdecomp}
\end{gather}
If $\del$ is small enough and if $\gamma(t)$ is an integral curve of the Hamiltonian of $\Delta_X$  then
(see  Section \ref{sec:black-hole} and either \cite{Datchev-Vasy} or \cite{Joshi-SaBarreto})
\begin{gather}
\text{ if } x(\gamma(t))<4\delta \text{ and } \frac{d x(t)}{dt}=0 \Rightarrow \frac{d^2 x(t)}{dt^2}<0.\label{bich-convexity1}
\end{gather}

We consider the operator $\Delta_X$ restricted to $X_1,$ and place it
into the setting of \cite{Wunsch-Zworski} as follows.
Let $X_1'$ be another Riemannian manifold extending $\tilde X_1=( r_\bH+\del/2,
r_{\sI}-\del/2, r_\sI) \times \ms^n$ (and thus $X_1$)
and which is Euclidean outside some compact set, and let
$\Delta_{X'_1}$ be a self-adjoint
operator extending $\Delta_X$ with principal symbol given by the
metric on $X'_1$ which is equal to the
Euclidean Laplacian on the ends. Let
\begin{gather}
P_1=h^2\Delta_{X_1'}-i\Upsilon, \;\ h \in (0,1) \label{operatorp1}
\end{gather}
where $\Upsilon\in \CI(X_1';[0,1])$ is such that $\Upsilon=0$ on $X_1$ and $\Upsilon=1$
outside $\tilde X_1$. Thus, $P_1-1$ is semiclassically elliptic on a
neighborhood of $X_1'\setminus\tilde X_1$. (In particular, this
implies that no bicharacteristic of $P_1-1$ leaves $X_1$ and returns
later, i.e.\ $X_1$ is bicharacteristically convex in $X_1'$,
since this is true for $X_1$ inside $\tilde X_1$, and $P_1-1$ is
elliptic outside $\tilde X_1$, hence has no characteristic set there.)
By Theorem 1 of \cite{Wunsch-Zworski} there exist positive constants $c,$ $C$ and $\eps$ independent of $h$ such that
\begin{gather}
||(P_1-\sigma^2)^{-1}||_{L^2\rightarrow L^2} \leq  Ch^{-N}, \;\ \sigma
\in (1-c,1+c) \times (-c,\eps h)\subset\Cx. \label{wuzwestimate}
\end{gather}

Due to the fact that $0<9m^2\Lambda<1,$ the function $\beta(r)=\ha
\frac{d}{dr}\alpha^2(r)$ satisfies $\beta(r_{\bH})>0$ and
$\beta(r_{\sI})<0.$ Set $\beta_{\bH}=\beta(r_{\bH})$ and
$\beta_{\sI}=\beta(r_{\sI}).$ The weight function we consider,
$\tilde\alpha\in\CO([r_{\bH},r_{\sI}]),$ is positive in the interior and
satisfies

\begin{equation}
\tilde{\alpha}(r)= \begin{cases}
\alpha ^{1/\beta _{\bH}}&\Mnear r_{\bH}\\
\alpha ^{1/|\beta _{\sI}|}&\Mnear r_{\sI}
\end{cases}
\label{SeClRe.1}\end{equation}
We will prove

\begin{thm}\label{globalest}
If $n=2,$ let $\eps>0$ be such that  \eqref{bonyhafnerest} holds.     
     In general, assume that $\delta$ is such that \eqref{bich-convexity1} is satisfied, and let $\eps>0$ be such that \eqref{wuzwestimate}  holds.  If $$0<\gamma<\min(\eps,\beta_{\bH},|\beta_{\sI}|,1),
$$
then for $b>\gamma$ there exist $C$ and $M$ such that if
$\im\sigma\le\gamma$ and $|\re\sigma|\ge1,$ 
\begin{equation}
||{\tilde{\alpha}}^b R(\sigma){\tilde{\alpha}}^bf||
_{L^2 ( X ;\Omega)}\leq C|\sigma|^M ||f||_{L^2 ( X ;\Omega) },
\label{mainest1}
\end{equation}
where $\Omega$ is defined in \eqref{measure}.
\end{thm}
\noindent

 This result can be refined by allowing the power of the weight, on either
side, to approach $\im\sigma$ at the expense of an additional logarithmic term,
see Theorem~\ref{globalest-strong}.

Two proofs of Theorem \ref{globalest} are given below. The first, which is
somewhat simpler but valid only for $n+1=3,$ is given in sections \ref{sec:decomposition} and \ref{sec:black-hole-proof}.  It uses techniques of Bruneau
and Petkov \cite{Bruneau-Petkov:Semiclassical} to glue the resolvent
estimates from Theorem~\ref{resolvent-bounds-simple} and the 
localized estimate  \eqref{bonyhafnerest}.
The second proof, valid in general dimension, uses the estimate \eqref{wuzwestimate} and the semiclassical
resolvent gluing techniques of Datchev and the third author
\cite{Datchev-Vasy}.  This is carried out in section
\ref{sec:black-hole-n+1-proof}.

Related weighted $L^2$ estimates for the resolvent on asymptotically
hyperbolic and asymptotically Euclidean spaces have been proved by Cardoso
and Vodev in \cite{Cardoso-Vodev:Uniform}.  However, such estimates,
combined with Theorem \ref{bhthm}, only give the holomorphic continuation
of the resolvent, as an operator acting on weighted $L^2$ spaces, for $|\re
\sigma|>1,$ to a region below a curve which converges polynomially to the
real axis. These weaker estimates do suffice to establish the asymptotic
behavior of solutions of the wave equation modulo rapidly decaying terms
(rather than exponentially decaying) and would give a different proof of
the result of Dafermos and Rodnianski \cite{Dafermos-Rodnianski:de-Sitter}.

In the case of a non-trapping asymptotically hyperbolic manifold which has
constant sectional curvature near the boundary, it was shown by Guillarmou
in \cite{Guillarmou:Absence} that there exists a strip about the real axis,
excluding a neighborhood of the origin, which is free of resonances.  In
the case studied here, the sectional curvature of the metric associated to
$\Delta_X$ is not constant near the boundary, and there exist trapped
trajectories.  However, see \cite{Sa-Barreto-Zworski:Distribution}, all
trapped trajectories of the Hamilton flow of $\Delta_X$ are hyperbolic, and
the projection of the trapped set onto the base space is contained in the
sphere $r=3m,$ which is known as the ergo-sphere. Since the effects of the
trapped trajectories are included in the estimates of Bony and H\"afner,
constructing the analytic continuation of the resolvent of a twisted
Laplacian that has the correct asymptotic behavior at infinity, uniformly
at high energies, allows one to obtain the desired estimates on weighted
Sobolev spaces via pasting techniques introduced by Bruneau and Petkov
\cite{Bruneau-Petkov:Semiclassical}. The main technical result is thus
Theorem~\ref{resolvent-bounds-simple}, and its strengthening,
Theorem~\ref{resolvent-bounds}.

Theorem~\ref{resolvent-bounds} is proved by the construction of
a high-energy parametrix for $\Delta_{g_\delta}+x^2W$. More precisely,
as customary, the problem is translated to the construction of a semiclassical
parametrix for
\begin{equation*}
P(h,\sigma)= h^2\Delta_g+h^2x^2 W-h^2\frac{n^2}{4}-\sigma^2
=h^2\left(\Delta_g+x^2 W-\frac{n^2}{4}-\left(\frac{\sigma}{h}\right)^2\right).
\end{equation*}
where now $\sigma\in (1-c,1+c)\times (-Ch,Ch)\subset\Cx$, $c,C>0$,
and $h\in (0,1)$, $h\to 0$, so the actual spectral parameter is
$$
\frac{n^2}{4}+\left(\frac{\sigma}{h}\right)^2,
$$
and $\im \frac{\sigma}{h}$ is bounded.
Note that for $\im\sigma<0$,
\begin{equation}\label{scresolvent}
R(h,\sigma)=P(h,\sigma)^{-1}:L^2(\Bn)\to H^2_0(\Bn)
\end{equation}
is meromorphic by the results of Mazzeo and the first author
\cite{Mazzeo-Melrose:Meromorphic}.
Moreover, while $\sigma$ is
not real, its imaginary part is $O(h)$ in the semiclassical sense, and is thus
not part of the semiclassical principal symbol of the operator.

The construction proceeds on the semiclassical resolution $M_{0,h}$ of the
product of the double space $\Bn\times_0\Bn$ introduced in
\cite{Mazzeo-Melrose:Meromorphic}, and the interval $[0,1)_h$ -- this space
  is described in detail in
  Section~\ref{semiclassical-double-space}. Recall that, for fixed $h>0$,
  the results of \cite{Mazzeo-Melrose:Meromorphic} show that the Schwartz
  kernel of $P(h,\sigma)^{-1}$, defined for $\im\sigma<0$, is well-behaved
  (polyhomogeneous conormal) on $\Bn\times_0\Bn$, and it extends
  meromorphically across $\im\sigma=0$ with a similarly polyhomogeneous
  conormal Schwartz kernel. Thus, the space we are considering is a very
  natural one. The semiclassical resolution is already needed away from all
  boundaries; it consists of blowing up the diagonal at $h=0$. Note that
  $P(h,\sigma)$ is a semiclassical differential operator which is elliptic
  in the usual sense, but its semiclassical principal symbol
  $g-(\re\sigma)^2$ is not elliptic (here $g$ is the dual metric
  function). Ignoring the boundaries for a moment, when a semiclassical
  differential operator is elliptic in both senses, it has a parametrix in
  the small semiclassical calculus, i.e.\ one which vanishes to infinite
  order at $h=0$ off the semiclassical front face (i.e.\ away from the
  diagonal in $\Bn\times_0\Bn\times\{0\}$).  However, as $P(h,\sigma)$ is
  not elliptic semiclassically, semiclassical singularities (lack of decay
  as $h\to 0$) flow out of the semiclassical front face.

It is useful to consider the flow in terms of Lagrangian geometry.  Thus,
the small calculus of order $-\infty$ semiclassical pseudodifferential
operators consists of operators whose Schwartz kernels are
semiclassical-conormal to the diagonal at $h=0$. As $P(h,\sigma)$ is not
semiclassically elliptic (but is elliptic in the usual sense, so it behaves
as a semiclassical pseudodifferential operator of order $-\infty$ for our
purposes), in order to construct a parametrix for $P(h,\sigma)$, we need to
follow the flow out of semiclassical singularities from the conormal bundle
of the diagonal. For $P(h,\sigma)$ as above, the resulting Lagrangian
manifold is induced by the geodesic flow, and is in particular, up to a
constant factor, the graph of the differential of the distance function on
the product space. Thus, it is necessary to analyze the geodesic flow and
the distance function; here the presence of boundaries is the main issue.
As we show in Section~\ref{distance-function}, the geodesic flow is
well-behaved on $\Bn\times_0\Bn$ as a Lagrangian manifold of the
appropriate cotangent bundle. Further, for $\delta>0$ small (this is where
the smallness of the metric perturbation enters), its projection to the
base is a diffeomorphism, which implies that the distance function is also
well-behaved.  This last step is based upon the precise description of the
geodesic flow and the distance function on hyperbolic space, see
Section~\ref{distance-function}.

In Section~\ref{full-sc-parametrix}, we then construct the parametrix by
first solving away the diagonal singularity; this is the usual elliptic
parametrix construction. Next, we solve away the small calculus error in
Taylor series at the semiclassical front face, and then propagate the
solution along the flow-out by solving transport equations. This is an
analogue of the intersecting Lagrangian construction of the first author
and Uhlmann \cite{Melrose-Uhlmann:Intersection}, see also the work of
Hassell and Wunsch \cite{Hassell-Wunsch:Semiclassical} in the semiclassical
setting.  So far in this discussion the boundaries of $M_{0,\semi}$ arising
from the boundaries of $\Bn\times_0\Bn$ have been ignored; these enter into
the steps so far only in that it is necessary to ensure that the
construction is uniform (in a strong, smooth, sense) up to these
boundaries, which the semiclassical front face as well as the lift of
$\{h=0\}$ meet transversally, and only in the zero front face, i.e.\ at the
front face of the blow-up of $\Bn\times\Bn$ that created $\Bn\times_0\Bn$.
Next we need to analyze the asymptotics of the solutions of the transport
equations at the left and right boundary faces of $\Bn\times_0\Bn$; this is
facilitated by our analysis of the flowout Lagrangian (up to these boundary
faces). At this point we obtain a parametrix whose error is smoothing and
is $O(h^\infty)$, but does not, as yet, have any decay at the zero front
face. The last step, which is completely analogous to the construction of
Mazzeo and the first author, removes this error.

As a warm-up to this analysis, in Section~\ref{3D-parametrix} we present a
three dimensional version of this construction, with worse, but still
sufficiently well-behaved error terms. This is made possible by a
coincidence, namely that in $\RR^3$ the Schwartz kernel of the resolvent of
the Laplacian at energy $(\lambda-i0)^2$ is a constant multiple of
$e^{-i\lambda r} r^{-1}$, and $r^{-1}$ is a homogeneous function on
$\RR^3$, which enables one to blow-down the semiclassical front face at
least to leading order. Thus, the first steps of the construction are
simplified, though the really interesting parts, concerning the asymptotic
behavior at the left and right boundaries along the Lagrangian, are
unchanged. We encourage the reader to read this section first as it is more
explicit and accessible than the treatment of arbitrary dimensions.

In Section~\ref{sec:L2-bounds}, we obtain weighted $L^2$-bounds for the
parametrix and its error. In Section~\ref{sec:resolvent-bounds} we used these
to prove
Theorem~\ref{resolvent-bounds-simple} and
Theorem~\ref{resolvent-bounds}.

In Section~\ref{sec:black-hole} we describe in detail the
de Sitter-Schwarzschild set-up. Then in
Section~\ref{sec:decomposition}, in dimension $3+1$,
we describe the approach of Bruneau and Petkov \cite{Bruneau-Petkov:Semiclassical} 
reducing the
necessary problem to the combination of analysis on the ends,
i.e.\ Theorem~\ref{resolvent-bounds}, and of the cutoff resolvent,
i.e.\ Theorem~\ref{bhthm}. Then, in Section~\ref{sec:black-hole-proof},
we use this method to prove Theorem~\ref{globalest}, and its strengthening,
Theorem~\ref{globalest-strong}. Finally, in
Section~\ref{sec:black-hole-n+1-proof} we give a different proof which
works in general dimension, and does not require knowledge of
estimates for the exact cutoff resolvent. Instead, it uses the results of Wunsch and Zworski
\cite{Wunsch-Zworski} for normally hyperbolic trapping in the presence
of Euclidean ends and of Datchev
and the third author \cite{Datchev-Vasy} which provide a method to combine these with our
estimates on hyperbolic ends. Since this method is described in detail
in \cite{Datchev-Vasy}, we keep this section fairly brief.

\paperbody

\section{Resolvent estimates for model operators}
In this section we state the full version of the main technical result,
Theorem~\ref{resolvent-bounds-simple}.

Let  $g_0$ be the metric on $\Bn$ given by
\begin{gather}
g_{0}=\frac{4 dz^2}{(1-|z|^2)^2}. \label{metg0}
\end{gather}
We consider a one-parameter family of perturbations of $g_{0}$ supported in
a neighborhood of $\p \Bn$  of the form
\begin{gather}
g_\delta= g_{0} + \chi_\delta(z) H(z,dz), \label{metgeps}
\end{gather}
where $H$ is a symmetric 2-tensor,  which is  $\CI$ up to $\p \Bn,$
$\chi\in \CI(\mr),$ with  $\chi(s)=1$ if $|s|<\ha,$  $\chi(s)=0$ if
$|s|>1,$ and $\chi_\delta(z)=\chi((1-|z|)\delta^{-1}).$

Let   $x=\frac{1-|z|}{1+|z|},$ $W\in\CI(\Bn)$ and let $R_\delta (\sigma)=(\Delta_{g_\delta}+x^2W-\sigma^2-1^2)^{-1}$ denote the resolvent of $\Delta_{g_\delta}+x^2W.$
The spectral theorem  gives that $R_\delta(\sigma)$ is well defined as a bounded operator in 
$L^2(\Bn)=L^2(\Bn;dg)$  if $\im \sigma <<0.$
The results of \cite{Mazzeo-Melrose:Meromorphic}
show that $R_{\delta}(\sigma)$ continues meromorphically to $\mc\setminus i\mn/2$ as an operator mapping $\CI$ functions vanishing to infinite order at $\p \Bn$ to distributions in
$\Bn.$
We recall, see for example \cite{Mazzeo:Edge}, that for $k\in \mn,$
 \begin{gather*}
  H^k_0(\Bn)=\{ u \in L^2(\Bn): (x\p_x, \p_\omega)^m u \in L^2(\Bn), \;\ m \leq k\},
  \end{gather*}
 and
  \begin{gather*}
  H_0^{-k}=\{ v\in \mathcal{D}'(\Bn):  \text{ there exists } u_{\beta} \in L^2(\Bn),  
  v= \sum_{|\beta|\leq k} (x\p_x, \p_\omega)^\beta u_\beta\}.
  \end{gather*}
Our main result is the following theorem:

\begin{thm}\label{resolvent-bounds}   There exist $\delta_0>0,$  such that if $0\leq \delta\leq \delta_0,$ then
$x^a R_{\delta}(\sigma) x^b$ continues holomorphically to $\im \sigma < M,$ $M>0,$ provided
$| \sigma| >K(\delta,M),$   $b> \im\sigma$ and   $a>\im \sigma.$ Moreover,
there exists $C>0$ such that
\begin{gather}
\begin{gathered}
|| x^a R_\delta (\sigma) x^b v||_{H^k_0(\Bn)} \leq  C |\sigma|^{-1+\frac{n}{2}+k}  ||v||_{L^2(\Bn)}, \;\ k=0,1,2, \\
|| x^a R_\delta (\sigma) x^b v||_{L^2(\Bn)} \leq  C |\sigma|^{-1+\frac{n}{2}+k}  ||v||_{H_0^{-k}(\Bn)}, \;\ k=0,1,2, \\
\end{gathered} \label{sobolev1}
\end{gather}
If $a=\im \sigma,$ or $b=\im\sigma,$ or $a=b=\im\sigma,$ let $\phi_N(x)\in\CI((0,1)),$ $\phi_N\geq 1,$ $\phi_N(x)=|\log x|^{-N},$ if $x<\oq$ $\phi_N(x)=1$ if $x>\ha.$ Then in each case the operator
\begin{gather}
\begin{gathered}
T_{a,b,N}(\sigma)=:x^{\im\sigma} \phi_N(x) R_{\delta}(\sigma) x^b, \text{ if } b>\im \sigma, \\
T_{a,b,N=:}x^{a} R_{\delta}(\sigma) x^{\im \la} \phi_N(x), \text{ if } a> \im \sigma, \\
T_{a.b,N}=:x^{\im\sigma} \phi_N(x) R_{\delta}(\sigma) x^{\im \sigma} \phi_N(x),
\end{gathered}\label{deftabn}
\end{gather}  
 continues holomorphically to $\im \sigma < M,$ provided $N>\ha,$ $| \sigma| >K(\delta,M).$ Moreover 
 in each case there exists $C=C(M,N,\delta)$ such that
\begin{gather}
\begin{gathered}
|| T_{a,b,N}(\sigma) v||_{H^k_0(\Bn)} \leq  C |\sigma|^{-1+\frac{n}{2}+k}  ||v||_{L^2(\Bn)}, \;\ k=0,1,2, \\
|| T_{a,b,N}(\sigma) v||_{L^2(\Bn)} \leq  C |\sigma|^{-1+\frac{n}{2}+k}  ||v||_{H_0^{-k}(\Bn)}, \;\ k=0,1,2.
\end{gathered}\label{sobolev2}
\end{gather}
\end{thm}

\section{The distance Function}\label{distance-function}

In the construction of the uniform parametrix for the resolvent we will
make use of an appropriate resolution of the distance function, and
geodesic flow, for the metric $g_\delta.$ This in turn is obtained by
perturbation from $\delta =0,$ so we start with an analysis of the
hyperbolic distance, for which there is an explicit formula. Namely, the
distance function for the hyperbolic metric, $g_{0},$ is given in terms of
the Euclidean metric on the ball by
\begin{equation}
\begin{gathered}
\dist_{0}:\Bno \times \Bno \longrightarrow \mr\Mwhere\\
\cosh(\dist_{0}(z,z'))=1+\frac{2|z-z'|^2}{(1-|z|^2)(1-|z'|^2)}.
\end{gathered}
\label{distf}\end{equation}

We are particularly interested in a uniform description as one or both of the
points approach the boundary, i.e.\ infinity. The boundary behavior is
resolved by lifting to the `zero stretched product' as is implicit in
\cite{Mazzeo:Edge}. This stretched product, $\Bnc\times_0 \Bnc,$ is the
compact manifold with corners defined by blowing up the intersection of the
diagonal and the corner of $\Bnc\times\Bnc:$
\begin{equation}
\begin{gathered}
\beta:\Bnc\times_0\Bnc=[\Bnc\times\Bnc;\pa\Diag]\longrightarrow\Bnc\times\Bnc,
\\
\pa\Diag=\Diag\cap \; (\p\Bnc \times \p \Bnc)=\{(z,z);|z|=1\},\\
\Diag=\{(z,z')\in \Bnc\times\Bnc;z=z'\}.
\end{gathered}\label{zero-blow-up}
\end{equation}
See Figure~\ref{fig1} in which $(x,y)$ and $(x',y')$ are
local coordinates near a point in the center of the blow up, with boundary
defining functions $x$ and $x'$ in the two factors.
\begin{figure}[int1]
\epsfxsize= 3.5in
\centerline{\epsffile{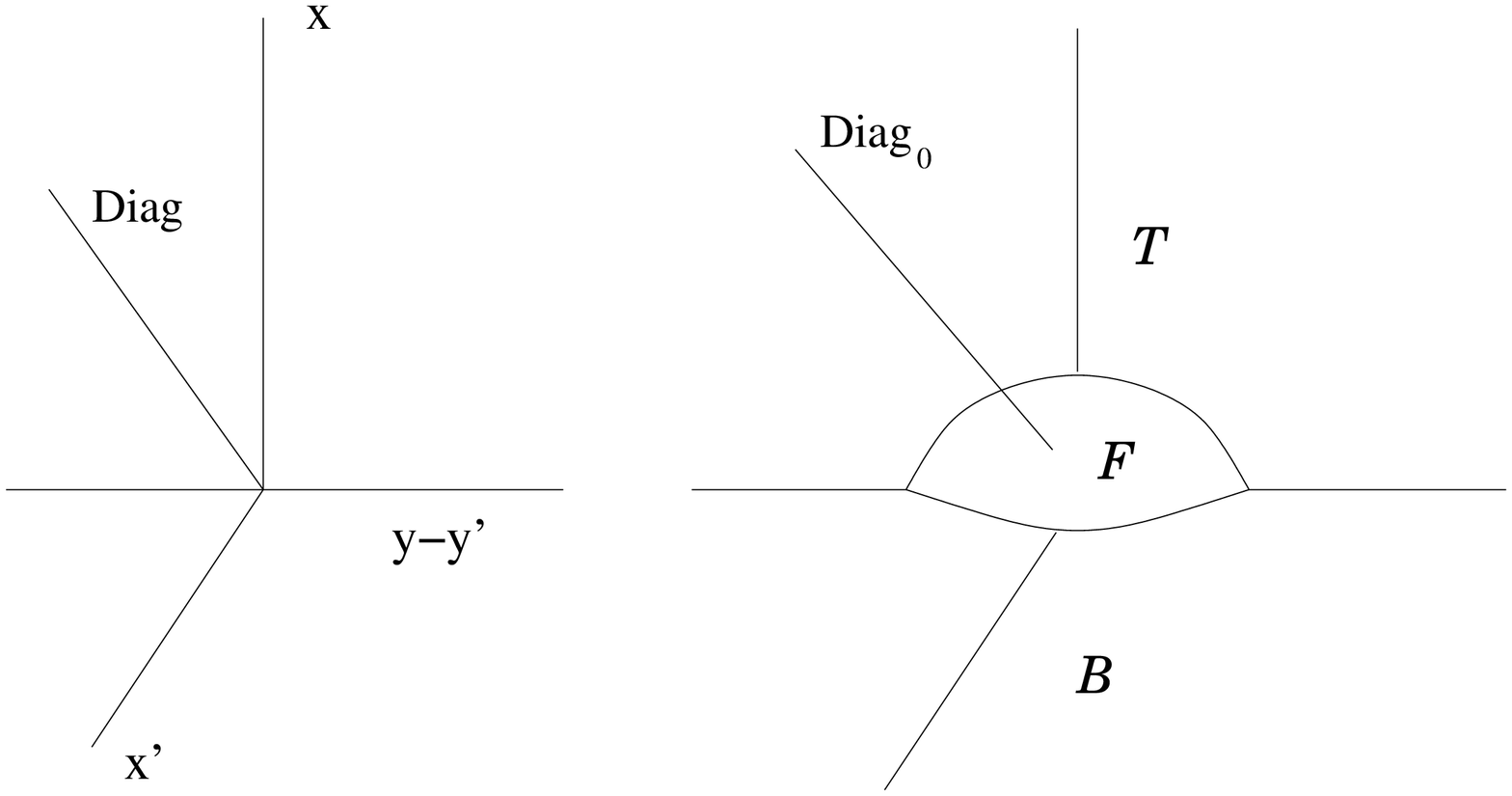}}
\caption{The stretched product $\Bnc\times_0\Bnc.$}
\label{fig1}
\end{figure}

Thus $\Bnc\times_0\Bnc$ has three boundary hypersurfaces, the front face
introduced by the blow up and the left and right boundary faces which map
back to $\pa\Bnc\times\Bnc$ and $\Bnc\times\pa\Bnc$ respectively under
$\beta.$ Denote by $\Diag_0$ the lift of the diagonal, which in this case
is the closure of the inverse image under $\beta$ of the interior of the
diagonal in $\Bnc\times\Bnc.$

\begin{lemma}\label{dist0} Lifted to the interior of $\Bnc\times_0\Bnc$ the
  hyperbolic distance function extends smoothly up to the interior of the
  front face, in the complement of $\Diag_0,$ where it is
  positive and, for an appropriate choice $\rho _L\in\CI(\Bnc\times_0\Bnc)$
  of defining function for the left boundary and 
  with $\rho _R$ its reflection,
\begin{equation}
\begin{gathered}
\beta^*\dist_{0}(z,z')=
-\log(\rho_L\rho_R)+ F,\\
0<F\in\CI(\Bnc \times_0 \Bnc\setminus \Diag_0 ),\
F^2\in\CI(\Bnc \times_0 \Bnc),
\end{gathered}
\label{blowd}\end{equation}
with $F ^2$ a quadratic defining function for $\Diag_0.$
\end{lemma}

\begin{proof} We show first that the square of the Euclidean distance
  function, $|z-z'|^2,$ lifts to be smooth on $\Bnc \times_0 \Bnc$ and to
  vanish quadratically on $\Diag_0$ and on the front face produced by the
  blow up 
\begin{equation}
\beta ^*(|z-z'|^2)=R^2f,\ f\in\CI(\Bnc \times_0 \Bnc).
\label{SeClRe.12}\end{equation}
Here $f\ge0$ vanishes precisely at $\Diag_0$ and does so
quadratically.  Indeed this is certainly true away from the front face
produced by the blow up. The spaces and the distance function are invariant
under rotational symmetry, which lifts under the blow up, so me may fix the
spherical position of one variable and suppose that $z'=(1-x',0),$ with
$x'>0$ and small, the blow up is then of $z=z',$ $x'=0.$ The fact that the
variables are restricted to the unit ball is now irrelevant, and using the
translation-invariance of the Euclidean distance we can suppose instead
that $z'=(x',0)$ and blow up $z=z',$ $x'=0.$ Since $|(x,0)-z|^2$ is
homogeneous in all variables it lifts to be the product of the square of a
defining function for the front face and a quadratic defining function for
the lift of the diagonal. This proves \eqref{SeClRe.12} after restriction
to the preimage of the balls and application of the symmetry.

The hyperbolic distance is given by \eqref{distf} where $1-|z|^2$ and
$1-|z'|^2$ are boundary defining functions on the two factors. If $R$ is
a defining function for the front face of $\Bnc\times_0\Bnc$ then these lift
to be of the form $\rho _LR$ and $\rho _RR$ so combining this with
\eqref{SeClRe.12}
\begin{equation}
\beta ^*\cosh(\dist_{0}(z,z'))=1+2\frac{f}{\rho _L\rho _R}.
\label{SeClRe.8}\end{equation}

Now $\exp(t)= \cosh t+ \left[ \cosh^2t-1\right]^\ha,$ for $t>0,$ and from
\eqref{SeClRe.8} it follows that
\begin{equation}
\exp (\dist_{0}(z,z'))=1+2\frac{f}{\rho _L\rho _R}+\left(2\frac{f}{\rho
  _L\rho _R}+4\frac{f^2}{\rho^2 _L\rho^2 _R}\right)^{\frac12}.
\label{SeClRe.9}\end{equation}
Near $\Diag_0$ the square-root is dominated by the first part and near the
left and right boundaries by the second part, and is otherwise positive and
smooth. Taking logarithms gives the result as claimed, with the
defining function taken to be one near $\Diag_0$ and to be everywhere smaller
than a small positive multiple of $(1-|z|^2)/R.$ 
\end{proof}

This result will be extended to the case of a perturbation of the
hyperbolic metric by constructing the distance function directly from
Hamilton-Jacobi theory, i.e.\ by integration of the Hamilton vector field
of the metric function on the cotangent bundle. The presence of only simple
logarithmic singularities in \eqref{blowd} shows, perhaps somewhat
counter-intuitively, that the Lagrangian submanifold which is the graph of
the differential of the distance should be smooth (away from the diagonal)
in the b-cotangent bundle of $M^2_0.$ Conversely if this is shown for
the perturbed metric then the analogue of \eqref{blowd} follows except for
the possibility of a logarithmic term at the front face.

Since the metric is singular near the boundary, the dual metric function on
$T^*\Bnc$ is degenerate there. In terms of local coordinates near a
boundary point, $x,$ $y$ where the boundary is locally $x=0,$ and dual
variables $\xi,$ $\eta,$ the metric function for hyperbolic space is of the
form
\begin{equation}
2p_0=x^2\xi^2 + 4 x^2(1-x^2)^{-2} h_0(\omega,\eta)
\label{SeClRe.15}\end{equation}
where $h_0$ is the metric function for the induced metric on the boundary. 

Recall that the $0$-cotangent bundle of a manifold with boundary $M,$
denoted $\To M,$ is a smooth vector bundle over $M$ which is a rescaled
version of the ordinary cotangent bundle. In local coordinates near, but
not at, the boundary these two bundles are identified by the (rescaling) map
\begin{equation}
T^* M\ni(x,y, \xi,\eta) \longmapsto (x,y,\la,\mu)=(x,y, x\xi,x \eta)\in \To M.
\label{SeClRe.40}\end{equation}
It is precisely this rescaling which makes the hyperbolic metric into a
non-degenerate fiber metric, uniformly up to the boundary, on this
bundle. On the other hand the b-cotangent bundle, also a completely natural
vector bundle, is obtained by rescaling only in the normal variable
\begin{equation}
T^* M\ni(x,y, \xi,\eta)\longmapsto (x,y,\la,\eta)=(x,y, x\xi,\eta)\in\Tb^*M.
\label{SeClRe.20}\end{equation}
Identification over the interior gives natural smooth vector bundle maps 
\begin{equation}
\iota_{\text{b}0}:\oT M\longrightarrow \Tb M,\
\iota^t_{\text{b}0}:\Tb^* M\longrightarrow \To M.
\label{SeClRe.27}\end{equation}
The second scaling map can be constructed directly in terms of blow up.

\begin{lemma}\label{SeClRe.26} If $\To\pa M\subset\Tos{\pa M}M$ denotes
  the annihilator of the null space, over the boundary, of $\iota
  _{\text{b}0}$ in \eqref{SeClRe.27} then there is a canonical diffeomorphism
\begin{equation}
\Tb^*M \longrightarrow [\To M,\To\pa M]\setminus
\beta^{\#}\left(\Tos{\pa M}M\right),\
\beta:[\To M,\To\pa M]\longrightarrow \To M,
\label{SeClRe.28}\end{equation}
to the complement, in the blow up, of the lift of the boundary: 
\begin{equation*}
\beta ^{\#}(\Tos{\pa M}M)=\overline{\beta^{-1}(\Tos{\pa M}M\setminus\To\pa M)}.
\label{SeClRe.29}\end{equation*}
\end{lemma}

\begin{proof} In local coordinates $x,$ $y_j,$ the null space of $\iota
  _{\text{b}0}$ in \eqref{SeClRe.27} is precisely the span of the
  `tangential' basis elements
  $x\pa_{y_j}$ over each boundary point. Its annihilator, $\To\pa M$ is
  given in the coordinates \eqref{SeClRe.20} by $\mu=0$ at $x=0.$
The lift of the `old boundary' $x=0$ is precisely the boundary
  hypersurface near which $|\mu|$ dominates $x.$ Thus, $x$ is a valid
  defining function for the boundary of the complement, on the right in
  \eqref{SeClRe.28} and locally in this set, above the coordinate patch in $M,$
  $\eta_j=\mu_j/x$ are smooth functions. The natural bundle map
  $\Tb^*M\longrightarrow \To M$ underlying \eqref{SeClRe.28} is given in
  these coordinates by $\lambda \frac{dx}x+\eta\cdot dy\longmapsto
  \lambda \frac{dx}x+x\eta\cdot\frac{dy}x,$ which is precisely the same
  map, $\mu=x\eta,$ as appears in \eqref{SeClRe.20}, so the result,
  including naturality, follows.
\end{proof}

The symplectic form lifted to $\To M$ is
\begin{gather*}
\bo \omega=\frac{1}{x} d\la \wedge dx +
\frac{1}{x} d\mu \wedge dy - \frac{1}{x^2} dx \wedge (\mu\cdot dy)
\end{gather*}
whereas lifted to $\Tb^*M$ it is 
\begin{equation}
{}^{\text{b}}\omega=\frac{1}{x} d\la \wedge dx+ d\eta \wedge dy .
\label{SeClRe.21}\end{equation}

Working, for simplicity of computation, in the non-compact upper half-space
model for hyperbolic space the metric function lifts to the non-degenerate
quadratic form on $\To M:$
\begin{equation}
2p_0=\la^2+h_0(\omega,\mu)
\label{SeClRe.16}\end{equation}
where $h_0=|\mu|^2$ is actually the Euclidean metric. The $0$-Hamilton
vector field of $p_0\in\CI(\To M),$ just the lift of the Hamilton
vector field over the interior, is determined by 
\begin{equation}
\bo\omega(\cdot, \boH_p)=dp.
\label{SeClRe.44}\end{equation}
Thus
\begin{gather*}
\boH_p= x\frac{\p p}{\p \la} \p_x + x \frac{\p p}{\p \mu} \cdot \p_y -
\left( \mu\cdot \frac{\p p}{\p \mu} +
x \frac{\p p}{\p x}\right) \p_{\la} -
\left( -\frac{\p p}{\p \la}\mu + x \frac{\p p}{\p y}\right)\cdot \p_\mu
\end{gather*}
and hence 
\begin{equation}
\boH_{p_0}=\la( x\p_x+\mu\p_\mu)-h_0\p_\la + \frac x2\sH_{h_0}
\label{SeClRe.17}\end{equation}
is tangent to the smooth (up to the boundary) compact sphere bundle given
by $p_0=1.$ 

Over the interior of $M=\bbB^n,$ the hyperbolic distance from any interior
point of the ball is determined by the graph of its differential, which is
the flow out inside $p_0=1,$ of the intersection of this smooth compact
manifold with boundary with the cotangent fiber to the initial
point. Observe that $\sH_{p_0}$ is also tangent to the surface $\mu=0$ over
the boundary, which is the invariantly defined subbundle $\To\pa M.$ Since the
coordinates can be chosen to be radial for any interior point, it follows
that all the geodesics from the interior arrive at the boundary at $x=0,$
$\mu=0,$ corresponding to the well-known fact that hyperbolic geodesics are
normal to the boundary. This tangency implies that $\sH_{p_0}$ lifts under
the blow up of $\To\pa M$ in \eqref{SeClRe.28} to a
smooth vector field on $\Tb M;$ this can also be seen by direct computation.

\begin{lemma}\label{SeClRe.19} The graph of the differential of the
  distance function from any interior point, $p\in\Bno,$ of hyperbolic
  space extends by continuity to a smooth Lagrangian submanifold of
  $\Tb^*(\Bnc\setminus\{p\})$ which is transversal to the boundary, is a
  graph over $\Bnc\setminus\{p\}$ and is given by integration of a
  non-vanishing vector field up to the boundary.
\end{lemma}

\begin{proof} Observe the effect of blowing up $\mu=0,$ $x=0$ on the
  Hamilton vector field in \eqref{SeClRe.17}. As noted above, near the
  front face produced by this blow up valid coordinates are given by
  $\eta=\mu/x,$ $\la$ and $y,$ with $x$ the boundary defining
  function. Since this transforms $\bo\omega$ to ${}^{\text{b}}\omega$ it
  follows that $\boH_{p_0}$ is transformed to  
\begin{equation}
{}^{\text{b}}\sH_{p_0}=x(\la \p_x-xh_0\p_{\la}+x\sH_{h_0})
\label{SeClRe.22}\end{equation}
where now $\sH_{h_0}$ is the Hamilton vector field with respect to $y$ and
$\eta.$

The constant energy surface $p_0=1$ remains smooth, but non-compact, near
the boundary, which it intersects transversally in $\la=1.$ From this
the result follows -- with the non-vanishing smooth vector field being
${}^{\text{b}}\sH_{p_0}$ divided by the boundary defining function $x$ near
the boundary.
\end{proof}

The logarithmic behaviour of the distance function in \eqref{blowd} with
one point fixed in the interior is a consequence of Lemma~\ref{SeClRe.26},
since the differential of the distance must be of the form $adx/x+b\cdot
dy$ for smooth functions $a$ and $b,$ and since it is closed, $a$ is
necessarily constant on the boundary.

To examine the distance as a function of both variables a similar
construction for the product, in this case $M^2,$ $M=\Bnc$ can be used. The
graph, $\Lambda,$ of the differential of the distance $d(p,p')$ as a
function on $M^2$ is the joint flow out of the conormal sphere bundle to
the diagonal in $T^*M^2=T^*M\times T^*M,$ under the two Hamiltonian vector
fields of the two metric functions within the product sphere bundles. As
before it is natural to lift to $\To M\times\To M$ where the two sphere
bundles extend smoothly up to the boundary. However, one can make a
stronger statement, namely that the lifted Hamilton vector fields are
smooth on the b-cotangent bundle of $M^2_0,$ and indeed even on the
`partially b'-cotangent bundle of $M^2_0$, with `partially' meaning it is
the standard cotangent bundle over the interior of the front face. This is
defined and discussed in more detail below; note that the identification of
these bundles over the interior of $M^2_0$ extends to a smooth map from
these bundles to the lift of $\To M\times\To M$, as we show later,
explaining the `stronger' claim.

Smoothness of the Hamilton vector field together with transversality
conditions shows that the flow-out of the conormal bundle of the diagonal
is a smooth Lagrangian submanifold of the cotangent bundle under
consideration; closeness to a particular Lagrangian (such as that for
hyperbolic space) restricted to which projection to $M^2_0$ is a
diffeomorphism, guarantees that this Lagrangian is also a graph over
$M^2_0.$ Thus, over the interior, $(M^2_0)^\circ$, it is the graph of the
differential of the distance function, and the latter is smooth; the same
would hold globally if the Lagrangian were smooth on $T^*M^2_0.$ The latter
cannot happen, though the Lagrangian will be a graph in the b-, and indeed
the partial b-, cotangent bundles over $M^2_0.$ These give regularity of
the distance function, namely smoothness up to the front face (directly for
the partial b bundle, with a short argument if using the b bundle), and the
logarithmic behavior up to the other faces. Note that had we only showed
the graph statement in the pullback of $\To M\times\To M,$ one would obtain
directly only a weaker regularity statement for the distance function;
roughly speaking, the closer the bundle in which the Lagrangian is
described is to the standard cotangent bundle, the more regularity the
distance function has.

In fact it is possible to pass from the dual of the lifted product
0-tangent bundle to the dual of the b-tangent bundle, or indeed the partial
b-bundle, by blow-up, as for the single space above. Observe first that the
natural inclusion
\begin{equation}
\iota_{0\text{b}}\times \iota_{0\text{b}}:\oT M\times\oT M\longrightarrow
\Tb M^2=\Tb M\times\Tb M
\label{SeClRe.31}\end{equation}
identifies the sections of the bundle on the left with those sections of the
bundle on the right, the tangent vector fields on $M^2,$ which are also
tangent to the two fibrations, one for each boundary hypersurface
\begin{equation}
\phi_L:\pa M\times M\longrightarrow M,\ \phi_R:M\times\pa M\longrightarrow M.
\label{SeClRe.33}\end{equation}

\begin{lemma}\label{SeClRe.32} The fibrations \eqref{SeClRe.33}, restricted
  to the interiors, extend by continuity to fibrations $\phi_L$, resp.\ $\phi_R$,
of the two `old'
  boundary hypersurfaces of $M^2_0$ and the smooth sections of the lift of
  $\oT M\times\oT M$ to $M^2_0$ are naturally identified with the subspace of the
  smooth sections of $\Tb M^2_0$ which are tangent to these fibrations and
  also to the fibres of the front face of the blow up,
  $\beta_0:\ff(M^2_0)\longrightarrow \pa M\times\pa M.$
\end{lemma}

\begin{proof} It is only necessary to examine the geometry and vector
  fields near the front face produced by the the blow up of the diagonal
  near the boundary. Using the symmetry between the two factors, it
  suffices to consider two types of coordinate systems. The first is valid
  in the interior of the front fact and up to a general point in the interior
  of the intersection with one of the old boundary faces. The second is
  valid near a general point of the corner of the front face, which has
  fibers which are quarter spheres.

For the first case let $x,$ $y$ and $x',$ $y'$ be the same local
coordinates in two factors. The coordinates
\begin{equation}
s=x/x',\ x',\ y\Mand Y=(y'-y)/x'
\label{SeClRe.35}\end{equation}
are valid locally in $M^2_0$ above the point $x=x'=0,$ $y=y',$ up to the
lift of the old boundary $x=0,$ which becomes locally $s=0.$ The fibration
of this hypersurface is given by the constancy of $y$ 
and the front face is $x'=0$ with fibration also given by the constancy of
$y.$ The vector fields
$$
x\pa_x,\  x\pa_y,\ x'\pa_{x'}\Mand x'\pa_{y'}
$$
lift to
$$
s\pa_s,\ sx'\pa_y-s\pa_Y,\ x'\pa_{x'}-s\pa_s-Y\cdot\pa_Y\Mand\pa_Y.
$$
The
basis $s\pa_s,$ $sx'\pa_y,$ $x'\pa_{x'}$ and $\pa_Y$ shows that these
vector fields are locally precisely the tangent vector fields also tangent
to both fibrations.

After relabeling the tangential variables as necessary, and possibly
switching their signs, so that $y'_1-y_1>0$
is a dominant variable, the coordinate system
\begin{equation}
t=y_1'-y_1,\ s_1=\frac{x}{y'_1-y_1},\
s_2=\frac{x'}{y'_1-y_1},\ Z_j=\frac{y'_j-y_j}{y'_1-y_1},\ j>1,\ y
\label{22.9.2010.1}\end{equation}
can be used at a point in the corner of the front face. The three boundary
hypersurfaces are locally $s_1=0,$ $s_2=0$ and $t=0$ and their respective
fibrations are given in these coordinates by 
\begin{equation}
\begin{gathered}
s_1=0,\ y=\text{const.,}\\
s_2=0,\ y'_1=y_1+t=\text{const.},\ y'_j=y_j+tZ_j=\text{const.},\ j>1,\\
t=0,\ y=\text{const.}
\end{gathered}
\label{22.9.2010.2}\end{equation}
Thus, the intersections of fibres of the lifted left or right faces with
the front face are precisely boundary hypersurfaces of fibres there. On the
other hand within the intersection of the lifted left and right faces the
respective fibres are transversal except at the boundary representing the
front face. The lifts of the basis of the zero vector fields is easily
computed:
\begin{equation}
\begin{gathered}
x\pa_x\longmapsto s_1\pa_{s_1},\\
x\pa_{y_1}\longrightarrow
-s_1t\pa_t+s_1^2\pa_{s_1}+s_1s_2\pa_{s_2}+s_1Z\cdot\pa_Z+s_1t\pa_{y_1},\\
x\pa_{y_j}\longmapsto s_1t\pa_{y_j}-s_1\pa_{Z_j},\\
x'\pa_{x'}\longmapsto s_2\pa_{s_2},\\
x'\pa_{y'_1}\longrightarrow s_2t\pa_t-s_2s_1\pa_{s_1}-s_2^2\pa_{s_2}-s_2Z\cdot\pa_{Z},\\
x'\pa_{y'_j}\longmapsto s_2\pa_{Z_j}.
\end{gathered}
\label{22.9.2010.6}\end{equation}
The span, over $\CI(M^2_0),$ of these vector fields is also spanned by
$$
s_1\pa_{s_1},\ s_2\pa_{s_2},\ s_2\pa_{Z_j},\ j>1,\ s_1(t\pa_{y_j}-\pa_{Z_j}),\ j>1,
\ s_1(t\pa_t-t\pa_{y_1}-Z\cdot\pa_Z)\Mand s_2t\pa_t.
$$
These can be seen to locally
span the vector fields tangent to all three boundaries and corresponding
fibrations, proving the Lemma.
\end{proof}

With $\phi=\{\phi_L,\phi_R,\beta_0\}$
the collection of boundary fibrations, we denote by $\Tph M^2_0$ the
bundle whose smooth sections are exactly the smooth vector fields
tangent to all boundary fibrations. Thus, the content of the preceding lemma
is that
$$
\beta^*(\oT M\times \oT M)=\Tph M^2_0.
$$

These fibrations allow the reconstruction of $\Tb^*M^2_0$ as a blow up of
the lift of $\To M\times\To M$ to $M^2_0.$ It is also useful, for more
precise results later on, to consider the `partially b-' cotangent bundle
of $M^2_0$, $\Tbff^*M^2_0$; this is the dual space of the partially
b-tangent bundle, $\Tbff M^2_0,$ whose smooth sections are smooth vector
fields on $M^2_0$ which are tangent to the old boundaries, but not
necessarily to the front face, $\ff$. Thus, in coordinates
\eqref{SeClRe.35}, $s\pa_s$, $\pa_{x'}$, $\pa_y$ and $\pa_Y$ form a basis
of $\Tbff M^2_0$, while in coordinates \eqref{22.9.2010.1}, $s_1\pa_{s_1}$,
$s_2\pa_{s_2}$, $\pa_t$, $\pa_{Z_j}$ and $\pa_y$ do so.  Let
$$
\iota^2_{0\bl}:\Tph M^2_0\to\Tb M^2_0,\ \iota^2_{0\bl,\ff}:\Tph M^2_0\to\Tbff M^2_0
$$
be the inclusion maps.

\begin{lemma}\label{SeClRe.34} The annihilators, in the lift of $\To M\times\To
M$ to $M^2_0,$ of the null space of either $\iota^2_{0\bl}$ or
$\iota^2_{0\bl,\ff}$ over the old boundaries, as in
  Lemma~\ref{SeClRe.32}, form transversal embedded p-submanifolds. After
  these are blown up, the closure of the annihilator of the nullspace of
$\iota^2_{0\bl}$, resp.\ $\iota^2_{0\bl,\ff}$, over
  the interior of the front face of $M^2_0$ is a p-submanifold, the
  subsequent blow up of which produces a manifold with corners with three
  `old' boundary hypersurfaces; the complement of these three hypersurfaces
  is canonically diffeomorphic to $\Tb^*M^2_0,$ resp.\ $\Tbff^* M^2_0$.
\end{lemma}

\begin{proof} By virtue of Lemma~\ref{SeClRe.26} and the product structure away
from the front face $\ff$ of $M^2_0$,
the statements here are trivially valid except possibly
  near $\ff$. We may again use the coordinate systems
  discussed in the proof of Lemma~\ref{SeClRe.32}. Consider the linear
  variables in the fibres in which a general point is $l x\pa_x+v\cdot
  x\pa_y+l'x'\pa_{x'}+v'\cdot x'\pa_{y'}.$

First consider the inclusion into $\Tb M^2_0$. In the interiors of $s_1=0$ and
  $s_2=0$ and the front face respectively, the null bundles of the
  inclusion into the tangent vector fields are
\begin{equation}
\begin{gathered}
l=l'=0,\ v'=0,\\
l=l'=0,\ v=0,\\
\begin{aligned}
l s\pa_s+v(sx'\pa_y-s\pa_Y)+l'(x'\pa_{x'}-s\pa_s-Y\pa_Y)+v'\pa_Y=0\Mat x'=0,\ s>0\\
\Longleftrightarrow
l=l'=0,\ v'=sv.
\end{aligned}
\end{gathered}
\label{22.9.2010.8}\end{equation}

The corresponding annihilator bundles, over the interiors of the boundary
hypersurfaces of $M^2_0,$ in the dual bundle, with basis 
\begin{equation}
\la \frac{dx}x+\mu \frac{dy}x+\la '\frac{dx'}{x'}+\mu'\frac{dy'}{x'}
\label{22.9.2010.9}\end{equation}
are therefore, as submanifolds,
\begin{equation}
\begin{gathered}
s_1=0,\ \mu =0,\\
s_2=0,\ \mu '=0\text{ and}\\
x'=0,\ \mu+s\mu'=0\Mor t=0,\ s_2\mu +s_1\mu'=0.
\end{gathered}
\label{22.9.2010.10}\end{equation}
Here the annihilator bundle over the front face is given with respect to both
the coordinate system \eqref{SeClRe.35} and \eqref{22.9.2010.1}.

Thus, the two subbundles over the old boundary hypersurfaces meet
transversally over the intersection, up to the corner, as claimed and so
can be blown up in either order. In the complement of the lifts of the old
boundaries under these two blow ups, the variables $\mu/s_1$ and
$\mu'/s_2$ become legitimate; in terms of these the subbundle over the
front face becomes smooth up to, and with a product decomposition at, all
its boundaries. Thus, it too can be blown up. That the result is a
(painful) reconstruction of the b-cotangent bundle of the blown up manifold
$M^2_0$ follows directly from the construction.

It remains to consider the inclusion into $\Tbff M^2_0$. The only changes are
at the front face, namely the third line of \eqref{22.9.2010.8} becomes
\begin{equation}\begin{gathered}
l s\pa_s+v(sx'\pa_y-s\pa_Y)+l'(x'\pa_{x'}-s\pa_s-Y\pa_Y)+v'\pa_Y=0\Mat x'=0,\ s>0\\
\Longleftrightarrow
l=l',\ v'=sv+l'Y.
\end{gathered}\end{equation}
Correspondingly the third line of \eqref{22.9.2010.10} becomes
\begin{equation}\begin{gathered}\label{22.9.2010.10b}
x'=0,\ \lambda+\lambda'+\mu' Y=0,\ \mu+s\mu'=0\\
\Mor t=0,\ s_2\mu +s_1\mu'=0,
\ s_2(\lambda+\lambda')+\mu'_1+\sum_{j\geq 2}\mu'_j Z_j=0.
\end{gathered}\end{equation}
The rest of the argument is unchanged, except that the conclusion is that
$\Tbff^* M^2_0$ is being reconstructed.
\end{proof}

Note that for any manifold with corners, $X,$ the b-cotangent
bundle of any boundary hypersurface $H$ (or indeed any boundary face) includes
naturally as a subbundle $\Tb^*H\hookrightarrow \Tb^*_HX.$

\begin{lemma}
The Hamilton vector field of $g_\delta$ lifts from either the left or the right factor
of $M$ in $M^2$ to a smooth vector field, tangent to the boundary hypersurfaces,
on $\Tb^* M^2_0$, as well as on $\Tbff^* M^2_0$,
still denoted by $\sH^L_{g_\delta}$, resp.\ $\sH^R_{g_\delta}$. Moreover,
$\sH^L_{g_\delta}=\rho_L V^L$, $\sH^R_{g_\delta}=\rho_R V^R$, where
$V_L$, resp.\ $V_R$ are smooth vector fields tangent to all hypersurfaces
except the respective
cotangent bundles over the left, resp.\ right, boundaries, to which
they are transversal, and where $\rho_L$ and $\rho_R$ are defining functions of the
respective cotangent bundles over these boundaries.
\end{lemma}

\begin{proof}
Inserting the explicit form of the Euclidean metric, the Hamilton vector
field in \eqref{SeClRe.17} becomes 
\begin{equation}
\boH_{p_0}=\la( x\p_x+\mu\p_\mu)-|\mu|^2\p_\la + x\mu\cdot\pa_y.
\label{SeClRe.37}\end{equation}
Consider the lift of this vector field to the product, $\To M\times\To M,$
from left and right, and then under the blow up of the diagonal in the
boundary. In the coordinate systems \eqref{SeClRe.35} and \eqref{22.9.2010.1} 
\begin{equation}
\begin{gathered}
\boH_{p_0}^L=\lambda (s\pa_s+\mu\pa_\mu)-|\mu|^2\pa_{\la}+
s\mu(x'\pa_y-\pa_Y)\\
\boH_{p_0}^R=\lambda'(x'\pa_{x'}-s\pa_s-Y\pa_Y+\mu'\pa_{\mu'})-|\mu'|^2\pa_{\la'}+\mu'\pa_Y
\\
\begin{aligned}
\boH_{p_0}^L=\la(s_1\pa_{s_1}+\mu\pa_{\mu})&
-|\mu|^2\pa_{\la}+s_1\sum\limits_{j\ge2}\mu_j(t\pa_{y_j}-\pa_{Z_j})\\
+&s_1\mu_1(t\pa_{y_1}-t\pa_t+s_1\pa_{s_1}+s_2\pa_{s_2}
+\sum\limits_{j\ge2}Z_j\pa_{Z_j})
\end{aligned}\\
\begin{aligned}
\boH_{p_0}^R=\la'(s_2\pa_{s_2}+\mu'\pa_{\mu'})&
-|\mu'|^2\pa_{\la'}+s_2\sum\limits_{j\ge2}\mu'_j\pa_{Z_j}\\
+&s_2\mu'_1(t\pa_t-s_1\pa_{s_1}-s_2\pa_{s_2}-\sum\limits_{j\ge2}Z_j\pa_{Z_j}).
\end{aligned}
\end{gathered}
\label{SeClRe.38}\end{equation}
Note that the bundle itself is just pulled back here, so only
the base variables are changed.

Next we carry out the blow ups of Lemma~\ref{SeClRe.34}. The centers of
blow up are given explicitly, in local coordinates, in \eqref{22.9.2010.10}, with the
third line replaced by \eqref{22.9.2010.10b} in the case of $\iota^2_{0\bl,\ff}$.
We are only interested in the behaviour of the lifts of the vector
fields in \eqref{SeClRe.38} near the front faces introduced in the blow
ups.

Consider $\iota^2_{0\bl}$ first.
For the first two cases there are two blow-ups, first of $\mu=0$ in
$s=0$ and then of $\mu+s\mu'=0$ in $x'=0.$ Thus, near the front face of the
first blow up, the $\mu$ variables are replaced by $\tilde\mu=\mu/s$ and
then the center of the second blow up is $\tilde\mu+\mu'=0,$ $x'=0.$ Thus,
near the front face of the second blow up we can use as coordinates $s,$
$x',$ $\mu'$ and $\nu=(\tilde\mu+\mu')/x',$ i.e.\ substitute
$\tilde\mu=-\mu'+x'\nu.$ In the coordinate patch \eqref{SeClRe.38} the
lifts under the first blow up are
\begin{equation}
\begin{gathered}
\boH_{p_0}^L\longmapsto
\lambda s\pa_s-s^2|\tilde\mu|^2\pa_{\la}+
s^2\tilde\mu(x'\pa_y-\pa_Y)
\\
\boH_{p_0}^R\longmapsto
\lambda'(x'\pa_{x'}-s\pa_s-Y\pa_Y+\tilde\mu\pa_{\tilde\mu}+
\mu'\pa_{\mu'})-|\mu'|^2\pa_{\la'}+
\mu'\pa_Y.
\end{gathered}
\label{SeClRe.39}\end{equation}

Thus under the second blow up, the left Hamilton vector field lifts to
\begin{equation}
\begin{gathered}
\boH_{p_0}^L\longmapsto sT,\
T=\lambda \pa_s-s|\tilde\mu|^2\pa_{\la}+
s\tilde\mu(x'\pa_y-\pa_Y)
\end{gathered}
\label{SeClRe.41}\end{equation}
where $T$ is transversal to the boundary $s=0$ where $\lambda\not=0.$ 

A similar computation near the corner shows the lifts of the two Hamilton
vector fields under blow up the fibrations of $s_1=0$ and $s_2=0$ in terms
of the new coordinates $\tilde\mu=\mu/s_1$ and $\tilde\mu'=\mu'/s_2$ to be 
\begin{equation}
\begin{gathered}
\begin{aligned}
\boH_{p_0}^L=&\la s_1\pa_{s_1}-s_1^2|\tilde\mu|^2\pa_{\la}+s_1^2\sum\limits_{j\ge2}\tilde\mu_j(t\pa_{y_j}-\pa_{Z_j})\\
+&s_1^2\tilde\mu_1
(t\pa_{y_1}-t\pa_t+s_1\pa_{s_1}-\tilde\mu\pa_{\tilde\mu}+s_2\pa_{s_2}
-\tilde\mu'\pa_{\tilde\mu'}+\sum\limits_{j\ge2}Z_j\pa_{Z_j}),
\end{aligned}\\
\begin{aligned}
\boH_{p_0}^R=&\la's_2\pa_{s_2}-s_2^2|\tilde\mu'|^2\pa_{\la'}+s_2^2\sum\limits_{j\ge2}\tilde\mu'_j\pa_{Z_j}\\
+&s_2^2\tilde\mu'_1
(t\pa_t-s_1\pa_{s_1}+{\tilde\mu}\pa_{\tilde\mu}-s_2\pa_{s_2}+
\tilde\mu'\pa_{\tilde\mu'}
-\sum\limits_{j\ge2}Z_j\pa_{Z_j}).
\end{aligned}
\end{gathered}
\label{SeClRe.42}\end{equation}
The final blow up is that of $t=0$, $\tilde\mu+\tilde\mu'=0$, near the front
face of this blow-up replacing $(t,\tilde\mu,\tilde\mu')$ by $(t,\tilde\mu,\tilde\nu)$,
$\tilde\nu=(\tilde\mu+\tilde\mu')/t$, as valid coordinates (leaving the others
unaffected). Then the vector fields above become
\begin{equation}
\begin{gathered}
\boH_{p_0}^L=s_1\tilde T^L,\ \boH_{p_0}^R=s_2\tilde T^R,\\
\begin{aligned}
\tilde T^L=&\la \pa_{s_1}-s_1|\tilde\mu|^2\pa_{\la}+s_1\sum\limits_{j\ge2}\tilde\mu_j(t\pa_{y_j}-\pa_{Z_j})\\
+&s_1\tilde\mu_1
(t\pa_{y_1}-t\pa_t+s_1\pa_{s_1}-\tilde\mu\pa_{\tilde\mu}+s_2\pa_{s_2}
+\sum\limits_{j\ge2}Z_j\pa_{Z_j}),
\end{aligned}\\
\begin{aligned}
\tilde T^R=&\la'\pa_{s_2}-s_2|\tilde\mu'|^2\pa_{\la'}+s_2\sum\limits_{j\ge2}\tilde\mu'_j\pa_{Z_j}\\
+&s_2\tilde\mu'_1
(t\pa_t-s_1\pa_{s_1}+{\tilde\mu}\pa_{\tilde\mu}-s_2\pa_{s_2}
-\sum\limits_{j\ge2}Z_j\pa_{Z_j}).
\end{aligned}
\end{gathered}
\label{SeClRe.42b}\end{equation}
Thus both left and right Hamilton vector fields are transversal to the
respective boundaries after a vanishing factor is removed, provided
$\lambda,$ $\lambda '\not=0.$

The final step is to show that the same arguments apply to the perturbed
metric. First consider the lift, from left and right, of
the perturbation to the Hamilton vector field arising from the
perturbation of the metric. By assumption, the perturbation $H$ is a
2-cotensor which is smooth up to the boundary. Thus, as a perturbation of
the dual metric function on $\To M$ it vanishes quadratically at the
boundary. In local coordinates near a boundary point it follows that the
perturbation of the differential of the metric function is of the form 
\begin{equation}
dp-dp_0=x^2(a\frac{dx}x+b dy+cd\lambda +ed\mu)
\label{SeClRe.46}\end{equation}
From \eqref{SeClRe.44} it follows that the perturbation of the
Hamilton vector field is of the form 
\begin{equation}
\sH_{p}-\sH_{p_0}=x^2(a'x\pa_x+b'x\pa_y+c'\pa_\lambda +e'\pa_\mu)
\label{SeClRe.45}\end{equation}
on $\To M.$ Lifted from the right or left factors to the product and then
under the blow-up of the diagonal to $M^2_0$ it follows that in the
coordinate systems \eqref{SeClRe.35} and \eqref{22.9.2010.1}, the
perturbations are of the form 
\begin{equation}
\begin{gathered}
\sH_{p}^L-\sH_{p_0}^L=s^2(x')^2V^L,\ \sH_{p}^R-\sH_{p_0}^R=(x')^2V^R,\ V^L,\ V^R\in\Vb,\\
\sH_{p}^L-\sH_{p_0}^L=s_1^2t^2W^L ,\ \sH_{p}^R-\sH_{p_0}^R=s_2^2t^2W^R,\ W^R,\ W^R\in\Vb
\end{gathered}
\label{SeClRe.47}\end{equation}
where $\Vb$ denotes the space of smooth vector fields tangent to all
boundaries. Since the are lifted from the right and left factors, $V^L$ and
$V^R$ are necessarily tangent to the annihilator submanifolds of the right
and left boundaries. It follows that the vector fields $sx'V^L$ and $x'V^R$
are tangent to both fibrations above a coordinate patch as in
\eqref{SeClRe.35} and $s_1tV^L$ and $s_2tV^R$ are tangent to all three
annihilator submanifolds above a coordinate patch
\eqref{22.9.2010.1}. Thus, after the blow ups which reconstruct
$\Tb^*M^2_0,$ the perturbations lift to be of the form
\begin{equation}
\sH_{p}^L-\sH_{p_0}^L=\rho _L\rho _{\ff}U^L,\ \sH_{p}^R-\sH_{p_0}^R=\rho _R\rho _{\ff}U^R
\label{SeClRe.48}\end{equation}
where $U^L$ and $U^R$ are smooth vector fields on $\Tb^*M^2_0.$

From \eqref{SeClRe.48} it follows that the transversality properties in
\eqref{SeClRe.41} and \eqref{SeClRe.42b} persist.

Now consider $\iota^2_{0\bl,\ff}$. First, \eqref{SeClRe.39} is unchanged, since
the annihilators on the `old' boundary faces are the same in this case.
In particular, we still have $\tilde\mu=\mu/s$ as one of our coordinates
after the first blow up; the center of the second blow up is then $x'=0$,
$\lambda+\lambda'+\mu'\cdot Y=0$, $\tilde\mu+\mu'=0$.
Thus,
near the front face of the second blow up we can use as coordinates $s,$
$x',$ $\mu'$ and $\sigma=(\lambda+\lambda'+\mu'\cdot Y)/x'$,
$\nu=(\tilde\mu+\mu')/x',$ i.e.\ substitute
$\tilde\mu=-\mu'+x'\nu$, i.e.\ $\mu=-\mu's+x's\nu$,
and $\lambda=-\lambda'-\mu'\cdot Y+x'\sigma$.
Thus under the second blow up, the left Hamilton vector field lifts to
\begin{equation}
\begin{gathered}
\boH_{p_0}^L\longmapsto sT',\
T'=\lambda \pa_s+\mu\cdot(-\nu\pa_\sigma+x'\pa_y-\pa_Y),
\end{gathered}
\label{SeClRe.41b}\end{equation}
so $T'$ is transversal to the boundary $s=0$ where $\lambda\not=0.$ 

In the other coordinate chart, again, \eqref{SeClRe.42} is unchanged since the
annihilators on the `old' boundary faces are the same.
The final blow up is that of
$$
t=0,\ \tilde\mu+\tilde\mu'=0,
\ \lambda+\lambda'+\tilde\mu'_1+\sum_{j\geq 2}\tilde\mu'_j Z_j=0,
$$
near the front
face of this blow-up replacing $(t,\tilde\mu,\tilde\mu',\lambda,\lambda')$
by $(t,\tilde\mu,\tilde\nu,\lambda,\tilde\sigma)$,
$$
\tilde\nu=(\tilde\mu+\tilde\mu')/t,\ \tilde\sigma=(\lambda+\lambda'+\tilde\mu'_1+\sum_{j\geq 2}\tilde\mu'_j Z_j)/t,
$$
as valid coordinates (leaving the others
unaffected). Then the vector fields above become
\begin{equation}
\begin{gathered}
\boH_{p_0}^L=s_1\hat T^L,\ \boH_{p_0}^R=s_2\hat T^R,\\
\begin{aligned}
\hat T^L=&\la \pa_{s_1}-s_1|\tilde\mu|^2\pa_{\la}+s_1\sum\limits_{j\ge2}\tilde\mu_j(t\pa_{y_j}-\pa_{Z_j}-\tilde\nu_j\pa_{\tilde\sigma})\\
+&s_1\tilde\mu_1
(t\pa_{y_1}-t\pa_t+s_1\pa_{s_1}+\tilde\sigma\pa_{\tilde\sigma}-\tilde\nu\pa_{\tilde\nu}
+s_2\pa_{s_2}-\tilde\nu_1\pa_{\tilde\sigma}
+\sum\limits_{j\ge2}Z_j\pa_{Z_j}),
\end{aligned}\\
\begin{aligned}
\hat T^R=&\la'\pa_{s_2}+s_2\sum\limits_{j\ge2}\tilde\mu'_j\pa_{Z_j}\\
+&s_2\tilde\mu'_1
(t\pa_t-s_1\pa_{s_1}+{\tilde\mu}\pa_{\tilde\mu}-s_2\pa_{s_2}-\tilde\sigma\pa_{\tilde\sigma}
-\sum\limits_{j\ge2}Z_j\pa_{Z_j}).
\end{aligned}
\end{gathered}
\label{SeClRe.42c}\end{equation}
Again, both left and right Hamilton vector fields are transversal to the
respective boundaries after a vanishing factor is removed, provided
$\lambda,$ $\lambda '\not=0.$ The rest of the argument proceeds as above.
\end{proof}

\begin{prop}\label{SeClRe.23}\label{distance} The differential of the
  distance function for the perturbed metric $g_\delta,$ for sufficiently
  small $\delta,$ on $M=\Bnc$ defines a global smooth Lagrangian
  submanifold $\Lambda_\delta,$ of $\Tb^*(M\times_0M),$ which is a smooth
  section outside the lifted diagonal and which lies in $\Tb^*\ff$ over the
  front face, $\ff(M^2_0)$ and in consequence there is a unique geodesic
  between any two points of $\Bo,$ no conjugate points and \eqref{blowd}
  remains valid for $\dist_{\delta}(z,z').$
\end{prop}

\begin{proof} For the unperturbed metric this already follows from
  Lemma~\ref{dist0}. We first reprove this result by integrating the
  Hamilton vector fields and then examine the effect of the metric
  perturbation. Thus we first consider the lift of the Hamilton vector
  field of the hyperbolic distance function from $\To M,$ from either the
  left of the right, to $\Tbff^*M^2_0,$ using the preceding lemma.

Although the global regularity of the Lagrangian which is the graph of the
differential of the distance is already known from the explicit formula in
this case, note that it also follows from the form of these two vector
fields. The initial manifold, the unit conormal bundle to the diagonal,
becomes near the corner of $M^2$ the variety of those 
\begin{equation}
\xi(dx-dx')+\eta(dy-dy')\text{ such that
}\xi^2+|\eta|^2=\frac1{2x^2},\ x=x'>0,\ y=y'.
\label{SeClRe.49}\end{equation}
In the blown up manifold $M^2_0$ the closure is smooth in $\Tb^*M^2_0$ and
is the bundle over the lifted diagonal given in terms of local coordinates
\eqref{SeClRe.35} by  
\begin{equation}
\la\frac{ds}s+\mu dY,\ \la^2+|\mu|^2=\frac12,\ s=1,\ Y=0;
\label{SeClRe.50}\end{equation}
the analogous statement also holds in $\Tbff^*M^2_0$ where one has the `same'
expression.
Using the Hamilton flow in $\Tbff^*M^2_0$, we deduce that
the flow out is
a global smooth  submanifold, where smoothness includes up to all
boundaries, of $\Tbff^*M^2_0$, and is also globally a graph away from the lifted
diagonal, as follows from the explicit form of the vector fields. Note that
over the interior of the front face, $\Tbff^*M^2_0$ is just the standard
cotangent bundle, so smoothness of the distance up to the front face
follows.
Over the
left and right boundaries the Lagrangian lies in $\lambda =1$ and $\lambda
'=1$ so the form \eqref{blowd} of the distance follows.

The analogous conclusion can also be obtained by using the flow in
$\Tb^*M^2_0$. In this setting, we need that
over the front face \eqref{SeClRe.50} is contained in the image of $\Tb^*\ff$ to which
both lifted vector fields are tangent. Thus it follows that the flow out is
a global smooth  submanifold, where smoothness includes up to all
boundaries, of $\Tb^*M^2_0$ which is contained in the image of $\Tb^*\ff$
over $\ff(M^2_0).$ Again, it is globally a graph away from the lifted
diagonal, as follows from the explicit form of the vector fields. Over the
left and right boundaries the Lagrangian lies in $\lambda =1$ and $\lambda
'=1$ so the form \eqref{blowd} of the distance follows.

For small $\delta$ these
perturbation of the Hamilton vector fields are also small in supremum norm
and have the same tangency properties at the boundaries used above to
rederive \eqref{blowd}, from which the Proposition follows.
\end{proof}

\section{The semiclassical double space}\label{semiclassical-double-space}

In this section we construct the semiclassical double space, $M_{0,\semi}$, which
will be the locus of our parametrix construction. To motivate the construction, we
recall that
Mazzeo and the first author \cite{Mazzeo-Melrose:Meromorphic} have analyzed the resolvent $R(h,\sigma),$ defined in equation \eqref{scresolvent}, for $\sigma/h\in\Cx$, though the construction is
not uniform as $|\sigma/h|\to\infty$. They achieved this
by constructing the Schwartz kernel of the parametrix $G(h,\sigma)$ to
$R(h,\sigma)$ as a conormal distribution defined on the manifold
$$
M_0=\BB^{n+1} \times_0 \BB^{n+1}
$$
defined in 
\eqref{zero-blow-up}, see also figure \ref{fig1}, with meromorphic dependence on
$\sigma/h$. 

The manifold $ \BB^{n+1}\times \BB^{n+1}$ is a $2n+2$ dimensional manifold with corners.  It contains two boundary components of codimension one, denoted as in \cite{Mazzeo-Melrose:Meromorphic} by
\begin{equation*}
\p_1^l( \BB^{n+1} \times \BB^{n+1})= \p \BB^{n+1} \times \BB^{n+1} \text{ and }
\p_1^r( \BB^{n+1} \times \BB^{n+1})=  \BB^{n+1} \times \p \BB^{n+1},
\end{equation*}
which have a common boundary
 $\p_2( \BB^{n+1}\times \BB^{n+1})=\p \BB^{n+1} \times \p \BB^{n+1}.$ The lift of
$\p_1^r( \BB^{n+1} \times \BB^{n+1})$ to $M_0$, which is the closure of
$$
\beta^{-1}(\p_1^r( \BB^{n+1} \times \BB^{n+1})\setminus \p_2( \BB^{n+1}\times \BB^{n+1})),
$$
with
$\beta:M_0\to\BB^{n+1}\times\BB^{n+1}$ the blow-down map, will be called the right face and denoted by $\mcr.$ Similarly,
the lift of $\p_1^l( \BB^{n+1} \times \BB^{n+1})$ will be called the left face and denoted by $\mcl.$   The lift of $\Diag \cap \p_2( \BB^{n+1}\times \BB^{n+1})$, which is its
inverse image under $\beta$, will be called the front face $\mcf,$ see figure \ref{fig1}.

  We briefly recall the definitions of their classes of pseudodifferential operators, and refer the reader to
\cite{Mazzeo-Melrose:Meromorphic} for full details. First they define the class $\Psi_0^m(\BB^{n+1})$ which consists of those pseudodifferential operators of order 
$m$ whose Schwartz kernels lift under the blow-down map $\beta$ defined in \eqref{zero-blow-up} to a distribution which is conormal (of order $m$) to the lifted diagonal and  vanish to infinite order at all faces, with the exception of the front face, up to
which it is $\CI$ (with values in conormal distributions).
Here, and elsewhere in the paper, we trivialized the right density bundle using
a zero-density; we conveniently fix this as
$|dg_\delta(z')|$. Thus, the Schwartz kernel of $A\in\Psi_0^m(\Bn)$ is
$K_A(z,z') |dg_\delta(z')|$, with $K_A$ as described above, so in particular is $\CI$
up to the front face.

  It then becomes necessary to introduce another class of operators whose kernels are singular at the right and left faces.  This class will be denoted by $\Psi_0^{m,a,b}(\BB^{n+1}),$ $a,b \in \CC.$  An operator $P\in \Psi_0^{m,a,b}(\BB^{n+1})$ if it can be written as a sum $P=P_1+P_2,$ where $P_1\in \Psi_0^m(X)$ and the Schwartz kernel $K_{P_2}|dg_\delta(z')|$
of the operator $P_2$ is such that $K_{P_2}$
lifts under $\beta$ to a conormal distribution which is smooth up to the front face, and which satisfies the following conormal regularity with respect to the other faces
\begin{equation}
\mcv_b^{k}  \beta^* K_{P_2} \in \rho_L^a \rho_R^b L^\infty(\BB^{n+1} \times_0 \BB^{n+1}), \;\ \forall \;\ k \in \mn, \label{boundary-regularity}
\end{equation}
where $\mcv_b$ denotes the space of vector fields on $M_0$ which are tangent to the right and left faces.

 Next we define the semiclassical blow-up of
$$
\Bn\times\Bn\times[0,1)_h,
$$
and  the corresponding classes of pseudodifferential operators associated with it that will be used in the construction of the parametrix.  The semiclassical 
double space
is constructed in two steps. First, as in \cite{Mazzeo-Melrose:Meromorphic}, we blow-up the intersection of the diagonal $\Diag \times [0,1)$ with $\p \BB^{n+1} \times \p \BB^{n+1} \times [0,1).$  Then we blow-up the intersection of the lifted diagonal times $[0,1)$ with $\{h=0\}.$  We define the manifold with corners
\begin{equation}
M_{0,\semi}=[\BB^{n+1}\times \BB^{n+1} \times [0,1]_{h }; \pa \Diag\times[0,1);\Diag_0\times\{0\}]. \label{scblowup1}
\end{equation}
See Figures~\ref{figscblow-up1} and \ref{figscblow-up2}.  We will denote the blow-down map
\begin{equation}
\beta_{\semi}: M_{0,\semi} \longrightarrow \BB^{n+1} \times \BB^{n+1} \times [0,1). \label{betasc}
\end{equation}
\begin{figure}[int1]
\epsfxsize= 3.5in
\centerline{\epsffile{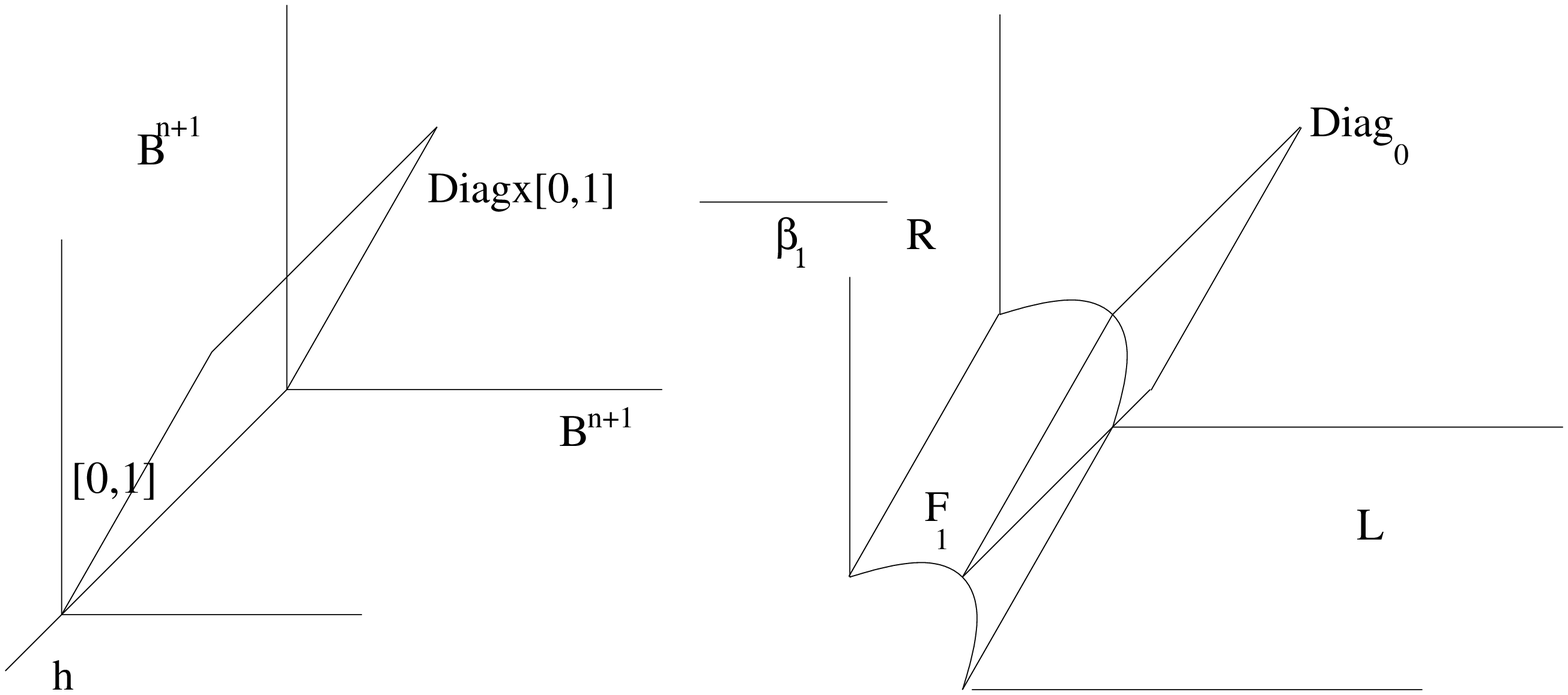}}
\caption{The stretched product $\BB^{n+1}\times_0 \BB^{n+1}\times [0,1).$}
\label{figscblow-up1}
\end{figure}

\begin{figure}[int1]
\epsfxsize= 3.5in
\centerline{\epsffile{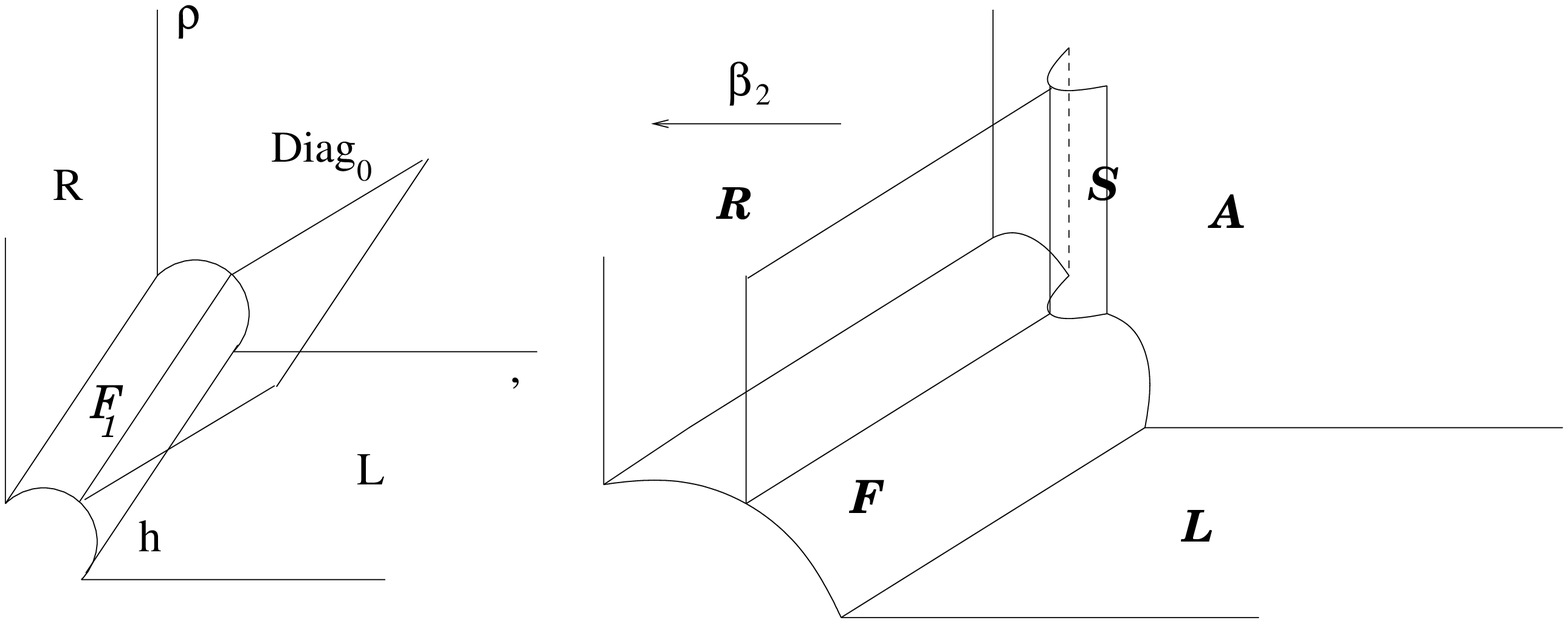}}
\caption{The semiclassical blown-up space $M_{0,h}$ obtained by blowing-up $\BB^{n+1}\times_0 \BB^{n+1} \times [0,1)$  along  $\Diag_0\cap \times \{h=0\}.$ }
\label{figscblow-up2}
\end{figure}

As above, we can define the right and left semiclassical faces as  the  lift of 
\begin{gather*}
\p_1^l\left( \BB^{n+1} \times \BB^{n+1} \times [0,1)\right)= \p \BB^{n+1} \times \BB^{n+1}  \times [0,1)\text{ and } \\ 
\p_1^r\left( \BB^{n+1} \times \BB^{n+1} \times [0,1))\right)=  \BB^{n+1} \times \p \BB^{n+1} \times [0,1),
\end{gather*}
by the blow-down map $\beta_{\semi}.$  We will denote the lift of the diagonal under the map $\beta_{\semi}$ by $\Diag_{\semi},$ i.e.
\begin{gather*}
\Diag_{\semi}= \text { the closure of }  \beta_{\semi}^{-1}\left( \Diag \times (0,1)\setminus (\Diag \cap (\p \Bn \times \p \Bn))\right).
\end{gather*}

The lift of $\pa \Diag \times [0,1)$ will be called the zero front face $\mcf,$ while the boundary obtained by the blow-up of $\Diag_0 \times [0,1]$ along $\Diag_0\times \{0\}$ will be called the semiclassical front face $\mcs.$  The face which is obtained by the lift of $\Bn \times \Bn \times \{0\}$ is the semiclassical face, and will be denoted by $\mca.$

We wish to find a parametrix such that $P(h,\sigma)$ acting on the left
produces the identity plus an error which  vanishes to high enough order on the right and left
faces, $\mcr$ and $\mcl,$ and to infinite order at the zero-front face $\mcf,$  the semiclassical front face $\mcs$ and the semiclassical face $\mca.$  Thus 
the error term is bounded as an operator acting between weighted $L^2(\Bn)$ spaces and its norm goes to zero as $h\downarrow 0.$

As in \cite{Mazzeo-Melrose:Meromorphic} we define the class of semiclassical pseudodifferential operators in two steps.  First we define the space
$P\in\Psi_{0,\semi}^m(\Bn)$ which consists of operators whose kernel
$$
K_P(z,z',h)\,|dg_\delta(z')|
$$
lifts to a conormal distribution of order $m$ to the lifted diagonal and
vanishes to infinite order at all faces, except the zero front face, up to which it is
$\CI$ (with values in conormal distributions)
and the semiclassical front face, up to which it is $h^{-n-1}\CI$ (with values in
conormal distributions). We then define the space
\begin{gather}
\mck^{ a,b,c}(M_{0,\semi})=\{ K \in L^{\infty}(M_{0,\semi}): \mcv_b^m K \in \rho_L^a \rho_A^b \rho_R^c \rho_S^{-n-1}L^\infty(M_{0,\semi}), \;\ m\in \NN \}, \label{sc-con}
\end{gather} 
where $\mcv_b$ denotes the Lie algebra of vector fields which are tangent to $\mcl,$ $\mca$ and $\mcr.$  Again, as in \cite{Mazzeo-Melrose:Meromorphic}, we define the space $\Psi_{0,\semi}^{m,a,b,c}(\Bn)$ as the operators $P$ which can be expressed in the form $P=P_1+P_2,$ with $P_1\in \Psi_{0,\semi}^{m}(\Bn)$ and  the kernel 
$K_{P_2}\,|dg_\delta(z')|$ of $P_2$ is such
$\beta_{\semi}^* K_{P_2} \in \mck^{ a,b,c}(M_{0,\semi}).$

\section{A semiclassical parametrix for the resolvent in dimension three}\label{3D-parametrix}

In this section we construct a parametrix for the  resolvent  $R(h,\sigma)$ defined in \eqref{scresolvent} in dimension three.  We do this case separately because it   
 is much simpler than in the general case, and one does not have to perform the semiclassical blow-up.  Besides,
we will also need part of this construction in the general case. More precisely, if $n+1=3,$  and the metric $g$ satisfies the hypotheses of Proposition \ref{distance}, we will use Hadamard's method to construct the leading asymptotic term of the parametrix
$G(\sigma,h)$ at the diagonal, and the top two terms of the semiclassical
asymptotics.  Our construction takes place on
$(\BB^3\times_0\BB^3)\times [0,1)_h$,  instead of its semiclassical
blow up, i.e.\ the blow up of the zero-diagonal at $h=0$, as described above.  This is made possible by a coincidence, namely
that in three dimensions, apart from an explicit exponential factor, the
leading term in the asymptotics lives on $\BB^3\times_0 \BB^3 \times [0,1)$. However, to obtain further
terms in the asymptotics would require the semiclassical blow up, as
it will be working in higher dimensions. For example, this method only give bounds for the resolvent of  width 1, while the more general construction gives the bounds on any strip.

We recall that in three dimensions the resolvent of the Laplacian in hyperbolic space, $\Delta_{g_0},$ 
$R_0(\sigma)=(\Delta_{g_0}-\sigma^2-1)^{-1}$
has a holomorphic continuation to $\mc$ as an operator from functions vanishing to infinite order at
$\p \Bc$ to distributions in $\Bc.$  
 The Schwartz kernel of $R_0(\sigma)$ is given by
\begin{gather}
R_0(\sigma,z,z')= \frac{e^{-i\sigma r_0}}{4\pi \sinh r_0}, \label{resol0}
 \end{gather}
 where  $r_0=r_0(z,z')$ is the geodesic distance between $z$ and $z'$ with respect to the metric
 $g_{0},$ see for example \cite{Mazzeo-Melrose:Meromorphic}.

 Since there are no conjugate points for the geodesic flow of $g,$  for each 
$z'\in \Bo,$ the exponential map for the metric $g,$
$\exp_{g}: T_{z'} \Bo \longrightarrow \Bo,$  is a  global diffeomorphism.  Let 
$(r,\theta)$ be geodesic normal coordinates for $g$ which are valid in $\Bo\setminus \{z'\};$ 
$r(z,z')=:d(z,z')$ is the distance function for the metric $g.$ Since $r(z,z')$ is globally defined, $g$ is a small perturbation of $g_0$ and the kernel of $R_0(\sigma)$ is given by \eqref{resol0}, it is reasonable to seek a parametrix of  $R(h,\sigma)$ which has  kernel of the form
\begin{gather}
G(h,\sigma,z,z')=e^{-i \frac{\sigma}{h} r}  h^{-2}U(h,\sigma,z,z'), \label{parametrix}
\end{gather}
 with  $U$  properly chosen.    
 
 We now reinterpret this as a semiclassical Lagrangian distribution to relate it to the results
 of Section~\ref{distance-function}. Thus, $-\sigma r=-\sigma d(z,z')$ is the phase function for semiclassical distributions corresponding
 to the backward left flow-out of the conormal bundle of the diagonal inside
the characteristic set of $2p_\ep-\sigma^2$. This flowout is the same as
the forward right flow-out of the conormal bundle of the diagonal, and is
also the dilated version, by a factor of $\sigma$, in the fibers of the
cotangent bundle, of the flow-out of
in the characteristic set of $2p_\ep-1$, which we described in detail
in Section~\ref{distance-function}.

In view of the results of  Section~\ref{distance-function},
for the characteristic set of $2p_\ep-1$ (the general case of
$2p_\ep-\sigma^2$ simply gives an overall additional factor of
$\sigma$ due to the dilation), the
 lift $\beta^*\pa_r$ of $\pa_r$ to $\Bc\times_0\Bc$ satisfies
 \begin{equation}\label{eq:pa_r-V_eps}
 \beta^*\pa_r=\widetilde{\Pi_\eps}_*V_\eps^L,
 \end{equation}
 and thus is a $\CI$ vector field on
 $(\Bc\times_0\Bc)\setminus\Diag_0$
 which is tangent to all boundary faces, and at the left face
 $L=B$,
 \begin{equation}\label{eq:pa_r-at-left-face}
 \pa_r=-\mathsf{R}_L+W_L,
 \end{equation}
 where
 $\mathsf{R}_L$ is the radial vector field corresponding to the left face, and
 $W_L\in\rho_L\Vb((\Bc\times_0\Bc)\setminus\Diag_0)$.
 
 It is convenient to
 blow up $\Diag_0$ and lift $\beta^*\pa_r$ further to this space.
 We thus define  
 $\Bc \times_1 \Bc$ to be the manifold obtained from $\Bc \times_0 \Bc$ by blowing-up the diagonal $\Diag_0$ as shown in Fig. \ref{fig2}. Let $\beta_d : \Bc\times_1 \Bc\longmapsto \Bc \times_0 \Bc$ denote the blow-down map and let $\beta_d\circ \beta=\beta_{0d}.$
 The vector field $\beta_{0d}^* \p_r$ is transversal to the
new face introduced by blowing up the diagonal; it still satisfies the lifted analogue of
\eqref{eq:pa_r-at-left-face}. Moreover, integral curves of $\beta_{0d}^* \p_r$ hit
the left face $L$ away from its intersection with the right face $R$ in finite time.

 \begin{figure}[int1]
\epsfxsize= 2.0in
\centerline{\epsffile{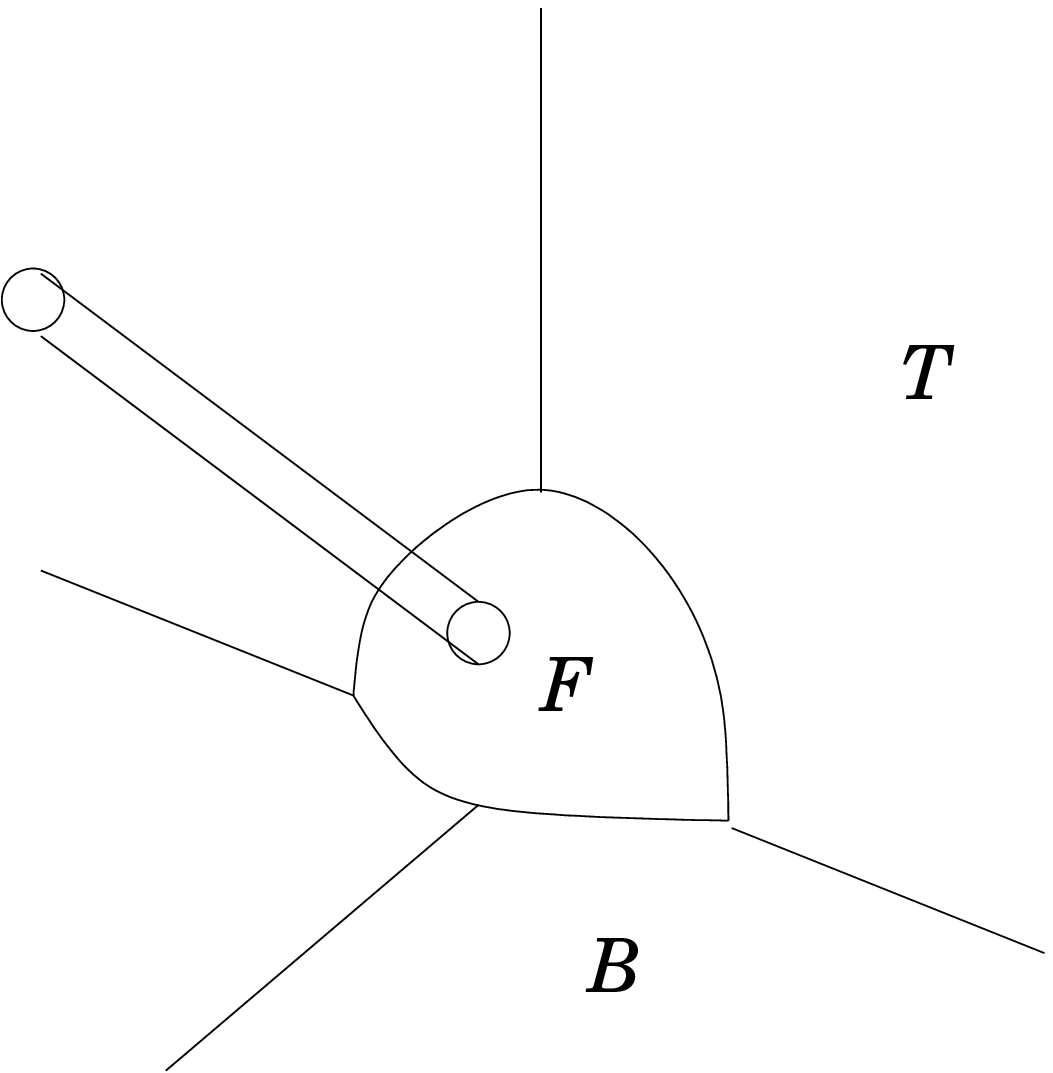}}
\caption{The manifold $\Bn \times_1 \Bn.$}
\label{fig2}
\end{figure}

In coordinates $(r,\theta)$ the metric $g$ is given by
\begin{gather}
g = dr^2 + H(r, \theta, d\theta),\label{gnorm}
\end{gather}
where $H(r,\theta,d\theta)$ is a $\CI$ 1-parameter family of  metrics on $\ms^2.$
The Laplacian with respect to $g$  in these coordinates  is given by
\begin{gather*}
\Delta_g= -\p_r^2 -V \p_r  + \Delta_H, \;\  V= \frac{1}{|g|^\ha} \p_r (|g|^\ha), 
\end{gather*}
where $|g|^\ha$ is the volume element of the metric $g$ and $\Delta_H$ is the Laplacian 
with respect to $H$ on $\ms^2.$  The volume element $|g|^\ha$ has the following expansion as $r\downarrow 0,$
\begin{gather}
|g|^{\ha}(r,\theta)= r^{2}(1+ r^2g_1(r,\theta)), \label{volume1}
\end{gather}
see for example page 144 of \cite{Gallot-Hulin-Lafontaine}. So
\begin{gather}
\Delta_g =-\p_r^2 -(\frac{2}{r}+rA) \p_r +\Delta_H \label{volume3}
\end{gather}

We want $U$  in \eqref{parametrix} to be of the form
\begin{gather}
U(h,\sigma,z,z')=U_0(\sigma,z,z')+ h U_1(\sigma,z,z'), \label{parametrix01}
\end{gather}
 and so
\begin{equation}\begin{split}
&\left(h^2(\Delta_g+x^2W-1)-\sigma^2\right) 
e^{-i\frac{\sigma}{h} r} h^{-2} U \\
&=e^{-i\frac{\sigma}{h}r}\Big( (\Delta_g+x^2 W-1) U_0+ 2i \frac{\sigma}{h} |g|^{-\oq} \p_r(|g|^\oq U_0)\\
&\qquad\qquad+
2i\sigma   |g|^{-\oq} \p_r(|g|^\oq U_1) + h (\Delta_g+x^2 W-1)U_1\Big).
\end{split}\label{parametrix1}
\end{equation}

Here the leading term in $h$ as $h\to 0$ both overall, and as far as $U_0$ is concerned, is
\begin{equation}\label{eq:0th-transport}
2i \frac{\sigma}{h} |g|^{-\oq} \p_r(|g|^\oq U_0),
\end{equation}
and the leading term as far as $U_1$ is concerned is
\begin{equation}\label{eq:1st-transport}
2i\sigma   |g|^{-\oq} \p_r(|g|^\oq U_1).
\end{equation}
In the interpretation as semiclassical Lagrangian distributions,
these are both the differential operators arising in the transport equations. For Hamiltonians
given by Riemannian metrics, these operators in the interior are well-known to be
Lie derivatives with respect to the Hamilton vector field, when interpreted as acting on half-densities.
This can also be read off directly from \eqref{eq:0th-transport}-\eqref{eq:1st-transport}, with
$|g|^\oq$ being the half-density conversion factor, $\sigma$ is due to working at energy
$\sigma$ (rather than $1$), and the factor of $2$ is due to the symbol of the Laplacian
being $2p_\eps$.

To get rid of the term in $h^{-1}$ we solve the 0th transport equation, i.e.\ we impose
\begin{gather*}
\p_r(|g|^\oq U_0)=0,
\end{gather*}
and we  choose $U_0(r,\theta)= \frac{1}{4\pi} |g(r,\theta)|^{-\oq}.$  From \eqref{volume1} we have 
\begin{gather}
|g|^{-\oq}(r,\theta)=r^{-1}(1+ r^2 g_2(r,\theta)) \text{ near } r=0. \label{expg}
\end{gather}

 Therefore, near $r=0,$ 
\begin{gather}
\Delta_g \frac{1}{4\pi} |g(r,\theta)|^{-\oq}= \delta(z,z')+ \frac{1}{4\pi}A r^{-1} +\frac{1}{4\pi}\Delta_g (r g_2).
\label{expato}
\end{gather}
This only occurs in three dimensions, and makes this construction  easier than in the general case.  In higher dimensions the power or $r$ in \eqref{expg} is
$r^{-\frac{n}{2}},$ and  does not coincide with the power of $r$ of the fundamental solution of the Laplacian, which, in dimension $n+1,$ is $r^{1-n},$ so one does not get the delta function in \eqref{expato}.

To get rid of the term independent of $h$ in \eqref{parametrix1}  in $r>0$ we solve the
first transport equation,
\begin{gather*}
2i \sigma |g|^{-\oq}\p_r(|g|^\oq U_1)+ (\Delta_g+x^2W-1)U_0=0 \text{ in } r>0,\\
U_1=0 \text{ at } r=0.
\end{gather*}

So
\begin{gather}
U_1(r,\theta)= -\frac{1}{8i\sigma \pi}|g(r,\theta)|^{-\oq} \int_0^r |g|^\oq(s,\theta) \left(\Delta_g+x^2W-1\right)|g|^{-\oq}(s,\theta) \; ds.\label{formu1}
\end{gather}
Since $|g|^\oq$ is $\CI$ up to $r=0,$ and vanishes at $r=0,$ it follows from \eqref{expato} that
 $|g|^\oq \Delta_g |g|^{-\oq}$ is $\CI$ up to $r=0.$ In particular the integrand in \eqref{formu1} is smooth up to $r=0.$ With these choices of $U_0$ and $U_1$ we obtain
\begin{gather}
\begin{gathered}
\left(h^2(\Delta_g+x^2W-1)-\sigma^2\right)  e^{-i\frac{\sigma}{h} r} h^{-2}U(h,\sigma,z,z')\\
= \delta(z,z')+
h e^{-i\frac{\sigma}{h} r}  (\Delta_g+x^2W-1) U_1(\sigma,z,z') 
 \end{gathered}\label{term1}
 \end{gather}

This gives, in principle, a parametrix  $G(h,\sigma,z,z')=e^{-i\frac{\sigma}{h} r} h^{-2}U(h,\sigma,z,z')$
in the interior of
$\Bc\times \Bc$ in the two senses that the diagonal singularity of $R(h,\sigma)$
 is solved away to leading order, which in view of the ellipticity of the operator means that the error
$$
E(h,\sigma)=\left( h^2(\Delta_g+x^2W-1)-\sigma^2\right) G(h,\sigma)-\Id
$$
is a semiclassical pseudodifferential operator of order $-1$ (in a large
calculus, i.e.\ with non-infinite order vanishing off the semiclassical
front face, which did not even appear in our calculations), and the
top two terms of the
semiclassical parametrix of $\left(h^2(\Delta_g+x^2W-1)-\sigma^2\right)^{-1}$ are also found.

In fact, our parametrix is better than this. To understand the behavior of  $G$ and the remainder 
\begin{gather}
E(h,\sigma,z,z')=h e^{-i\frac{\sigma}{h} r}(\Delta_g +x^2W-1)U_1 \label{defE}
\end{gather}
near the boundary of $\Bc\times \Bc,$  we need to analyze the behavior of $U_0$
and $U_1$ at the left and right boundary faces.
 We will do the computations for arbitrary dimensions, since we will need some of these estimates in the general case, but in this special situation we have $n=2.$

We start by noting that the asymptotics of $U_0$ and $U_1$ follow from the transport equation
which they satisfy. Indeed, much like we analyzed the flow-out of the conormal bundle of the
diagonal, we show now that
\begin{equation}\label{eq:U_0-U_1-asymp}
U_0\in \rho_D^{-1}\rho_L^{n/2}\rho_R^{n/2}\CI(\Bn\times_1\Bn),\ U_1\in R^2\rho_L^{n/2}\rho_R^{n/2}\CI(\Bn\times_1\Bn);
\end{equation}
here $\rho_D$ is the defining function of the front face of the blow-up creating $\Bn\times_1\Bn$.
Note that we have already shown this claim near this front face; the main content of the
statement is the precise behavior as $\rho_L,\rho_R\to 0$.

We start with $U_0$.
First, the conclusion away from the right face, $\rho_R=0$, follows immediately from
\eqref{eq:pa_r-at-left-face} since integral curves emanating from the lifted diagonal
hit this region at finite time, and solutions of the Lie derivative equation have this form
near the boundary. To have the analogous conclusion away from the left face, we
remark that solutions of the left transport equation automatically solve the right transport
equation; one can then argue by symmetry, or note
directly that as $-\pa_{r'}$ is the radial vector field at the right face, modulo an element
of $\rho_R\Vb(\Bn\times_0\Bn)$, and as integral curves of $\pa_{r'}$
emanating from the lifted diagonal
hit this region at finite time, and solutions of the Lie derivative equation have this form
near the boundary. It remains to treat the corner where both $\rho_L=0$ and $\rho_R=0$.
The conclusion here now follows immediately as integral curves of $\pa_r$ reach
this corner in finite time from a punctured neighborhood of this corner, and in this
punctured neighborhood we already have the desired regularity. This proves
\eqref{eq:U_0-U_1-asymp} for $U_0$.

To treat $U_1$, it suffices to prove that
\begin{equation}\label{eq:U_0-error-asymp}
(\Delta_g+x^2W-1)U_0\in R^2\rho_L^{n/2+2}\rho_R^{n/2}\CI(\Bn\times_0\Bn\setminus \Diag_0),
\end{equation}
for then
\begin{equation}\label{eq:U_1-asymp}
U_1\in R^2\rho_L^{n/2}\rho_R^{n/2}\CI(\Bn\times_0\Bn\setminus \Diag_0),
\end{equation}
by the same arguments as those giving the asymptotics of $U_0$, but now applied to
the inhomogeneous transport equation.

On the other hand, \eqref{eq:U_0-error-asymp} follows from
$$
\beta^*\pi_L^*(\Delta_g+x^2W-1)\in \Diffb^2(\Bn\times_0\Bn)
$$
with
$$
\beta^*\pi_L^*\Big((\Delta_g+x^2W-1)-(\Delta_{g_0}-1)\Big)\in R^2\rho_L^2\Diffb^2(\Bn\times_0\Bn);
$$
here $\pi_L$ is added to emphasize the lift is that of the differential operator acting on the
left factor; lifting the operator on the right factor results in an `error'
$R^2\rho_R^2\Diffb^2(\Bn\times_0\Bn)$. These two in turn follow immediately from
the form of the metric, namely $g_0-g_\ep\in x^2\CI(X;\zT^*X\otimes \zT^*X)$.

This completes the proof of \eqref{eq:U_0-U_1-asymp}, and also yields that, with
$n+1=3$,
\begin{gather}
\beta^*\left((\Delta_g+x^2W-1) U_1 \right)\in R^2\rho_L^{n/2+2}\rho_R^{n/2}\CI(\Bn\times_0\Bn\setminus \Diag_0).\label{boundrest}
\end{gather}

Therefore, in the case $n+1=3,$ we have proved the following
\begin{thm}\label{3Dparametrix-thm} There exists a pseudodifferential operator, $G(h,\sigma),$ $\sigma\not=0,$ whose kernel is of the form
\begin{gather*}
G(h,\sigma,z,z') = e^{-i\frac{\sigma}{h}}h^{-2}\left(U_0(h,\sigma,z,z')+ h U_1(h,\sigma,z,z') \right)
\end{gather*}
with $U_0$ and $U_1$ satisfying 
\eqref{eq:U_0-U_1-asymp} and such that the error
$E(h,\sigma)= P(h,\sigma)G(h,\sigma)-\Id$ is given by \eqref{defE} and satisfies
\eqref{boundrest}.
\end{thm}

\section{The structure of the semiclassical resolvent }\label{resolvent-structure}\label{full-sc-parametrix}

In this section  we construct the general right semiclassical parametrix $G(h,\sigma)$ for the resolvent.  We will prove the following
\begin{thm} \label{nD-parametrix} There exists a pseudodifferential operator $G(h,\sigma),$ $\sigma\not=0,$ such that  its kernel is of the form
 \begin{gather}
 G(h,\sigma,z,z')=e^{-i \frac{\sigma r}{h}} U(h,\sigma,z,z'),
\ U\in\Psi_{0,\semi}^{-2,\frac{n}{2},-\frac{n}{2}-1,\frac{n}{2}}(\Bn). \label{kernelofg}
 \end{gather}
 and  that, using the notation of section  \ref{semiclassical-double-space},
 \begin{gather*}
 P(h,\sigma) G(h,\sigma)-\Id \in \rho_{\mcf}^\infty \rho_{\mcs}^\infty \Psi_{0,\semi}^{-\infty, \infty,\infty,\novt +i\sigh}(\Bn).
 \end{gather*}
 \end{thm}
Now, $r^2$ is a $\CI$ function on the zero double space away from the
left and right faces and in a quadratic
sense it defines the diagonal non-degenerately; $r$ has
an additional singularity at the zero-diagonal. Correspondingly
$\left(\frac{\sigma r}{h}\right)^2$ is $\CI$ on $M_{0,\semi}$ away from $\mcl$, $\mcr$
and $\mca$ and
defines the lifted diagonal non-degenerately in a quadratic sense;
$\frac{\sigma r}{h}$ has an additional singularity at the zero diagonal. In particular,
$e^{-i \frac{\sigma r}{h}}$ is $\CI$ on $M_{0,\semi}$ away from $\mcl$, $\mcr$,
$\mca$ and the lifted diagonal, $\Diag_{\semi}$; at $\Diag_{\semi}$
it has the form of $1$ plus
a continuous conormal function vanishing there. Thus, its presence in any
compact subset of $M_{0,\semi}\setminus(\mcl\cup\mcr\cup\mca)$ is not
only artificial, but introduces an irrelevant singularity at $\Diag_{\semi}$, so
it is better to think of $G$ as
 \begin{equation*}\begin{split}
 G(h,\sigma,z,z')&=G'+e^{-i \frac{\sigma r}{h}} U'(h,\sigma,z,z'),\\
&G'\in\Psi_{0,\semi}^{-2},\ U'\in\Psi_{0,\semi}^{-\infty,\frac{n}{2},-\frac{n}{2}-1,\frac{n}{2}}, 
 \end{split}\end{equation*}
 where $G'$ is supported near $\Diag_{\semi}$ (i.e.\ its support intersects only the
boundary hypersurfaces $\mcs$ and $\mcf$, but not the other boundary hypersurfaces),
and $U'$ vanishes near $\Diag_{\semi}$.

Indeed,
the first step in the construction of $G$ is to construct a piece of $G'$, namely to
find $G_0\in\Psi_{0,\semi}^{-2}$ such that $E_0=P(h,\sigma)G_0(h,\sigma)-\Id$
has no
singularity at the lifted diagonal, $\Diag_{\semi},$ with $G_0$ (hence $E_0$)
supported in a neighborhood of $\Diag_{\semi}$ in $M_{0,\semi}$ that only intersects the boundary of $M_{0,\semi}$ at $\mcs$ and $\mcf$. Thus,
\begin{equation}
 P(h,\sigma)G_0-\Id= E_0, \;\ \text{ with } E_0 \in \Psi_{0,\semi}^{-\infty}(\Bn). \label{consg0}
\end{equation}

Since $P(h,\sigma)$ is elliptic in the interior of $\Bn,$ the construction of $G_0$ in the interior follows from the standard Hadamard parametrix construction.  We want to do this construction uniformly up to the zero front face and the semiclassical front face in the blown-up manifold $M_{0,\semi}.$ We notice that the lifted diagonal intersects the boundary of $M_{0,\semi}$ transversally at the zero and semiclassical front faces.   Since $\Diag_{\semi}$ intersects the boundary of $M_{0,\semi}$ transversally at $\mcs$ and $\mcf,$ we proceed as in
\cite{Mazzeo-Melrose:Meromorphic}, and extend $\Diag_{\semi}$ across
the boundary of $M_{0,\semi},$ and want to extend the Hadamard parametrix construction across the boundary as well.  To do that we have to make sure the operator $P(h,\sigma)$ lifts to be uniformly transversally
elliptic at the lift of $\Diag_{\semi}$ up to the boundary of $M_{0,\semi},$ i.e.\ it is
elliptic on the conormal bundle of this lift, and thus
it can be extended as a transversally (to the extension of the lifted diagonal)
elliptic operator across the boundaries of $M_{0,\semi}.$ Note that up to
$\mcs\setminus\mcf$, this is the standard semiclassical elliptic
parametrix construction,
while up to $\mcf\setminus\mcs$, this is the first step in the conformally compact
elliptic
parametrix construction of Mazzeo and the first author \cite{Mazzeo-Melrose:Meromorphic},
so the claim here is that these constructions are compatible with each other and extend
smoothly to the corner $\mcs\cap\mcf$ near $\Diag_{\semi}$.

To see the claimed ellipticity, and facilitate further calculations, we remark
that one can choose a defining function of the boundary $x$ such that the metric $g$ can be written in the form 
\begin{equation*}
g=\frac{dx^2}{x^2} + \frac{H_\eps(x,\omega)}{x^2},
\end{equation*}
where $H_\eps$ is one-parameter family of $\CI$ metrics on $\ms^n.$  In these coordinates the operator $P(h,\sigma)$ is given by
\begin{equation}
P(h,\sigma)= h^2\left( -(x\p_x)^2+ n x\p_n + x^2 A(x,\omega) \p_x+ x^2 \Delta_{H_\eps(x,\omega)}+ x^2 W -\frac{n^2}{4}\right) -\sigma^2. \label{formulaofp}
\end{equation}
We then conjugate $P(h,\sigma)$ by $x^{\frac{n}{2}},$ and we obtain
\begin{equation}
Q(h,\sigma)= x^{-\frac{n}{2}} P(h,\sigma) x^{\frac{n}{2}}= h^2\left( -(x\p_x)^2 + x^2 A \p_x + x^2 \Delta_{H_\eps} + x^2 B\right) - \sigma^2,
\end{equation}
where $B=-\frac{n}{2}A+W.$ To analyze the lift of $Q(h,\sigma)$ under $\beta_{\semi}$  we work in projective coordinates for the blow-down map.   We denote the coordinates on the left factor of $\BB^{n+1}$ by $(x,\omega),$ while the coordinates on the right factor will be denoted by $(x',\omega').$ Then  
we define projective coordinates
 \begin{equation}
 x'=\rho, \;\  X=\frac{x}{x'}, \;\ Y=\frac{\omega-\omega'}{x'}, \label{projectivecoord}
 \end{equation}
 which hold away from the left face.  The front face is given by $\mcf=\{\rho=0\}$ and the 
lift of the diagonal is  $\Diag_0=\{X=1, Y=0\}.$  The lift of $Q(h,\sigma)$ under the zero blow-down map $\beta$ is equal to
 \begin{gather*}
 Q_0(h,\sigma)=\beta^*Q(h,\sigma)= \\ h^2\left( -(X\p_X)^2 + X^2A \rho \p_X+ X^2 \Delta_{H_\eps(\rho X,\omega'+\rho Y)}(D_Y)
 -X^2\rho^2B \right) -\sigma^2.
 \end{gather*}
In this notation,  the coefficients of $\Delta_{H_\eps(\rho X,\omega'+\rho Y)}(D_Y)$ depend on 
$\rho,\omega'$ and $Y,$ but the derivatives are in $Y.$   This operator is transversally
elliptic in a neighborhood of $\{X=1,\ Y=0\}$, away from $h=0$.

The restriction of the lift of $Q_0(h,\sigma)$ to the front face  $\mcf=\{\rho=0\},$ is given by
\begin{gather*}
N_\mcf (Q_0(h,\sigma))=h^2( -(X\p_X)^2 + X^2\Delta_{H_\eps(0,\omega')}(D_Y))-\sigma^2.
\end{gather*}
As in \cite{Mazzeo-Melrose:Meromorphic},  $N_{\mcf}(Q_0(h,\sigma))$
will be called the normal operator  of $Q_0(h,\sigma)$ at the zero front face $\mcf.$ Notice that it can be identified with Laplacian with respect to the hyperbolic metric on the half-plane $\{X>0, Y\in \RR^n\}$ with metric $X^{-2}(dX^2+ H_\eps(0,\omega'))$ conjugated by $X^{\frac{n}{2}}.$

 Now we blow-up the intersection of $\Diag_0\times [0,h)$ with $h=0.$  We define projective coordinates
 \begin{equation}\label{eq:X-semi-def}
 X_{\hbar}= \frac{X-1}{h}, \;\ \;\ Y_{\hbar}=\frac{Y}{h}.
 \end{equation}
The lift of $Q(h,\sigma)$ under the semiclassical blow-down map $\beta_{\semi}$ is given in these coordinates by
\begin{equation*}\begin{split}
Q_{\semi}&=\beta_{\semi}^*Q(h,\sigma)\\
&= -((1+h X_{\hbar}) \p_{X_{\hbar}})^2 + (1+h X_{\hbar})^2 \Delta_{H_\eps}(D_{Y_{\hbar}})\\
&\qquad\qquad- (1+hX_{\hbar})^2 A \rho h \p_{X_{\hbar}}+ h^2\rho^2 (1+h X_{\hbar})^2 B -\sigma^2,
\end{split}\end{equation*}
where $H_\eps=H_\eps(\rho(1+h X_{\hbar}), \omega'+\rho h Y_{\hbar}).$  This operator is transversally elliptic to $\{X_{\hbar}=0,\ Y_{\hbar=0}\}$ near $h=0.$

The restriction of the lift of $Q_{\semi}$ to the semiclassical face, $\mcs=\{h=0\}$ will be called the  normal operator of $Q_{\semi}$ at the semiclassical face, and  is equal to
\begin{equation}
N_\mcs(Q_{\semi})=-\p_{X_{\hbar}}^2 +\Delta_{H_\eps(\rho,\omega')}(D_{Y_{\hbar}})-\sigma^2. \label{norm-op-semic-face}
\end{equation}
This is a  family of differential operators on $\RR^{n+1}_{X_{\hbar},Y_{\hbar}}$
depending on $\rho$ and $\omega',$ and for each $\omega'$ and $\rho$ fixed, $N_\mcs(Q_{\semi})+\sigma^2$
is the Laplacian with respect to the metric
$$
\delta_{\semi}=dX_{\hbar}^2+ \sum H_{\eps,ij}(\rho,\omega) dY_{\hbar,i} dY_{\hbar,j},
$$
which is isometric to the Euclidean metric under a linear change of variables
$(X_{\hbar},Y_{\hbar})$ for fixed $(\rho,\omega')$, and the change of variables can
be done smoothly in $(\rho,\omega')$. Note that the fibers of the semiclassical
blow-down map $\beta_{\semi}$ on $\cS$ are given exactly by $(\rho,\omega')$
fixed.

Therefore, the operator $Q(h,\sigma)$ lifts under $\beta_{\semi}$ to an operator $Q_{\semi}$ which is elliptic in a neighborhood of the lifted diagonal uniformly  up to the zero front face and the semiclassical front face.   Since the diagonal meets the two faces transversally, one can extend it  to a neighborhood of $\mcf$ and $\mcs$ in the double of the manifold $M_{0,\semi}$ across $\mcs$ and $\mcf,$ and one can also extend the operator $Q_{\semi}$ to be elliptic in that neighborhood.  Now,  using standard elliptic theory
(or, put somewhat differently, the standard theory of conormal distributions to an
embedded submanifold without boundary, in this case the extension of
the diagonal), one can find a $G_0\in \Psi_{0,\semi}^{-2}(\Bn)$ whose Schwartz kernel
lifts to a distribution  supported in a neighborhood of $\Diag_{\semi}$  such that
\begin{equation}
Q(h,\sigma) G_0-\Id= E_0 \in \Psi_{0,\semi}^{-\infty}(\Bn). \label{term0}
\end{equation}

Next we  will  remove the error at the semiclassical front face.  We will find an operator 
 $$
G_1= G'_1+e^{-i\frac{\sigma}{h}r }U'_1,
\ G'_1\in\Psi_{0,\semi}^{-\infty}(\Bn),
\ U'_1\in \Psi_{0,\semi}^{-\infty,\infty,-\frac{n}{2}-1,\infty}(\Bn)
$$
such that
$G'_1$ is supported near $\Diag_{\semi}$ while $U'_1$ is supported away from it, and
 \begin{gather}
  \begin{gathered}
 Q(h,\sigma)G_1-E_0=E_1, \;\\
E_1=E'_1+ e^{-i \sigh r} F_1, \;\ E'_1\in\rho_S^\infty\Psi_{0,\semi},\ \ F_1
 \in  \rho_S^\infty \Psi_{0,\semi}^{-\infty,\infty,-\frac{n}{2}-1,\infty} \\
 \text{ and  with } \beta_{\semi}^* K_{E_1}\text{ supported away from} \ \mcl,\ \mcr,
 \end{gathered}\label{term1n}
  \end{gather}
and $K_{E'_1}$, resp.\ $K_{F_1}$ supported near, resp.\ away from, $\Diag_{\semi}$.
In other words, the the error term  $E_1$ is such that the kernel of
$E'_1$ vanishes to infinite order at all boundary faces (hence from now on we can
regard it as trivial and ignore it), while
the kernel of $F_1$ lifts to a $\CI$ function which is supported near $\mcs$ (and in particular vanishes to infinite order at the right and left faces),  vanishes to infinite order at the semiclassical front face, and also  vanishes to order $\frac{n}{2}$ at the boundary face $\mca.$   

We will use the facts discussed above about the normal operator at the semiclassical face, $N_{\mcs}(Q_{\semi}).$
Notice that $\mcs,$ the semiclassical front face, is itself a $\CI$ manifold with boundary which intersects the zero front face, $\mcf,$ transversally, and therefore it can be extended across $\mcf.$ Similarly,  the operator $N_{\mcs}(Q_{\semi})$ can be extended to an elliptic operator across $\mcf.$
We deduce from   \eqref{norm-op-semic-face} that for each $\rho$ and $\omega'$ fixed, and for $\im\sigma<0$, the inverse of
$N_{\mcs}(Q_{\semi})$ is essentially the resolvent of the Euclidean Laplacian at energy $\sigma^2,$  pulled back by the linear change of variables corresponding to
$H_{\eps}(\rho,\omega');$ for $\im\sigma\geq 0$ we use the analytic continuation
of the resolvent from $\im\sigma<0$.
Here is where we need to make a choice corresponding to
the analytic continuation of the resolvent of $P(h,\sigma)$
we wish to construct, i.e.\ whether
we proceed from $\im\sigma>0$ or $\im\sigma<0$; we need to make the
corresponding choice for the Euclidean resolvent.

Let $R_0$ denote the analytic continuation of the inverse $L^2\to H^2$
of the {\em family} (depending on $\rho,\omega'$)
$N_{\mcs}(Q_{\semi})$ from $\im\sigma<0$; it is thus (essentially, up to a linear change
of coordinates, depending smoothly on $\rho,\omega'$)
the analytic continuation of the resolvent of the Euclidean
Laplacian.
Since we are working with the analytic continuation of the resolvent, it is not
automatic that one can solve away exponentially growing errors which arise
in the construction below (i.e.\ that one can apply $R_0$ iteratively to errors that arise),
and thus it is convenient to make the following construction quite explicit order in $h$ we are merely in the
`limiting absorption principle' regime (i.e.\ with real spectral parameter),
thus the construction below is actually stronger than what is needed below.
Moreover, from this point of view the construction can be interpreted as
an extension of the semiclassical version of the intersecting
Lagrangian construction of \cite{Melrose-Uhlmann:Intersection}
extended to the 0-double space; from this perspective the method we present
is very `down to earth'.

Via the use of a partition of unity, we may assume that there is a coordinate patch
$U$ in $\Bn$ (on which the coordinates are denoted by $z$)
such that $E_0$ is supported in
$\beta_{\semi}^{-1}(U\times U\times [0,1))$. Note that coordinate charts of this form
cover a neighborhood of $\mcs$,
so in particular $E_0$ is in
$\dCI(\Bn\times_0\Bn\times[0,1))$ outside these charts, hence can already be regarded
as part of the final error term and we can ignore these parts henceforth.

Now, near $\mcs$,
$\beta_{\semi}^{-1}(U\times U\times [0,1))$ has a product
structure $\overline{\zT}U\times [0,\delta_0)= U\times\Bn\times[0,\delta_0)$,
where $\overline{\zT}U$
denotes the fiber-compactified zero tangent bundle, and $[0,\delta_0)$ corresponds
to the boundary defining function $\rho_S$. Indeed, the normal bundle of
$\Diag_0$ in $\Bn\times_0\Bn$ can be identified with $\zT\Bn$, via lifting 0-vector
fields from on $\Bn\times_0\Bn$ via the left projection, which are transversal
to $\Diag_0$, hence
the interior of the inward pointing spherical normal bundle of $\Diag_0\times\{0\}$
in $M_0=\Bn\times_0\Bn\times[0,1)$ can be identified with $\zT\Bn$,
while the inward pointing spherical normal bundle itself can be identified
with the radial compactification of $\zT\Bn$.
However, it is fruitful to choose the identification in a particularly convenient form
locally. Namely, away from
$\pa\Diag_0\times\{0\}$, coordinates $z$ on $U$ give coordinates
$$
z',Z_{\semi}=(z-z')/h,\rho_S
$$
near the interior of the front face $\mcs$
(here $Z_{\semi}$ is the coordinate
on the fiber of $T_U\Bn$ over $z'$), while near $\pa\Diag_0\times\{0\}$,
$(x',\omega',X_{\semi},Y_{\semi},\rho_S)$ (see \eqref{eq:X-semi-def}) are coordinates
near the interior of $\mcs$ (now $(X_{\semi},Y_{\semi})$ are the coordinates
on the fiber of $\zT_U\Bn$ over $(x',\omega')$). To obtain coordinates valid
near the corner, one simply
needs to radially compactify the fibers of $\zT_U\Bn$, i.e.\ replace the linear
coordinates $Z_{\semi}$, resp.\ $(X_{\semi},Y_{\semi})$ by radially compactified versions
such as $|Z_{\semi}|$ and $\hat Z_\semi=Z_{\semi}/|Z_{\semi}|\in \bbS^n$ in the former case.

Moreover, if a function, such as $E_0$,
is supported away from $\mca$, then its support is compact in the {\em interior}
of the fibers $\Bn$ of the fiber-compactified tangent space.
Now, the interior of $\Bn$ is a vector space, $T_p U$, $p\in U$, and in particular
one can talk about fiberwise polynomials.
Over compact subsets of the fibers, the boundary defining function $\rho_S$ is
equivalent to $h$, and indeed we may choose boundary defining functions $\rho_A$
and $\rho_S$ such that
$$
h=\rho_A\rho_S.
$$
Note that $\rho_A$ is thus a boundary defining function of the compactified fibers
of the tangent bundle; it is convenient to make a canonical choice using
the metric $g_\ep$, which is an inner product on $\zT_p U$, hence a translation
invariant metric on the fibers of $\zT U$, namely to make the defining function
$\rho_A$ the reciprocal of the distance function from the zero section (i.e.\ the
diagonal under the identification), smoothed out at the zero section.
In particular, if $U$ is a coordinate chart near $\pa \Bn$ then
$\rho_{\mca}=\left((X_1)^2+ |Y_1|_{H_\eps}^2\right)^{-\frac{1}{2}}.$
This is indeed consistent with our previous calculations since
$$
\rho_{\mca}=h\left( (X-1)^2+|Y|^2\right)^{-\ha}
$$
and therefore it is, away from $\Diag_{\semi},$  a defining function of the semiclassical face $\mca.$

If $v\in\CI(M_{0,\semi})$, then expanding $v$ in Taylor series around $\mcs$ up to
order $N$, we have
$$ 
v=\sum_{k\leq N} \rho_S^k v'_k+v',\ v'_k\in\CI(\overline{T}U),
\ v'\in\rho_S^{N+1}\CI(M_{0,\semi}).
$$ 
In terms of the local coordinates valid near the corner $\mcs\cap\mca$ over an
interior coordinate chart $U$,
\begin{gather*}
\rho_S=|z-z'|,\ \frac{z-z'}{|z-z'|},\ \rho_A=\frac{h}{|z-z'|},\ z',
\end{gather*}
$v_k$ is a $\CI$ function of $|z-z'|,\ \frac{z-z'}{|z-z'|},\ z'$. It is convenient to
rewrite this as
\begin{equation}\label{eq:mod-TS-at-mcs}
v=\sum_{k\leq N} h^k |z-z'|^{-k} v'_k+v'=\sum_{k\leq N} h^k v_k+v',
\ v_k\in \rho_A^{-k}\CI(\overline{T}U),
\ v'\in\rho_S^{N+1}\CI(M_{0,\semi}),
\end{equation}
for the reason that the vector fields $D_{z_j}$ are tangent to the fibers given by
constant $h$,
i.e.\ commute with multiplication by $h$. One can rewrite $v$ completely
analogously,
$$
v=\sum_{k\leq N} h^k v_k+v',
\ v_k\in \rho_A^{-k}\CI(\overline{T}U),
\ v'\in\rho_S^{N+1}\CI(M_{0,\semi}),
$$
for coordinate charts $U$ at $\pa\Bn$.

In addition, if
$a\in\CI(\Bn\times_0\Bn\times[0,1))$, then expanding $a$ in Taylor series
around $\Diag_0\times\{0\}$, shows that for any $N$,
modulo $\rho_S^{N+1}\CI(U\times\Bn\times[0,\delta_0))$, it is of the form
\begin{gather*}
\sum_{|\alpha|+k\leq N}
a_{\alpha,k}(z') (z-z')^\alpha h^k
=\sum_{|\alpha|+k\leq N} a_{\alpha,k}(z') Z_\semi^\alpha h^{k+|\alpha|}\\
=\sum_{|\alpha|+k\leq N} a_{\alpha,k}(z') \hat Z_\semi^\alpha \rho_S^{k+|\alpha|}\rho_A^k,
\end{gather*}
where $\hat Z_\semi^\alpha=Z_\semi^{\alpha}/|Z_\semi|^{|\alpha|}$ (except near the zero section) is
$\CI$ on $\Bn$. While the last expression is the most geometric way of encoding
the asymptotics at $\pa\Bn$, it is helpful to take advantage of the stronger statement
on the previous line, which shows that the coefficients are polynomials in the fibers,
of degree $\leq N$.

The vector fields $hD_z$, resp.\ $hD_X$ and $hD_Y$, acting on a modified
Taylor series as in \eqref{eq:mod-TS-at-mcs}, become
$D_{Z_{\semi}}$, resp.\ $D_{X_{\semi}}$ and $D_{Y_{\semi}}$, i.e.\ act on the coefficients
$v_k$ only (and on $v'$, of course) so we obtain that,
modulo coefficients in
$\rho_S^{N+1}\CI(U\times\Bn\times[0,\delta_0))$, $P(h,\sigma)$ lifts to a differential
operator with polynomial coefficients on the fibers, depending smoothly on the
base variables, i.e.\ an operator of the form
$$
\sum_{|\alpha|+k\leq N,|\beta|\leq 2} a_{\alpha,k,\beta}(z') Z_\semi^\alpha h^{k+|\alpha|} D_{Z_{\semi}}^\beta,
$$ 
with an analogous expression in the $(X_{\semi},Y_{\semi})$ variables. Here the
leading term in $h$, corresponding to $h^0$, is $N_{\mcs}(Q_{\semi})=
\Delta_{g_{\ep}(x',\omega')}-\sigma^2$, i.e.\ we have
\begin{gather*}
P(h,\sigma)h^mv_m\\
=h^m(\Delta_{g_{\ep}(x',\omega')}-\sigma^2)v_m+
h^m\sum_{0<|\alpha|+k\leq N,|\beta|\leq 2} a_{\alpha,k,\beta}(z') Z_\semi^\alpha h^{k+|\alpha|} D_{Z_{\semi}}^\beta v_m.
\end{gather*}

Thus, one can iteratively solve away $E_0$ as follows. Write
$$
E_0=h^{-n-1}\sum h^mE_{0,m}
$$
as in \eqref{eq:mod-TS-at-mcs}, and note that each $E_{0,m}$ is compactly supported.
Let
$$
G_{1,0,m}=R_0 E_{0,m},
$$
so
\begin{gather*}
P(h,\sigma)h^{m-n-1} G_{1,0,m}\\
=h^mE_{0,m}+h^{m-n-1}\sum_{0<|\alpha|+k\leq N,|\beta|\leq 2} a_{\alpha,k,\beta}(z') Z_\semi^\alpha h^{k+|\alpha|-n-1} D_{Z_{\semi}}^\beta R_0 E_{0,m},
\end{gather*}
and thus we have replaced the error $h^mE_{0,m}$ by an error of the form
\begin{gather*}
\sum_{0<|\alpha|+k\leq N,|\beta|\leq 2} a_{\alpha,k,\beta}(z') Z_\semi^\alpha h^{m+k+|\alpha|-(n+1)} D_{Z_{\semi}}^\beta R_0 E_{0,m}\\
=\sum_{1\leq\ell\leq N} h^{m+\ell-(n+1)}
\sum_{|\alpha|\leq \ell}\sum_{|\beta|\leq 2} a_{\alpha,\ell-|\alpha|,\beta}(z') Z_\semi^\alpha D_{Z_{\semi}}^\beta R_0 E_{0,m}\\
=\sum_{1\leq\ell\leq N} h^{m+\ell-(n+1)} L_\ell R_0 E_{0,m}
\end{gather*}
which has the feature that not only does it vanish to (at least)
one order higher in $h$, but the $h^{m+\ell-(n+1)}$ term is given by
a differential operator $L_\ell$ with polynomial coefficients of degree $\leq \ell$
applied to $R_0E_{0,m}$,
with $E_{0,m}$ compactly supported. As the next lemma states,
one can apply $R_0$ to an expression
of the form $L_\ell R_0 E_{0,m}$, and thus iterate the construction.

\begin{lemma}
Suppose $M_j$, $j=1,2,\ldots,N$,
are differential operators with polynomial coefficients of degree $m_j$
on $\RR_w^{n+1}$. Then
$$ 
R_0 M_1 R_0 M_2\ldots R_0 M_N R_0:\CI_c(\RR^{n+1})\to\CmI(\RR^{n+1})
$$ 
has an analytic extension from $\im\sigma<0$ to $\re\sigma>1$,
$\im\sigma\in\bbR$, and
$$
R_0 M_1 R_0 M_2\ldots R_0 M_N R_0:\CI_c(\RR^{n+1})\to e^{-i\sigma\langle w\rangle}
\langle w\rangle^{-n/2+N+m}\CI(\Bn),
$$
with $m=\sum m_j$.
\end{lemma}

\begin{proof}
For $\im\sigma<0$, $D_{w_k}$ commutes with $\Delta-\sigma^2$, hence
with $R_0$, while commuting $D_{w_k}$ through a polynomial gives rise to a
polynomial of lower order, so we can move all derivatives to the right, and also assume
that $M_j=w^{\alpha^{(j)}}$, $\alpha^{(j)}\in\Nat^{n+1}$,
$|\alpha^{(j)}|\leq m_j$, so we are reduced to
examining the operator
$$
R_0 w^{\alpha^{(1)}} R_0 w^{\alpha^{(2)}}\ldots R_0 w^{\alpha^{(N)}} R_0.
$$
It is convenient to work in the Fourier transform representation.
Denoting the dual variable of $w$ by $\zeta$;
for $\im\sigma<0$,
$R_0$ is multiplication by $(|\zeta|^2-\sigma^2)^{-1}$, while $w^{\alpha^{(j)}}$ is the
operator $(-D_\zeta)^{\alpha^{(j)}}$. Rewriting
$\cF (R_0 w^{\alpha^{(1)}} R_0 w^{\alpha^{(2)}}\ldots R_0 w^{\alpha^{(N)}} R_0 f)$, the product rule
thus gives an expression of the form
\begin{gather*}
\sum_{|\beta|\leq \alpha^{(1)}+\ldots+\alpha^{(N)}}
(|\zeta|^2-\sigma^2)^{-(N+1+|\alpha^{(1)}|+\ldots+|\alpha^{(N)}|)} Q_{\alpha,N,\beta}(\zeta)
(-D_\zeta)^\beta
\Fr  f\\
=\sum_{|\beta|\leq \alpha^{(1)}+\ldots+\alpha^{(N)}}
(|\zeta|^2-\sigma^2)^{-(N+1+|\alpha^{(1)}|+\ldots+|\alpha^{(N)}|)} \Fr Q_{\alpha,N,\beta}(D_{w})
w^\beta  f,
\end{gather*}
where $Q_{\alpha,N,\beta}$ is a polynomial in $\zeta$. Since we are considering
compactly supported $f$, the differential operator $Q_{\alpha,N,\beta}(D_w)
w^\beta$ is harmless, and we only need to consider $R_0^{N+1+m}$ applied
to compactly supported functions. This can be further rewritten as
a constant multiple of $\pa_\sigma^{N+m} R_0$, so the well-known results
for the analytic continuation of the Euclidean resolvent yield the stated
analytic continuation and the form of the result; see
Proposition 1.1 of \cite{Melrose: Geometric Scattering}.
\end{proof}

Applying the lemma iteratively, we construct
$$
\tilde G_{1,m}=e^{-i\sigma\langle Z_\semi\rangle}
\langle Z_\semi\rangle^{-n/2+N+m}G'_{1,m},\ G'_{1,m}\in\CI(U\times\Bn),
$$
such that
$$
P(h,\sigma) \sum_{m=0}^\infty
h^{m-(n+1)} \tilde G_{1,m}\sim\sum_{m=0}^\infty h^{m-(n+1)} E_{0,m},
$$
where the series are understood as formal series (i.e.\ this is a statement of the
equality of coefficients). Borel summing $\langle z\rangle^{m}h^m G'_{1,m}
=\rho_S^m G'_{1,m}$, and obtaining $G_1\in\CI(U\times\Bn\times[0,\delta_0))$
as the result, which we may arrange to be supported where $\rho_S$ is small,
we deduce that \eqref{term1n} holds.

The next step is to remove the error at the semiclassical face $\mca.$  We want to construct
$$
G_2=e^{-i\sigh r} U_2,\ U_2\in\rho_S^\infty\Psi_{0,\semi}^{-\infty,\frac{n}{2},
-\frac{n}{2}-1,\frac{n}{2}},
$$
such that
\begin{gather}
\begin{gathered}
P(h,\sigma) G_2-e^{-i \sigh r} F_1=E_2, \;\ E_2=e^{-i\sigh r} F_2, \;\\
\ F_2 \in h^\infty\rho_L^{n/2}\rho_R^{n/2}\CI(\Bn\times_0\Bn\times[0,1))=\rho_{\mcs}^\infty \Psi_{0,\semi}^{-\infty,\frac{n}{2}, \infty, \frac{n}{2}}.
\end{gathered}
\label{term2}
\end{gather}
In other words, we want the error to vanish to infinite order at the semiclassical face $\mca$ and at the semiclassical front face $\mcs,$ and to vanish to order 
$\frac{n}{2}$ at the left and right faces; the infinite order vanishing
at $\mcs$ means that $\mcs$ can be blown-down (i.e.\ does not need to be
blown up), which, together with $h$ being the joint defining function
of $\mca$ and $\mcs$, explains the equality of the indicated two spaces. 
This construction is almost identical to the one carried out in section \ref{3D-parametrix}.
We begin by observing that the semiclassical face $\mca$ consists of the stretched product $\Bn\times_0 \Bn$ with $\Diag_0$ blown-up, which is exactly the manifold $\Bn\times_1 \Bn$ defined in section \ref{3D-parametrix}.  Moreover, as $F_1$ vanishes to infinite order at $\mcs$, see
\eqref{term1n}, the latter
can be blown down, i.e.\ $F_1$ can be regarded as being of the form
\begin{gather}
  \begin{gathered}
F_1 \in  h^{-n/2-1}\CI(\Bn\times_0\Bn\times[0,1)),
\  \text{ with } K_{F_1}\text{ supported away from} \ \mcl,\ \mcr,\\
 \text{ and vanishing to infinite order at}\ \Diag_0\times[0,1).
 \end{gathered}
  \end{gather}

Now, $F_1$ has an asymptotic expansion at the boundary face $h=0$ of the form
\begin{gather*}
F_1\sim h^{-\frac{n}{2}-1} \sum_{j=0}^\infty h^j F_{1,j}, \;\  F_{1,j} \in \CI(\Bn\times_0\Bn), \\ F_{1,j} \text{  vanishing to infinite order at  } \Diag_0\times[0,1),
\ \text{ supported near }  \Diag_0\times[0,1).
\end{gather*}

So we think of $F_1$ as an element of $\Bn\times_1 \Bn \times [0,1),$ where the blow-up $\Bn\times_1 \Bn,$ was defined in section \ref{3D-parametrix},
see figure \ref{fig2}, with an expansion
\begin{gather*}
F_1\sim h^{-\frac{n}{2}-1} \sum_{j=0}^\infty h^j F_{1,j},\\
F_{1,j}\in \CI(\Bn\times_1 \Bn), \text{ vanishing to infinite order at }  D.
\end{gather*}
So we seek $U_2 \sim h^{\frac{n}{2}} \sum_{j}  h^j U_{2,j}$  with $U_{2,j}$ vanishing
to infinite order at $D,$ such that
\begin{gather*}
P(h,\sigma) e^{-i\sigh r} U_2-e^{-i \sigh r} F_1=e^{-i \sigh r} R, \\ R \in h^\infty \CI(\Bn\times_1 \Bn), \text{ vanishing to infinite order at } D.
\end{gather*}
Matching the coefficients of the expansions we get the following set of transport equations
\begin{gather*}
2i\sigma |g|^{-\oq} \p_r (|g|^{\oq} U_{2,0})=-F_{1,0}, \text{ if } r>0, \\
U_{2,0}=0 \text{ at } r=0,
\end{gather*}
and for $j\geq 1,$
\begin{gather*}
2i\sigma |g|^{-\oq} \p_r (|g|^{\oq} U_{2,j})= (\Delta+x^2 W-\frac{n^2}{4})U_{2,j-1}- F_{1,j}, \text{ if } r>0, \\
U_{2,j}=0 \text{ at } r=0.
\end{gather*}
Notice that $F_{j,0}$ is compactly supported and, as seen in equations \eqref{eq:U_0-U_1-asymp},   $U_{2,0}\in \rho_L^{\frac{n}{2}}  \rho_R^{\frac{n}{2}} \CI(\Bn\times_1 \Bn).$  Moreover, as in \eqref{eq:U_0-error-asymp}, one gets that
$(\Delta+x^2 W-\frac{n^2}{4})U_{2,0}\in R^2 \rho_{R}^{\frac{n}{2}} \rho_L^{\frac{n}{2}+2} \CI(\Bn\times_1 \Bn),$
 thus one can solve the transport equation for
$U_{2,1},$ and gets that  $U_{2,1} \in \rho_L^{\frac{n}{2}}  \rho_R^{\frac{n}{2}} \CI(\Bn\times_1 \Bn).$ 

One obtains by using induction that $U_{2,j}\in \rho_L^{\frac{n}{2}}  \rho_R^{\frac{n}{2}} \CI(\Bn\times_1 \Bn),$ and
$(\Delta+x^2 W-\frac{n^2}{4})U_{2,j}\in R^2 \rho_{L}^{\frac{n}{2}+2} \rho_R^{\frac{n}{2}} \CI(\Bn\times_1 \Bn),$
for all $j.$ Then one sums the series asymptotically using Borel's lemma. This gives $U_2$ and proves \eqref{term2}.

The last step in the parametrix construction is to remove the error at the zero front face. So far we have 
\begin{gather*}
P(h,\sigma) \left(G_0-G_1+G_2\right)-\Id=E_2, \;\ 
\end{gather*}
with
$$
E_2=e^{-i\sigh r} F_2, \;\  F_2 \in \rho_{\mcs}^\infty \Psi_{0,\semi}^{-\infty,\frac{n}{2}+2, \infty, \frac{n}{2}}.
$$
Now we want to construct $G_3$ such that
\begin{gather}
P(h,\sigma) G_3-E_2=E_3 \in  e^{-i\sigh r}\rho_{\mcs}^\infty\rho_{\mcf}^\infty \Psi_{0,\semi}^{-\infty,2+\frac{n}{2}, \infty, \frac{n}{2}}(\Bn). \label{term3}
\end{gather}

We recall from Proposition \ref{distance}, that $r=-\log(\rho_R \rho_L)+F,$ $F>0.$ So,
$$
\exp(-i\sigh r)=(\rho_R \rho_L)^{i\sigh r} \exp(-i \sigh F).
$$
Therefore
the error term $E_2$ in \eqref{term2} satisfies
\begin{gather*}
\widetilde{E_2}=\beta_{\semi}^* E_2 \in \rho_{\mcs}^\infty \rho_{\mca}^\infty \rho_{R}^{\frac{n}{2}+i \sigh} \rho_{L}^{2+\frac{n}{2}+i \sigh} \CI(M_{0,\semi}).
\end{gather*}
 We write 
 \begin{gather*}
 \widetilde{E_2}\sim \sum_{j=0}^\infty \rho_{\mcf}^j E_{2,j},  \;\ E_{2,j} \in \rho_\mcs^\infty \rho_\mca^\infty \rho_{R}^{\frac{n}{2}+i \sigh} \rho_{L}^{2+\frac{n}{2}+i \sigh} \CI(\mcf),
 \end{gather*}
 and we want to construct $G_3$ such that 
 \begin{gather*}
 \beta_{\semi}^* (G_3)\sim \sum_{j=0}^\infty \rho_{\mcf}^j G_{3,j},  
 \end{gather*}
 and that
 \begin{gather*}
 N_{\mcf}(Q_{\semi}) G_{3,j}= E_{2,j}, \;\ 
 G_{3,j} \in  \rho_\mcs^\infty \rho_\mca^\infty \rho_{R}^{\frac{n}{2}+i \sigh} \rho_{L}^{\frac{n}{2}+i \sigh} \CI(\mcf).
 \end{gather*} 
The asymptotic behavior in $\rho_R$ and $\rho_L$  follows from an application of Proposition 6.15 of \cite{Mazzeo-Melrose:Meromorphic}, and the fact that
$N_\mcf(Q_{\semi})$ can be identified with the Laplacian on the hyperbolic space.  Now we just have to make sure, this does not destroy the asymptotics at the faces
$\mcs$ and $\mca.$   But one can follow exactly the same construction we have used above, now restricted to the zero front face instead of $M_{0,\semi}$ to construct
$G_{3,j}$ vanishing to infinite order at the faces $\mcs$ and $\mca.$

This gives a parametrix,  $\widetilde{G}=G_0-G_1+G_2-G_3 \in \Psi_{0,\semi}^{-2, \novt+i\sigh,-\novt-1,\novt +i\sigh}$ that satisfies
\begin{gather}
P(h,\sigma) \widetilde{G}-\Id=R=E_3+E_3' \in \rho_{\mcf}^\infty \rho_{\mcs}^\infty \Psi_{0,\semi}^{-\infty, 2+\novt+i\sigh,\infty,\novt+i\sigh}(\Bn). \label{weakparametrix}
\end{gather}
Since
\begin{gather*}
E_3'\in h^\infty\Psi_{0,\semi}^{-\infty},\\
E_3=e^{-i\sigh r} F_3, \;\  F_3 \in \rho_{\mcs}^\infty \rho_{\mcf}^\infty
\Psi_{0,\semi}^{-\infty,2+\frac{n}{2}, \infty, \frac{n}{2}},
\end{gather*}
and $E_3$ is supported away from $\Diag_\semi$.  The last step in the construction is to remove the error term at the left face and it will be done using the indicial operator as in section 7 of  
\cite{Mazzeo-Melrose:Meromorphic}.   Since in the region near the left face is away from the semiclassical face, this is in fact the same construction as in \cite{Mazzeo-Melrose:Meromorphic}, but with the parameter $h.$   Using equation \eqref{formulaofp} and the projective coordinates \eqref{projectivecoord}, we find that the operator $P(h,\sigma)$ lifts to
\begin{gather*}
P_0(h,\sigma)=h^2( -(X\p_X)^2+nX\p_X+\rho A X^2\p_X + X^2 \Delta_{H_\eps} +\rho^2 X^2 W -\frac{n^2}{4})-\sigma^2.
\end{gather*}
In these coordinates the left face is given by $\{X=0\}.$  Therefore, the kernel of the composition $K(P(h,\sigma) \widetilde{G})$ when lifted to $\Bn\times_0 \Bn$ is, near the left face, equal to 
\begin{gather*}
K(P(h,\sigma) \widetilde G)=\left( h^2( -(X\p_X)^2+nX\p_X -\frac{n^2}{4})-\sigma^2\right) K(\widetilde{G})  + O(X^2).
\end{gather*}
The operator $I(P(h,\sigma))= h^2( -(X\p_X)^2+nX\p_X -\frac{n^2}{4})-\sigma^2$ is called the indicial operator of $P(h,\sigma).$  
Since  $\widetilde{G} \in \Psi_{0,\semi}^{-2, \novt+i\sigh,-\novt-1,\novt +i\sigh}(\Bn),$ then  near the left face 
$K(\widetilde{G})\in \mck^{\novt+i\sigh,-\novt-1,\novt +i\sigh}(M_{0,\semi}).$
But $I(P(h,\sigma)) X^{\novt+i\sigh}=0.$ So we deduce that  near $\mcl,$ 
$K(P(h,\sigma) \widetilde G)\in \mck^{\novt+i\sigh+1,-\novt-1,\novt +i\sigh}(M_{0,\semi}).$ That is, we gain one order of vanishing at the left face. Since we already know from
\eqref{weakparametrix} that the kernel of the error vanishes to infinite order $\mca$ and at the front face $\mcf,$ and $x=R \rho_R,$ $x'=R\rho_L,$
 the kernel of the error $R,$ on the manifold $\Bn\times \Bn$  satisfies
\begin{gather*}
K(R)\in h^\infty x^{\novt+i\sigh+1} {x'}^{\novt+i\sigh} \CI(\Bn\times \Bn).
\end{gather*}
Then one can use a power series argument to find $G_4$ with Schwartz kernel in
$h^\infty x^{\novt+i\sigh+1}{x'}^{\novt+i\sigh} \CI(\Bn \times \Bn)$ such that

\begin{gather*}
P(h,\sigma) G_4-R \in  \Psi^{-\infty, \infty,\infty,\novt+i\sigh}(\Bn).
\end{gather*}
So $G=G_0-G_1+G_2-G_3-G_4 \in \Psi_{0,\semi}^{-2, \novt+i\sigh,-\novt-1,\novt +i\sigh}$ is the desired parametrix.

\section{$L^2$-bounds for the semiclassical resolvent}
\label{sec:L2-bounds}

We now prove bounds for the semiclassical resolvent,
$R(h,\sigma)=P(h,\sigma)^{-1}$:

\begin{thm}\label{semiclassical-resolvent-bounds}  Let $M>0,$ $h>0$ and $\sigma\in \mc$ be such that $\frac{\im \sigma}{h}<M.$  Let $a,b\geq\max\{ 0, \frac{\im\sigma}{h} \}.$ Then there exists $h_0>0$ and $C>0$ independent of $h$ such that for $h\in (0,h_0),$ 
\begin{gather}
||x^a R(h,\sigma) x^b f||_{L^2(\Bn)} \leq C h^{-1-\frac{n}{2}}||f||_{L^2(\Bn)}. \label{eq:newbounds-res}
\end{gather}
\end{thm}
As usual, we prove Theorem \ref{semiclassical-resolvent-bounds} by obtaining bounds for the 
parametrix $G$ and its error $E$ on
weighted $L^2$ spaces. As a preliminary remark, we recall that elements of
$\Psi^0_{\semi}$ on compact manifolds without boundary are $L^2$-bounded with
an $h$-independent bound;
the same holds for elements of $\Psi^0_{0,\semi}$. Thus, the diagonal singularity
can always be ignored (though in our setting, due to negative orders of the operators
we are interested in, Schur's lemma gives this directly in any case).

The following lemma follows from the argument of Mazzeo \cite[Proof of
Theorem~3.25]{Mazzeo:Edge}, since for the $L^2$ bounds, the proof
in that paper only utilizes estimates on the Schwartz kernel, rather than its
derivatives. Alternatively, it can be proved using Schur's lemma if one
writes the Schwartz kernel relative to a b-density.

\begin{lemma}\label{schurs}
Suppose that the Schwartz kernel of $B$ (trivialized by $|dg_\delta(z')|$) satisfies 
\begin{gather*}
|B(z,z')| \leq C \rho_L^{\alpha}\rho_R^{\beta},
\end{gather*}
then we have four situations:
\begin{gather*}
\text{ If }  \alpha,\beta>n/2, \text{ then }  \|B\|_{\cL(L^2)}\leq C' C. \\
\text{  If } \alpha=n/2, \;\ \beta>n/2, \text{ then } \||\log x|^{-N} B\|_{\cL(L^2)}\leq C' C, \text{ for } N>\ha. \\
\text{ If } \alpha>n/2, \;\ \beta=n/2, \text{  then }  \|| B|\log x|^{-N}\|_{\cL(L^2)}\leq C' C, \text{ for } N>\ha. \\
\text{ If } \alpha=\beta=n/2, \text{ then} \||\log x|^{-N} B |\log x|^{-N}\|_{\cL(L^2)}\leq C' C, \;\  N>\ha.
\end{gather*}
\end{lemma}

Now let $B(h,\sigma)$ have Schwartz kernel $B(z,z',\sigma,h)$ 
supported in $r>1$, and suppose that  
$$ 
B(z,z',\sigma,h)=e^{-i\sigma r/h} h^k \rho_L^{n/2+\gamma}
\rho_R^{n/2}x^a (x')^{b} H,
\ H \in L^\infty, \text{ and } \frac{\im \sigma}{h}<N.
$$ 

 Since from Proposition \ref{distance}, $r=-\log \rho_R \rho_L +F,$ $F\geq 0,$ $e^{-i \sigh r}=\rho_R^{i\sigh} \rho_L^{i\sigh} e^{-i\sigh F},$ and $\frac{\im \sigma}{h}<N,$
$|e^{-i\sigh F}|<C=C(N),$ it follows that
 
$$
|B(z,z',\sigma,h)|\leq C h^k \rho_L^{n/2+\gamma-\im\sigma/h+a}
\rho_R^{n/2-\im\sigma/h+b}R^{a+b}
$$
As an immediate consequence of Lemma \ref{schurs}, if $a+b\geq 0$, $\delta_0>0$ and
$\gamma-\im\sigma/h+a>\delta_0$ and
$-\im\sigma/h+b>\delta_0$, then
$$
\|B\|_{L^2}\leq C'C h^k.
$$
If either $\gamma-\im\sigma/h+a=0$ or $-\im\sigma/h+b=0$, we have to add  the weight $|\log x|^{-N},$ with $N>\ha.$
On the other hand, suppose now that the Schwartz kernel of $B$ is supported
in $r<2$, and (again, trivialized by $|dg_\delta(z')|$) 
$$
|B(z,z',\sigma,h)|\leq Ch^k\langle r/h\rangle^{-\ell}\langle h/r\rangle^s,\ s<n+1.
$$
Note that for fixed $z',\sigma,h$, $B$ is $L^1$ in $z$, and similarly with
$z'$ and $z$ interchanged. In fact,
since the volume form is bounded by $\tilde C r^n\,dr(z')\,d\omega'$ in $r<2$,
uniformly in $z$, Schur's lemma yields
$$
\|B\|_{L^2}\leq CC'' h^k\int_0^2 \langle h/r\rangle^s\langle r/h\rangle^{-\ell} r^{n}\,dr.
$$
But
\begin{gather*}
\int_0^2 \langle h/r\rangle^s\langle r/h\rangle^{-\ell} r^{n}\,dr\\
\leq C_0\left(\int_0^h (h/r)^s r^{n}\,dr
+\int_h^2 (r/h)^{-\ell} r^{n}\,dr\right)=C_1(h^{n+1}+h^{-\ell}),
\end{gather*}
so we deduce that
$$
\|B\|_{L^2}\leq CC' (h^{n+1+k}+h^{k-\ell}).
$$

First, if, with $n+1=3$, $G$ is the parametrix, with error $E$, constructed in
 Section~\ref{3D-parametrix}, then, writing $G=G_1+G_2$, and provided $\frac{\im\sigma}{h}<1$ and  $|\sigma|>1$ (recall that $U_1$ has a factor $\sigma^{-1}$),
with
$G_1$ supported in $r<2$, $G_2$ supported in $r>1$, then
$|G_1|\leq C h^{-2}r^{-1}$,
$|G_2|\leq C e^{\im\sigma r/h}\rho_L^{n/2}\rho_R^{n/2}$,
and thus, using Proposition \ref{distance},
\begin{gather*} 
|x^a (x')^b G_1(z,z',\sigma,h)|\leq C h^{-n-1}\langle h/r\rangle^{n-1}
\langle r/h\rangle^{-n/2},\\
|x^a (x')^b G_2(z,z',\sigma,h)|\leq C h^{-n-1}x^a (x')^b\rho_L^{n/2-\im\sigma/h}
\rho_R^{n/2-\im\sigma/h}.
\end{gather*} 
On the other hand,
\begin{gather}
|x^a (x')^b E(z,z',\sigma,h)|\leq C h x^a (x')^b\rho_L^{n/2+2-\im\sigma/h}\label{weakerror}
\rho_R^{n/2-\im\sigma/h}.
\end{gather}
Thus, we deduce the following bounds:

\begin{prop}\label{l2boundrest-3D}
Suppose $n+1=3$.
Let $G(h,\sigma)$  be the operator whose kernel is given by \eqref{parametrix},
and let $E(h,\sigma)=P(h,\sigma)G(h,\sigma)-\Id.$
Then for $|\sigma|>1,$ $\frac{\im\sigma}{h}<1,$  $a> \frac{\im \sigma}{h},$ $a\geq 0,$ 
and 
$\frac{\im\sigma}{h}< b<2- \frac{\im \sigma}{h},$ $b\geq 0$,
we have, with $C$ independent of $h$,
\begin{gather}
\begin{gathered}
||x^{a} G(h,\sigma) x^b f||_{L^2(\Bn)} \leq C h^{-1-\frac{n}{2}}  ||f||_{L^2(\Bn)} \text{ and } \\
||x^{-b} E(h,\sigma) x^b f||_{L^2(\Bn)} \leq C h  ||f||_{L^2(\Bn)}. 
\end{gathered}\label{l2bound1-3D}
\end{gather}
If either $a= \frac{\im \sigma}{h}$  or $b= \frac{\im \sigma}{h},$  or $a=b= \frac{\im \sigma}{h},$  one has to replace the factor $x^{\pm\frac{\im\sigma}{h} }$ 
in \eqref{l2bound1-3D} with $\left(x^{\frac{\im\sigma}{h}} |\log x|^{-N}\right)^{\pm 1},$ $N>\ha,$ to  obtain the $L^2$ bounds. 
\end{prop}
In view of \eqref{weakerror} this cannot be improved using the methods of section \ref{3D-parametrix}.  To obtain bounds on any strip we need the sharper bounds on the error term given by 
Theorem \ref{nD-parametrix}. We now turn to arbitrary $n$ and use the  parametrix $G$ and error $E$,
constructed in Theorem \ref{nD-parametrix}.  Writing $G=G_0+G_1+G_2$, with $G_0\in
\Psi_{0,\semi}^{-2}$, $G_1$ supported in $r<2$, $G_2$ supported in $r>1$,
$G_1\in e^{-i\sigma r/h}\Psi_{0,\semi}^{-\infty,\infty,\frac{n}{2}-1,\infty}$,
$G_2\in e^{-i\sigma r/h}\Psi_{0,\semi}^{-\infty,\frac{n}{2}+2,-\frac{n}{2}-1,\frac{n}{2}}$, 
then for $|\sigma|>1$ and $\frac{\im\sigma}{h}<N,$ 
\begin{gather*} 
x^aG_0(z,z',\sigma,h) (x')^b\in \Psi_{0,\semi}^{-2},\\
|x^a (x')^b G_1(z,z',\sigma,h)|\leq C h^{-n-1}\langle r/h\rangle^{-n/2},\\
|x^a (x')^b G_2(z,z',\sigma,h)|\leq C h^{-n-1}x^a (x')^b\rho_L^{n/2-\im\sigma/h}
\rho_R^{n/2+\im\sigma/h}.
\end{gather*} 
On the other hand, writing
$E=E_1+E_2$, with $E_1$ supported in $r<2$, $E_2$ supported in $r>1$,
$E_1\in h^\infty\rho_{\mcf}^\infty\Psi_{0,\semi}^{-\infty,\infty,\frac{n}{2}-1,\infty}$,
$E_2\in h^\infty \rho_{\mcf}^\infty\Psi_{0,\semi}^{-\infty,\infty,-\frac{n}{2}-1,\frac{n}{2}}$,
then for any $k$ and $M$ (with $C=C(M,N)$),
\begin{gather*}
|x^{-b} (x')^b E_1(z,z',\sigma,h)|\leq C h^k,\\
|x^{-b} (x')^b E_2(z,z',\sigma,h)|\leq C h^k x^M (x')^b
\rho_R^{n/2-\im\sigma/h}.
\end{gather*}
In this case, we deduce the following bounds:

\begin{prop}\label{l2boundrest-gen}
Let $G(h,\sigma)$  be the operator whose kernel is given by \eqref{kernelofg},
and let $E(h,\sigma)=P(h,\sigma)G(h,\sigma)-\Id$.
Then for $|\sigma|>1,$ $\frac{\im\sigma}{h}<M,$ $M>0,$ $ a> \frac{\im \sigma}{h},$ $a\geq 0,$ and 
$b>\frac{\im\sigma}{h},$ $b\geq 0,$ and $N$ arbitrary, we have, with $C$ independent of $h$, 
\begin{gather}
\begin{gathered}
||x^{a} G(h,\sigma) x^b f||_{L^2(\Bn)} \leq C h^{-1-\frac{n}{2}}  ||f||_{L^2(\Bn)} \text{ and } \\
||x^{-b} E(h,\sigma) x^b f||_{L^2(\Bn)} \leq C h^N  ||f||_{L^2(\Bn)}. 
\end{gathered}\label{l2bound1-gen}
\end{gather}
If either $a= \frac{\im \sigma}{h}$  or $b= \frac{\im \sigma}{h},$  or $a=b= \frac{\im \sigma}{h},$  one has to replace the factor $x^{\pm\frac{\im\sigma}{h} }$ 
in \eqref{l2bound1-gen} with $\left(x^{\frac{\im\sigma}{h}} |\log x|^{-k}\right)^{\pm 1},$ $k>\ha,$ to  obtain the $L^2$ bounds.
\end{prop}

Now we can apply these estimates to prove Theorem \ref{semiclassical-resolvent-bounds}.  We know that
\begin{gather*}
P(h,\sigma) G(h,\sigma)=I+ E(h,\sigma).
\end{gather*}
Since $R(h,\sigma)$ is bounded on $L^2(\Bn)$ for $\im\sigma<0$ we can write for $\im\sigma<0,$
\begin{gather*}
G(h,\sigma)= R(h,\sigma)( I+E(h,\sigma) ).
\end{gather*}
Therefore we have, still for $\im \sigma<0,$
\begin{gather*}
x^a G(h,\sigma) x^b= x^a R(h,\sigma)x^b (I + x^{-b} E(h,\sigma) x^{-b}).
\end{gather*}
For $a,b$ and $\sigma$ as in Proposition \ref{l2boundrest-gen} we can pick $h_0$ so that
$$
||x^{-b} E(h,\sigma) x^b f||_{L^2\rightarrow L^2} \leq \ha.
$$
In this case we have
\begin{gather*}
x^a G(h,\sigma) x^b (I + x^{-b} E(h,\sigma) x^{-b})^{-1}= x^a R(h,\sigma)x^b
\end{gather*}
and the result is proved.

\section{Proof of Theorem~\ref{resolvent-bounds}}
\label{sec:resolvent-bounds}
Now we are ready to prove Theorem \ref{resolvent-bounds}.   To avoid using the same notation for different parameters, we will denote the spectral parameter in the statement of Theorem \ref{resolvent-bounds} by $\la,$ instead of $\sigma.$  
 We  write, for $|\re\la|>1,$
\begin{equation*}
(\Delta_{g_\delta}+ x^2 W-\la^2-\frac{n^2}{4})= (\re \la)^2\left[ \frac{1}{(\re\la)^2}\left(\Delta_{g_\delta}+ x^2 W-\frac{n^2}{4}\right)-\frac{\la^2}{(\re \la)^2}\right].
\end{equation*}
Thus, if we denote $h=\frac{1}{\re \la}$ and $\sigma=\frac{\la}{\re \la},$ and $E(h,\sigma)$ and $G(h,\sigma)$ are the operators in Proposition \ref{l2boundrest-gen}, 
 we have
\begin{gather*}
(\Delta_{g_\delta}+ x^2 W-\la^2-\frac{n^2}{4})G(\frac{1}{\re \la}, \frac{\la}{\re \la})=
 (\re\la)^2\left(\Id + E(\frac{1}{\re\la},\frac{\la}{\re \la})\right).
\end{gather*}
Since $R_\del(\la)=\left( \Delta_{g_\delta}+ x^2 W-\la^2-\frac{n^2}{4}\right)^{-1}$ is a well defined bounded operator if $\im \la<0,$ 
we can write, 
\begin{gather*}
G(\frac{1}{\re \la},\frac{\la}{\re \la})= (\re\la)^2 R_\delta(\la) \left(\Id +  E(\frac{1}{\re\la},\frac{\la}{\re\la})\right), \text{ for }  \im\la<0.
\end{gather*}
Therefore,
\begin{gather*}
x^aG(\frac{1}{\re \la},\frac{\la}{\re \la})x^b= (\re\la)^2 x^a R_\delta(\la) x^b\left(\Id +  x^{-b} E(\frac{1}{\re\la},\frac{\la}{\re\la}) x^{b}\right).
\end{gather*}
According to Proposition \ref{l2boundrest-gen}, if $\im\la<M,$ we can pick  $K$ such that if  $|\re \la|>K,$ and $\im \la < b$ then, 
 \begin{equation*}
 ||x^{-b} E(\frac{1}{\re \la},\frac{\la}{\re \la} )x^b||<\ha.
 \end{equation*}
 Therefore
$(\Id +  x^{-b} E(\frac{1}{\re \la},\frac{\la}{\re \la}) x^{b})^{-1}$ is holomorphic in $\im\la<M,$ and 
bounded as an operator in $L^2(\Bc, g),$ with norm independent of $\re\la,$ provided $b>\im \la.$
 On the other hand,  if $a>\im \la,$ then  from Proposition \ref{l2boundrest-gen},
\begin{gather*}
||x^{a} G(\frac{1}{\re \la},\frac{\la}{\re \la}) x^b f||_{L^2(\Bn)} \leq C (\re \la)^{1+\frac{n}{2}}  ||f||_{L^2(\Bn)}. 
\end{gather*}
Since,
\begin{gather*}
x^a R_\delta(\la) x^b =(\re \la)^{-2}x^a G(\frac{1}{\re \la},\frac{\la}{\re \la}) x^b\left( I+ x^{-b} E(\frac{1}{\re \la},\frac{\la}{\re \la}) x^{b}\right)^{-1},
\end{gather*}
then, for and for $a,b$ in this range, and $|\re \la|>K,$ $x^a R_\del(\la) x^b$ is holomorphic and 
\begin{gather*}
||x^a R_\delta(\la) x^b f||_{L^2(\Bc,g)} \leq C (\re\la)^{\frac{n}{2}-1} ||f||_{L^2(\Bc,g)}.
\end{gather*}

When either $a=\im \la$ or $b=\im \la$ we have to introduce the logarithmic weight and in Proposition \ref{l2boundrest-gen}. This
concludes the proof of the 
$L^2$ estimates of Theorem \ref{resolvent-bounds}.

The Sobolev estimates follow from these $L^2$ estimates and interpolation.
First we observe that the following commutator properties hold:  There are $\CI(\Bc)$ functions $A_i$ and $B_j,$ $i=1, 2,$ and $1\leq j \leq 5,$ such that
\begin{gather*}
[\Delta_{g_\delta},x^a]= A_1 x^a + A_2 x^a xD_x, \\
[\Delta_{g_\delta}, x^a (\log x)^{-N}]=  B_1 x^a (\log x)^{-N} + B_2 x^a (\log x)^{-N-1} + \\
B_3 x^a (\log x)^{-N-2} + \left( B_4 x^a (\log x)^{-N} + B_5 x^a (\log x)^{-N-1}\right) x D_x.
\end{gather*}
Hence
\begin{gather*}
\Delta_{g_\delta} x^a R_\delta(\sigma) x^b v= x^a \Delta_{g_\delta} R_\delta(\sigma)  x^b v+
A_1 x^a R_\delta(\sigma) x^b  v + A_2 x^a x D_x R_\delta(\sigma) x^b  v.
\end{gather*}
Since $a,b\geq 0,$ 
$\Delta_{g_\delta}$ is elliptic,  and $\Delta_{g_\delta} R_\delta(\sigma)=\Id + (\sigma^2+1-x^2 W(x) ) R_\delta$
it follows that, see for example \cite{Mazzeo:Edge},
that there exists a constant $C>0$ such that
\begin{equation}\begin{split}
&||x^a R_\delta(\sigma) x^b v||_{ H^2_0(\Bc)} \\
&\qquad\leq C \left( \sigma^2 ||x^b v||_{L^2(\Bc)} +
||x^a R_\delta(\sigma) x^b v||_{L^2(\Bc)} + ||x^a R_\delta(\sigma) x^b v||_{H^1_0(\Bc)} \right).\label{h1bound-res}
\end{split}\end{equation}
By interpolation between Sobolev spaces we know that there exists $C>0$ such that
\begin{gather}
||x^av||_{H^1_0(\Bc)}^2 \leq C ||x^a v||_{L^2(\Bc)}  ||x^a v||_{H^2_0(\Bc)}\label{interp2}
\end{gather}
Therefore, for any $\eps>0,$ 
\begin{gather}
|| x^a v ||_{H^1_0(\Bc)} \leq C\left( \eps ||x^a v ||_{H^2_0(\Bc)} + \eps^{-1} ||x^a v ||_{L^2(\Bc)}\right), \label{interp1}
\end{gather}
and if one takes $\eps$ small enough, \eqref{h1bound-res} and \eqref{interp1} give
\eqref{sobolev1} with $k=2.$   If one uses \eqref{interp2} and \eqref{sobolev1} with $k=2$ one obtains
\eqref{sobolev1} with $k=1.$  The proof of \eqref{sobolev2} follows by the same argument.

This completes the proof of Theorem~\ref{resolvent-bounds}.

\section{Structure of $\Delta_X$ near the boundaries}
\label{sec:black-hole}

We begin the proof of Theorem \ref{globalest} by analyzing the structure of
the operator $\Delta_X$ near $r=r_{\bH}$ and $r=r_{\sI}.$  We recall that
$\beta(r)=\ha \frac{d}{dr} \alpha^2(r)$ and that $\beta_\bH=\beta(r_\bH)$
and $\beta_\sI=\beta(r_\sI).$

We show that near
these ends, after rescaling $\alpha,$ the operator $\alpha^{\novt} \Delta_X
\alpha^{-\novt}$ is a small perturbation of the Laplacian of the hyperbolic
metric of constant negative sectional curvature $-\beta_{\bH}^2$ near $r_{\bH}$ and
$-\beta_{\sI}^2$ near $r_{\sI}.$

Since $\alpha'(r)\not=0$ near $r=r_{\bH}$ and $r=r_{\sI},$ $\Delta_X$ can be
written in terms of $\alpha$ as a `radial' coordinate
\begin{equation*}
\Delta_X=
\beta r^{-n}\alpha D_\alpha( \beta r^n\alpha D_\alpha)+\alpha^2 r^{-2}
\Delta_\omega.
\end{equation*}
We define a $\CI$ function $x$ on $[r_{\bH}, r_{\sI}]$ which is positive in the interior of the interval and rescales  $\alpha$ near the ends by
\begin{gather}
\alpha= 2r_{\bH}\beta_{\bH} x \text{ near } r=r_{\bH}\text{ and  }
\alpha= 2r_{\sI}|\beta_{\sI}| x \text{ near }r=r_{\sI}.
\label{rescale}
\end{gather}
Using $x$ instead of $\alpha$ near the ends $r_\bH$ and $r_\sI,$ we obtain
\begin{gather}
\begin{gathered}
\Delta_X=
\beta r^{-n} x D_{x}( \beta r^n x D_{x})+4{x}^2\beta_{\bH}^2r_{\bH}^2 r^{-2}
\Delta_\omega, \text{ near } r=r_{\bH}, \\
\Delta_X=\beta r^{-n}xD_{x}(\beta r^n xD_{x})+
4x^2\beta_{\sI}^2r_{\sI}^2 r^{-2} \Delta_\omega, \text{ near } r=r_{\sI}.
\end{gathered}\label{conjugation}
\end{gather}

\begin{prop}\label{modelnearends}    There exists $\del>0$ such that, if we  identify  each of the neighborhoods of  $\{x=0\}$   given by $r\in [r_{\bH},r_{\bH}+\del)$ and $r\in (r_{\sI}-\del, r_{\sI}],$
with a neighborhood of the boundary of the ball 
$\Bn,$ then there exist two $\CI$ functions,
$W_{\bH}(x)$ defined near $r=r_{\bH},$ and 
$W_{\sI}(x)$ defined near $r= r_{\sI},$  such that 
\begin{gather}
\begin{gathered}
\alpha^\novt \Delta_X \alpha^{-\novt}=x^\novt \Delta_X x^{-\novt} = \Delta_{g_{\bH}} + x^2 W_{\bH} -\beta_{\bH}^2\frac{n^2}{4}, 
\text{ near } r=r_{\bH} \text{ and } \\
\alpha^\novt\Delta_X \alpha^{-\novt}=x ^\novt\Delta_X {x}^{-\novt} = \Delta_{g_{\sI}} + x^2 W_{\sI} -\beta_{\sI}^2\frac{n^2}{4},  
\text{ near }  r=r_{\sI},
\end{gathered}\label{modelnearends:eq}
\end{gather}
where  $g_{\bH}$ and $g_{\sI}$ are small perturbations of the hyperbolic metrics with sectional curvature 
$-\beta_{\bH}^2$ and $-\beta_{\sI}^2$
respectively on  the interior of $\Bn,$ i.e.
\begin{gather}
\begin{gathered}
g_{\bH}=  \frac{ 4 dz^2}{ \beta_{\bH}^2(1-|z|^2)^2}+H_{\bH}, \text{ and }
g_{\sI}=  \frac{ 4 dz^2}{ \beta_{\sI}^2(1-|z|^2)^2}+H_{\sI},
\end{gathered}\label{metform}
\end{gather}
where $H_{\bH}$ and $H_{\sI}$ are symmetric 2-tensors $\CI$ up to the
boundary of $\Bn.$
\end{prop}
\begin{proof}    It is only necessary to prove the result near one of the ends.  The computation near the other end is identical and one only needs to replace
the index $\bH$ by $\sI.$ From \eqref{conjugation} we we find that near $r=r_{\bH},$  
\begin{equation*}\begin{split}
 x^{\novt} \Delta_X {x}^{-\novt}=  &\beta^2(x D_{x})^2 +in \beta^2 {x} D_{x}
 +\beta r^{-n} (x D_{x}(\beta r^n) )x D_{x} +(2 \beta_{\bH} r_{\bH})^2  r^{-2} x^2 \Delta_\omega \\ & -i\novt \beta r^{-n} x D_{x}(\beta r^n)-\beta^2\nsq.
\end{split}\end{equation*}

 Let $g_\bH$  be  the metric defined on a neighborhood of $\p \Bn$ given by
\begin{gather}
\begin{gathered}
g_{\bH}=  \frac{dx^2}{\beta^2 x^2} +  \la_{\bH}^{-2} {r^2}\frac{d\omega^2}{x^2}, \text{ where } \lambda_\bH= 2|\beta_\bH| r_\bH.
\end{gathered} \label{refmet}
\end{gather}
The Laplacian of this metric is
\begin{gather*}
\Delta_{g_\bH}= \beta^2(x D_x)^2 +in \beta^2 x D_x + \beta r^{-n} (x D_x(\beta r^n))x D_x + 
 \lambda_\bH^2 r^{-2} x^2\Delta_\omega.
  \end{gather*}
Therefore we conclude that near the ends $r=r_{\bH}$ 
\begin{gather*}
x^\novt \Delta_X x^{-\novt} = \Delta_{g_\bH} -\beta^2\nsq -i \novt\beta r^{-n} x D_x(\beta r^2).
\end{gather*}

Since $r=r(x^2),$ we  can write near $r_{\bH}$ and $r_{\sI},$
\begin{gather*}
r=r_\bH + x^2 A_\bH(x^2) \text{ and }
\beta(r)= \beta_\bH + x^2 B_\bH(x^2)\text{ near } r=r_\bH.
\end{gather*}
Therefore, near $r=r_{\bH}$ 
\begin{gather*}
\frac{1}{\beta^2}= \frac{1}{\beta_\bH^2} + x^2 \widetilde{B}_\bH(x^2).
\end{gather*}

We conclude that there exist a symmetric 2-tensor $H_\bH(x^2,d x,d\omega)$ near 
$r=r_\bH$  which is $\CI$ up to
$\{x=0\},$ and such that the metric $g_\bH$ given by \eqref{refmet} can be written 
near $r=r_\bH$  as
\begin{gather}
\begin{gathered}
g_\bH= \frac{d x^2}{\beta_\bH^2 x^2} + \frac{ d\omega^2}{4\beta_\bH^2 x^2} + H_\bH \text{ near } r=r_\bH.
\end{gathered} \label{modelend}
\end{gather}

Let  $\tilde{g}$ be the metric on the interior of $\Bn$ which is given by
\begin{gather*}
\tilde{g}= \frac{4 |dz|^2}{ c^2 (1-|z|^2)^2}.
\end{gather*}
We consider local coordinates valid for $|z|>0$ given by $(x,\omega),$ where $\omega=z/|z|,$ and
$x=\frac{1-|z|}{1+|z|}.$
The metric $\tilde{g}$ written in terms of these coordinates is given by
\begin{gather*}
\tilde{g}= \frac{ dx^2}{c^2 x^2} +  (1-x^2)^2 \frac{d\omega^2}{4c^2 x^2}.
\end{gather*}
Therefore,  near $x=0$
\begin{gather*}
\tilde{g}= \frac{ dx^2}{c^2 x^2} +  \frac{ d\omega^2}{4 c^2 x^2} + H(x^2,\omega,dx,d\omega),
\end{gather*}
where $H$ is a symmetric 2-tensor smooth up to the boundary of $\Bn.$
This concludes the proof of the Proposition.
\end{proof}

\section{From cut-off and models to stationary resolvent}
\label{sec:decomposition}

Next we use the method of Bruneau and Petkov \cite{Bruneau-Petkov:Semiclassical} to decompose the
operator $R(\sigma)$ in terms of its cut-off part $\chi R(\sigma) \chi$ and the contributions from the ends, which are controlled by Theorem~\ref{resolvent-bounds}.
For that one needs to define some suitable cut-off functions. For $\delta>0$  let $\chi_j,$  $\chi_j^1,$ and $\tilde{\chi}_j,$ $j=1,2,$  defined by
\begin{gather*}
\chi_1(r)= 1 \text{ if } r>r_{\bH}+4\delta, \;\  \chi_1(\alpha)=0 \text{ if } r< r_{\bH} + 3\delta, \\
\chi_1^1(r)= 1 \text{ if } r>r_{\bH}+2\delta, \;\  \chi_1^1(\alpha)=0 \text{ if } r< r_{\bH} + \delta, \\
\tilde{\chi}_1(r)= 1 \text{ if } r>r_{\bH}+6\delta, \;\  \tilde{\chi_1}(\alpha)=0 \text{ if } r< r_{\bH} +5 \delta, \\
\chi_2(r)= 1 \text{ if } r<r_{\sI}-4\delta, \;\  \chi_2(\alpha)=0 \text{ if } r> r_{\sI} - 3\delta, \\
\chi_2^1(r)= 1 \text{ if } r<r_{\sI}-2\delta, \;\  \chi_2^1(\alpha)=0 \text{ if } r> r_{\sI} - \delta, \\
\tilde{\chi}_2(r)= 1 \text{ if } r<r_{\sI}-6\delta, \;\  \tilde{\chi}_2(\alpha)=0 \text{ if } r> r_{\sI} - 5\delta,
\end{gather*}
and let
\begin{gather*}
\chi_3(r)= 1-(1-\chi_1)(1-\chi_1^1)-(1-\chi_2)(1-\chi_2^1).
\end{gather*}
$\chi_3(r)$ is supported in $[r_{\bH}+\delta, r_{\sI}-\delta]$ and $\chi_3(r)=1$ 
if $r\in[r_{\bH}+2\delta, r_{\sI}-2\delta].$ Let $\chi \in C_0^\infty(r_{\bH},r_{\sI})$ with
$\chi(r)=1$ if $r\in [r_{\bH}+\delta/2, r_{\sI}-\delta/2].$

Now we will use Proposition \ref{modelnearends} for $\delta$ small enough. Let $g_{\bH}$ and $g_{\sI}$
be the metrics given  on the interior of $\Bn$ given by \eqref{metform} and let
\begin{gather}
g_{\bH,\del}=  \frac{ 4 dz^2}{ \beta_{\bH}^2 (1-|z|^2)^2}+ (1-\tilde{\chi}_1) H_{\bH}, \text{ and }
g_{\sI,\del}=  \frac{ 4 dz^2}{ \beta_{\sI}^2 (1-|z|^2)^2}+(1-\tilde{\chi}_2)H_{\sI}. \label{g+andg++}
\end{gather}
Since $g_{\bH,\del}=g_{\bH}$ if $\tilde{\chi}_1=0,$ and  $g_{\sI,\del}=g_{\sI}$ if $\tilde{\chi}_2=0,$ it follows from Proposition \ref{modelnearends} that, for $x$ given by equation \eqref{rescale}
\begin{gather*}
\alpha^\novt \Delta_X \alpha^{-\novt} (1-{\chi}_1) f= (\Delta_{g_{\bH,\del}} +x^2 W_{\bH}-\nsq \beta_{\bH}^2-\sigma^2) (1-\chi_1) f, \text{ and } \\
\alpha^\novt \Delta_X \alpha^{-\novt} (1-\chi_1^1) f= (\Delta_{g_{\bH,\del}} +x^2 W_{\bH}-\nsq \beta_{\bH}^2-\sigma^2) (1-\chi_1^1) f. \\
\alpha^\novt \Delta_X \alpha^{-\novt} (1-\chi_2) f= (\Delta_{g_{\sI,\del}} +x^2 W_{\sI}-\nsq\beta_{\sI}^2-\sigma^2) (1-\chi_2) f, \text{ and } \\
\alpha^\novt \Delta_X \alpha^{-\novt} (1-\chi_2^1) f= (\Delta_{g_{\sI,\del}} +x^2 W_{\sI}-\nsq\beta_{\sI}^2-\sigma^2) (1-\chi_2^1) f.
\end{gather*}
Let
 \begin{gather*}
 R_\alpha(\sigma)= \alpha^\novt R(\sigma) \alpha^{-\novt}=(\alpha^\novt \Delta_X \alpha^{-\novt}-\sigma^2)^{-1}, \;\
 \im \sigma<0,
 \end{gather*}
and let 
\begin{gather*}
R_{\bH}(\sigma)=(\Delta_{g_{\bH,\del}} +x^2 W_{\bH} -\sigma^2-\nsq \beta_{\bH}^2)^{-1} \text{ and } \\
R_{\sI}(\sigma)=(\Delta_{g_{\sI,\del}} +x^2 W_{\sI} -\sigma^2-\nsq \beta_{\sI}^2)^{-1}
\end{gather*}
be operators acting on functions defined on $\Bn.$

If,  as in Proposition \ref{modelnearends}, we identify neighborhoods of $r=r_{\bH}$ and $r=r_{\sI}$ with a neighborhood of the boundary of $\Bn,$ we obtain the following identity for the resolvent
\begin{gather}
\begin{gathered}
R_\alpha(\sigma)= R_\alpha(\sigma)\chi_3+ (1-\chi_1)R_{\bH}(\sigma)(1-\chi_1^1) +
(1-\chi_2) R_{\sI}(\sigma)(1-\chi_2^1) - \\
 R_\alpha(\sigma)[\Delta_{g_{\bH}}, 1-\chi_1] R_{\bH}(\sigma)(1-\chi_1^1) -
 R_\alpha(\sigma)[\Delta_{g_{\sI}}, 1-\chi_2] R_{\sI}(\sigma)(1-\chi_2^1).
\end{gathered}\label{residentity1}
\end{gather}
Similarly, one obtains
\begin{gather}
\begin{gathered}
R_\alpha(\sigma)= \chi_3 R_\alpha(\sigma)+ (1-\chi_1^1)R_{\bH}(\sigma)(1-\chi_1) +
(1-\chi_2^1) R_{\sI}(\sigma)(1-\chi_2) + \\
 (1-\chi_1^1)R_{\bH}(\sigma)[\Delta_{g_{\bH}}, 1-\chi_1]  R_\alpha(\sigma)+
(1-\chi_2^1)R_{\sI}(\sigma)[\Delta_{g_{\sI}},1- \chi_2] R_\alpha(\sigma).
\end{gathered}\label{residentity2}
\end{gather}
These equations can be verified by applying $\alpha^\novt \Delta_X \alpha^{-\novt}-\sigma^2$ on the left and on the right of both sides of the identities.
Substituting \eqref{residentity2} into \eqref{residentity1} we obtain
\begin{equation}
\begin{split}
R_\alpha(\sigma)&= M_1(\sigma) \chi R_\alpha(\sigma) \chi M_2(\sigma) +\\
&\qquad\qquad
 (1-\chi_1)R_{\bH}(\sigma)(1-\chi_1^1)+ (1-\chi_2)R_{\sI}(\sigma)(1-\chi_2^1), \\
 &\text{ where} \\
M_1(\sigma)&= \chi_3 + (1-\chi_1^1)R_{\bH}(\sigma) (1-\tilde{\chi}_1)[\Delta_{g_{\bH}},1-\chi_1]\\
&\qquad\qquad
+ (1-\chi_2^1)R_{\sI}(\sigma)(1-\tilde{\chi}_2)[\Delta_{g_{\sI}},1-\chi_2], \\
M_2(\sigma)&= \chi_3 -[\Delta_{g_{\bH}},1-\chi_1](1-\tilde{\chi}_1)R_{\bH}(\sigma)(1-\chi_1^1)\\
&\qquad\qquad-[\Delta_{g_{\sI}},1-\chi_2](1-\tilde{\chi}_2)R_{\sI}(\sigma)(1-\chi_2^1).
\end{split}\label{brpkid}
\end{equation}

This gives a decomposition of $R_\alpha(\sigma)$ in terms of the cutoff resolvent, studied by Bony and H\"afner \cite{Bony-Haefner:Decay}, and the resolvents of he Laplacian of a metric which are small perturbations of the Poincar\'e metric in $\Bn.$ The mapping properties of such operators were established in 
Section \ref{sec:resolvent-bounds}.  Next we will put the estimates together and finish the proof of our main result.

\section{The proof of Theorem~\ref{globalest} in $3$ dimensions}
\label{sec:black-hole-proof}

We now prove Theorem \ref{globalest} for $n+1=3$ using
Theorem \ref{resolvent-bounds} and Theorem \ref{bhthm}. We first restate a strengthened version of the theorem which includes the case where the weight
$b=\im\sigma.$

\begin{thm}\label{globalest-strong} Let $\eps>0$
be such that \eqref{bonyhafnerest}
holds and suppose
$$
0<\gamma<\min(\eps,\beta_{\bH},|\beta_{\sI}|,1).
$$
Then for $b>\gamma$
there exist $C$ and $M$ such that if $\im\sigma\le\gamma$ and $|\re\sigma
|\ge1,$ 
\begin{equation}
||{\tilde{\alpha}}^b R(\sigma){\tilde{\alpha}}^bf||
_{L^2 ( X ;\Omega)}\leq C|\sigma|^M ||f||_{L^2 ( X ;\Omega) },
\label{mainest1-r}
\end{equation}
where $\tilde{\alpha}$ was defined in \eqref{SeClRe.1} and the measure $\Omega$ was defined in \eqref{measure}.
Moreover, for $N>\ha$ and $0<\del<<1,$ choose $\psi_N(r)\in\CI(r_{\bH}, r_{\sI})$ with
$\psi_N(r)\ge1,$ such that
\begin{equation}
\psi_N=|\log \alpha|^{-N}\Mif r-r_{\bH}<\del\Mor r_{\sI}-r<\del.
\label{SeClRe.4}\end{equation}
Then with $\gamma$ as above, there exists $C>0$ and $M\geq 0$ such that for
$|\re\sigma|\ge1$ and  $|\im\sigma|\le\gamma$ 
\begin{gather}
\begin{gathered}
||{\tilde{\alpha}}^{\im \sigma} \psi_N(\alpha) R(\sigma)
{\tilde{\alpha}}^{b} f||_{L^2 ( X ;\Omega )}
\leq C |\sigma|^M ||f||_{L^2 ( X ;\Omega)},
\\ 
||{\tilde{\alpha}}^{b} R(\sigma)
{\tilde{\alpha}}^{\im \sigma}\psi_N(\alpha) f||_{L^2 ( X ;\Omega)}
\leq C |\sigma|^M ||f||_{L^2 ( X ;\Omega) }\Mand
\\
||{\tilde{\alpha}}^{\im \sigma} \psi_N(\alpha) R(\sigma)
{\tilde{\alpha}}^{\im\sigma}\psi_N(\alpha) f||_{L^2 (X;\Omega)}
\leq C |\sigma|^M ||f||_{L^2 ( X ;\Omega)}.
\end{gathered}\label{mainest11}\end{gather}
\end{thm}

\begin{proof}  Recall  from \eqref{g+andg++} that
\begin{gather*}
g_{\bH,\delta}= \frac{1}{\beta_\bH^2} g_\delta \text{ and } g_{\sI,\delta}= \frac{1}{\beta_\sI^2} g_\delta,
\end{gather*}
where $g_\delta$ is of the form \eqref{metgeps}.
So we obtain $\Delta_{g_{\bH,\delta}}=\beta_\bH^2 \Delta_{g_\delta}$ and similarly
$\Delta_{g_{\sI,\delta}}=\beta_\sI^2 \Delta_{g_\delta}.$  Therefore,
\begin{gather}
\begin{gathered}
R_\bH(\sigma)=\left( \beta_\bH^2\Delta_{g_\delta}+ x^2 W_\bH-\sigma^2-\nsq \beta_\bH^2\right)^{-1}= \\ \beta_\bH^{-2}\left(\Delta_{g_\delta}+ x^2 \beta_\bH^{-2} W_\bH-\sigma^2\beta_\bH^{-2}-\nsq\right)^{-1}= \beta_\bH^{-2} R_\delta( \sigma|\beta_\bH|^{-1}).
\end{gathered}\label{bulletid}
\end{gather}
Therefore, by replacing $\sigma$ with $\sigma |\beta_\bH|^{-1}$ in \eqref{sobolev1} and
 \eqref{sobolev2} setting $a=\frac{A}{|\beta_\bH|},$ and $b=\frac{B}{|\beta_\bH|},$
  we deduce from Theorem \ref{resolvent-bounds}  that  there exists $\delta_0>0$ such that if $0<\delta<\delta_0,$
 for  $\im \sigma<A$ and $\im\sigma < B$ and $|\re \sigma|> K(\delta),$
 \begin{gather}
  \begin{gathered}
 ||x^{\frac{A}{|\beta_\bH|}}  R_\bH(\sigma) x^{\frac{B}{|\beta_\bH|}} v||_{H^k(\Bc)} \leq
 C |\sigma|^k || v||_{L^2(\Bc)}, \;\ k=0,1,2, \\
  ||x^{\frac{A}{|\beta_\bH|}}  R_\bH\sigma) x^{\frac{B}{|\beta_\bH}} v||_{L^2(\Bc)} \leq
 C |\sigma|^k || v||_{H_0^{-k}(\Bc)}, \;\ k=0,1,2.
 \end{gathered}\label{sobolev3}
  \end{gather}
  Of course, the same argument applied to $R_{\sI}$ gives
   \begin{gather}
  \begin{gathered}
 ||x^{\frac{A}{|\beta_\sI|}}  R_\sI(\sigma) x^{\frac{B}{|\beta_\sI|}} v||_{H^k(\Bc)} \leq
 C |\sigma|^k || v||_{L^2(\Bc)}, \;\ k=0,1,2, \\
  ||x^{\frac{A}{|\beta_\sI|}}  R_\sI\sigma) x^{\frac{B}{|\beta_\sI}} v||_{L^2(\Bc)} \leq
 C |\sigma|^k || v||_{H_0^{-k}(\Bc)}, \;\ k=0,1,2.
 \end{gathered}\label{sobolev3sI}
  \end{gather}

 When $A={\im\sigma},$ or  $B={\im\sigma},$ 
  we define   $T_{\bH,A,B,N}$  and $T_{\sI,A,B,N}$  and
  as in \eqref{deftabn} by replacing $R_\delta(\sigma)$ with either
  $R_\bH(\sigma)$  or $R_\sI(\sigma),$ $a$ with $A$ and $b$ with $B.$  Using \eqref{sobolev2} we obtain
  for $J=\bH$ or $J=\sI,$
 \begin{gather}
  \begin{gathered}
 || T_{J,A,B,N} (\sigma)  v||_{H_0^k(\Bc)} \leq  C |\sigma|^k  ||v||_{L^2(\Bc)}, \;\ k=0,1,2, \\
|| T_{J,A,B,N} v||_{L^2(\Bc)} \leq  C |\sigma|^k  ||v||_{H_0^{-k}(\Bc)}, \;\ k=0,1,2.
\end{gathered}\label{sobolev4}
 \end{gather}

Now we recall that $\alpha=2 r_{\bH}\beta_{\bH}$ near $r_{\bH}$ and similarly
$\alpha=2 r_{\sI}\beta_{\sI}$ near $r_{\sI}.$ We will use these estimates, identity \eqref{brpkid} and Theorem \ref{bhthm} to prove
Theorem \ref{globalest}.  Indeed,  in the case $a>\im\sigma,$  $b>\im \sigma$ we write
\begin{gather*}
{\tilde{\alpha}}^a R_\alpha {\tilde{\alpha}}^b= {\tilde{\alpha}}^a M_1(\sigma) {\tilde{\alpha}}^b 
{\tilde{\alpha}}^{-b} \chi R_\alpha(\sigma) \chi {\tilde{\alpha}}^{-b} {\tilde{\alpha}}^b M_2(\sigma) 
{\tilde{\alpha}}^b 
+ \\  (1-\chi_1) {\tilde{\alpha}}^a R_{\bH}(\sigma){\tilde{\alpha}}^b (1-\chi_1^1) +(1-\chi_2) 
{\tilde{\alpha}}^aR_{\sI}(\sigma){\tilde{\alpha}}^b (1-\chi_2^1).
\end{gather*}
Notice that the measure in $\Bn$ is $x^{-n-1}dxd\omega,$ which corresponds to
$\alpha^{-n-1}d\alpha d\omega$ which in turn corresponds to $\alpha^{-n-2} dr d\omega.$   In this case $n=2,$ but this part of the argument is the same for all dimensions, and we will not set $n=2.$
Thus, we deduce from Theorem \ref{resolvent-bounds} that
\begin{gather*}
|| (1-\chi_1) {\tilde{\alpha}}^a R_{\bH}(\sigma){\tilde{\alpha}}^b (1-\chi_1^1) ||_{L^2(X;\alpha^{-n-2}  dr d\omega)} \leq C ||v||_{L^2(X;\alpha^{-n-2} dr d\omega)}, \\
||(1-\chi_2) {\tilde{\alpha}}^aR_{\sI}(\sigma){\tilde{\alpha}}^b (1-\chi_2^1) v||_{L^2(X;\alpha^{-n-2}  dr d\omega)} \leq C||v||_{L^2(X;\alpha^{-n-2}  dr d\omega)}.
\end{gather*}
Recall that $\Omega=\alpha^{-2} r^2 dr d\omega.$ Since $r\in [r_\bH,r_\sI],$ $r_\bH>0,$ this gives
\begin{gather}
\begin{gathered}
|| (1-\chi_1) {\tilde{\alpha}}^a \alpha^{-\novt}R_{\bH}(\sigma)\alpha^{\novt} {\tilde{\alpha}}^b (1-\chi_1^1) ||_{L^2(X;\Omega)} \leq C ||v||_{L^2(X;\Omega)}, \\
||(1-\chi_2) {\tilde{\alpha}}^a\alpha^{-\novt}R_{\sI}(\sigma)\alpha^\novt {\tilde{\alpha}}^b (1-\chi_2^1) v||_{L^2(X;\Omega)} \leq C||v||_{L^2(X;\alpha^{-n-2} \Omega)}.
\end{gathered}\label{finalestimate1}
\end{gather}
Similarly, using the Sobolev estimates in Theorem \ref{resolvent-bounds}, we obtain

\begin{gather}
\begin{gathered}
|| {\tilde{\alpha}}^a \alpha^{-\novt} M_1(\sigma) \alpha^\novt {\tilde{\alpha}}^b v||_{L^2(X;\Omega)} \leq
C |\sigma| ||v||_{L^2(X;\Omega)}, \\
|| {\tilde{\alpha}}^a \alpha^{-\novt} M_2(\sigma)\alpha^\novt {\tilde{\alpha}}^b v||_{L^2(X;\Omega)} \leq
C |\sigma| ||v||_{L^2(X;\Omega)}.
\end{gathered}\label{finalestimate2}
\end{gather}

Since $\chi$ is compactly supported in the interior of $X,$ it follows from Theorem \ref{bhthm} that
\begin{gather}
||{\tilde{\alpha}}^{-b} \chi R_\alpha(\sigma) \chi {\tilde{\alpha}}^{-b} v||_{L^2(X;\Omega)} \leq C||v||_{L^2(X;\Omega)}.\label{finalestimate3}
\end{gather}

Estimates \eqref{finalestimate1}, \eqref{finalestimate2} and \eqref{finalestimate3} imply that
\begin{gather*}
||{\tilde{\alpha}}^{-b} \alpha^{-\novt} R_\alpha(\sigma)\alpha^\novt \chi {\tilde{\alpha}}^{-b} v||_{L^2(X;\Omega)} \leq C||v||_{L^2(X;\Omega)}.
\end{gather*}
But $\alpha^{-\novt} R_\alpha(\sigma)\alpha^\novt=R(\sigma).$ This proves \eqref{mainest1-r}. 
\end{proof}

\section{The proof of Theorem~\ref{globalest} in general dimension}\label{sec:black-hole-n+1-proof}

We will outline the main steps necessary to connect to
the results
of \cite{Datchev-Vasy}, and refer the reader to \cite{Datchev-Vasy} for more details.
 
First choose $\delta$ so small  \eqref{bich-convexity1} holds.
Let $X_0$ and $X_1$ be as defined  in \eqref{Xdecomp},
then we recall that for $\delta$ small,
 \begin{gather*}
\alpha^{-\frac{n}{2}} P\alpha^{-\frac{n}{2}}|_{X_0}= P(h,\sigma),
\end{gather*}
where $P(h,\sigma)$ stands for model near either $r_{\bH}$ or near
$r_{\sI}.$ By
Theorem \ref{semiclassical-resolvent-bounds}
there exists $h_0>0$ such that for $h\in (0,h_0),$
\begin{gather*}
||x^a (h^2 P_0-\sigma)^{-1} x^b||_{L^2\rightarrow L^2} \leq C h^{-1-\frac{n}{2}} \;\ \sigma \in (1-c,1+c) \times (-c, \eps h)\subset\Cx.
\end{gather*}
 
Let $P_1$ be the operator defined in \eqref{operatorp1} and let  $\eps>0$ be such that \eqref{wuzwestimate} holds. Then it follows from Theorem 2.1 of
 \cite{Datchev-Vasy} that  there exist $h_1>0,$ $C>0$ and $K>0$ such that for $h\in (0,h_1),$
 \begin{equation}\begin{split}\label{eq:davaestimate}
 ||{\tilde{\alpha}}^b \alpha^{-\frac{n}{2}} (h^2 \Delta_X-\sigma^2)^{-1} &{\tilde{\alpha}}^a \alpha^{-\frac{n}{2}}||_{L^2\rightarrow L^2} \leq  C h^{-K},\\
 &\sigma \in (1-c,1+c) \times (-c, \eps h)\subset\Cx.
 \end{split}\end{equation}
 The estimate in Theorem \ref{globalest} follows by restating \eqref{eq:davaestimate}
 in the non-semiclassical language, i.e.\ multiplying it by
 $h^2$ and replacing $\sigma$ by $h^{-1}\sigma$.

\def\cprime{$'$} \def\cprime{$'$}

\end{document}